\documentclass[11pt,letterpaper]{amsart}
    \pdfoutput=1
\usepackage[maxbibnames=99, style=alphabetic]{biblatex}
    \usepackage{graphicx}
    \usepackage[english]{babel}
    \usepackage{graphicx}
    \usepackage{framed}
    \usepackage[normalem]{ulem}
    \usepackage{amsmath}
    \usepackage{dynkin-diagrams}
    \usepackage{amsthm}
    \usepackage{amssymb}
    \usepackage{hyperref} 
    \usepackage{comment}
    \usepackage{tikz}
    \usetikzlibrary{matrix,calc}
    \usepackage{mathtools}
    \usepackage{amsfonts}
    \usepackage{enumerate}
    \usepackage[utf8]{inputenc}
    \usepackage[top=1 in,bottom=1in, left=1 in, right=1 in]{geometry}
    \usepackage[capitalize]{cleveref}
    \usepackage[colorinlistoftodos]{todonotes}
    \usepackage[all]{xy}
    \usepackage{mathrsfs}
    \usepackage{longtable} 
    \usepackage{stmaryrd}

    \newcommand{\Z}{\mathbb{Z}} 
    \newcommand{\A}{\mathbb{A}} 
    \newcommand{\G}{\mathbb{G}} 
    
    \newcommand{\LSL}{\mathfrak{sl}}
    \newcommand{\SL}{\operatorname{SL}}

    \newcommand{\LG}{\mathfrak{g}}
    
    \newcommand{\LT}{\mathfrak{t}}
    \newcommand{\LN}{\mathfrak{n}}
    \newcommand{\LB}{\mathfrak{b}}
    
    \renewcommand{\a}{\alpha}

    \newcommand{\Hom}{\operatorname{Hom}}
    
    \renewcommand{\Hom}{\operatorname{Hom}} 
    \newcommand{\uHom}{\underline{\operatorname{Hom}}} 
    \newcommand{\uEnd}{\underline{\operatorname{End}}} 
    \newcommand{\QCoh}{\operatorname{QCoh}}
    \newcommand{\IndCoh}{\operatorname{IndCoh}}
    
    \newcommand{\C}{\mathcal{C}} 
    \renewcommand{\O}{\mathcal{O}} 
    \newcommand{\F}{\mathcal{F}} 
    \newcommand{\D}{\mathcal{D}} 
    \newcommand{\E}{\mathcal{E}} 
    \renewcommand{\\}{\backslash}
    
    \usepackage[mmddyyyy]{datetime}

    \theoremstyle{definition}
    \newtheorem{Theorem}{Theorem}[section]
    
    \newtheorem{Corollary}[Theorem]{Corollary}
    \newtheorem{Definition}[Theorem]{Definition}
    \newtheorem{Proposition}[Theorem]{Proposition}
    \newtheorem{Remark}[Theorem]{Remark}
    \newtheorem{Example}[Theorem]{Example}
    \newtheorem{Lemma}[Theorem]{Lemma}

    \title{The Universal Category $\mathcal{O}$ and the Gelfand-Graev Action}
    \author{Tom Gannon}

    \newcommand{\CatN}{\mathcal{D}(N \backslash G/N)_{\text{nondeg}}}

    \newcommand{\CatTwTw}{\mathcal{D}(N\backslash G/N)^{\text{(}T \times T\text{, w)}}_{\text{nondeg}}}

    \newcommand{\Hpsi}{\mathcal{H}_{\psi}}

    \newcommand{\AvN}{\text{Av}_*^N}
    \newcommand{\Avpsi}{\text{Av}_{!}^{\psi}}
    \newcommand{\Symt}{\text{Sym(}\mathfrak{t}\text{)}}
    
    \newcommand{\LTd}{\LT^{\ast}}

    \newcommand{\Wext}{\tilde{W}^{\text{aff}}}
    \newcommand{\Waff}{W^{\text{aff}}}

    \newcommand{\DGCatContk}{\text{DGCat}^k_{\text{cont}}}
\newcommand{\DGCatContL}{\text{DGCat}^L_{\text{cont}}}
\newcommand{\Aone}{\mathbf{A}}
\newcommand{\Atwo}{\mathcal{L}\text{-mod}(\Aone)}
\newcommand{\algobj}{\mathcal{A}}
\newcommand{\newalgobj}{\mathcal{A}'}

\newcommand{\AvNshifted}{\AvN[\text{dim}(N)]}
\newcommand{\Avpsishifted}{\Avpsi[-\text{dim}(N)]}
\newcommand{\HN}{\D(N \backslash G/N)}
\newcommand{\ClGlobalDiffOp}{\text{H}^0\Gamma(\mathcal{D}_{G/N})}
\newcommand{\GlobalDiffOp}{\Gamma(\mathcal{D}_{G/N})}

\newcommand{\HNTw}{\D(N \backslash G/N)^{T \times T, w}}
\newcommand{\indsch}{\mathcal{X}}
\newcommand{\characterlatticeforT}{X^{\bullet}(T)}
\newcommand{\rootlattice}{\mathbb{Z}\Phi}
\newcommand{\Spec}{\text{Spec}}
\newcommand{\fieldpossiblydifferentfromgroundfield}{K}

\newcommand\Tstrut{\rule{0pt}{2.8ex}}
\newcommand\smallTstrut{\rule{0pt}{2.5ex}}
\newcommand\Bstrut{\rule[-1.5ex]{0pt}{0pt}}

\newcommand{\Hn}{\mathcal{D}(N\backslash G/N)}

  \newcommand{\newGeometricCategoryO}{\widetilde{\mathcal{O}}_{\text{geom}}}
   \newcommand{\newAlgebraicCategoryO}{\widetilde{\mathcal{O}}_{\text{alg}}}
   
\begin{document}
    \maketitle
    \begin{abstract}
We show that the definition and many useful properties of Soergel's functor $\mathbb{V}$ extend to \lq universal\rq{} variants of the BGG category $\mathcal{O}$, such as the category which drops the semisimplicity condition on the Cartan action. We show that this functor naturally factors through a quotient known as the nondegenerate quotient and show that this quotient admits a \lq Gelfand-Graev\rq{} action of the Weyl group. Although these categories are not finite length in any sense, we explicitly compute the kernel of this functor in a universal case and use this to extend these results to an arbitrary DG category with an action of a reductive group. 

Along the way, we prove a result which may be of independent interest, which says that a functor of $G$-categories with a continuous adjoint is an equivalence if and only if it is an equivalence at each field-valued point, strengthening a result of Ben-Zvi--Gunningham--Orem.
    \end{abstract}
    
\tableofcontents 

    \section{Introduction}
    \subsection{Overview} The purpose of this paper is to study an analogue of Soergel's quotient functor $\mathbb{V}$ on \lq universal\rq{} variants of the BGG category $\mathcal{O}$ (see \cref{Variants of Category O Definition}) such as the category $\newAlgebraicCategoryO$ which drops the semisimplicity criterion on the Cartan action. Although these categories do not have finite length in any sense, in \cref{Kernel in Universal Case Subsubsection} we give a description for the kernel of this functor for each of these categories. This functor naturally factors through a certain quotient which we refer to as the \textit{nondegenerate quotient}, which we also prove admits an action of the Weyl group at the higher categorical level, see \cref{Nondegenerate G-Categories Have N Invariants Admit Weyl Group Action}. 
    
    We use the tools of categorical representation theory to reduce to the finite length case. In particular, we prove an extension of \cite{BZGO} which says that a functor of categories with an action of a reductive group which admits a continuous adjoint is an equivalence if and only if the induced functor is an equivalence at every field-valued point of a dual Cartan subalgebra $\LTd$, which we state more precisely in \cref{Can Check Equivalence on Each field-valued Point for Groups}. These methods give an analogue of Soergel's quotient functor for any DG category with an action of a reductive group, which is used in the companion paper \cite{GannonClassificationOfNondegenerateGCategories}.
    \subsection{Soergel Theory}It is often useful to prove claims about a category by reducing them to claims about a more tractable quotient. This technique has been particularly successful in understanding the BGG category $\mathcal{O}$ of a reductive group $G$ in terms of \textit{Soergel modules}. We now briefly survey this for the BGG category $\mathcal{O}_0$ of objects of generalized central character zero; we also provide a slightly more detailed overview near arbitrary central character in \cref{Soergel Classification of Olambda Subsection}. For a more detailed survey the reader may consult, for example, \cite[Part III]{EliasMakisumiThielWilliamsonIntroToSoergelBimodules}. 
    
    To exhibit the category $\mathcal{O}_0$ as modules for some ring, Soergel constructs an exact functor \begin{equation}\label{Definition of mathbbV in the Introduction} \mathbb{V}: \mathcal{O}_{0} \to \text{Vect}^{\heartsuit} \end{equation} on the BGG category $\mathcal{O}_{0}$ to the abelian category of vector spaces, which naturally factors through a quotient category via the quotient functor \begin{equation}
    \label{Nondegenerate Quotient in Introduction Functor}\mathcal{O}_{0} \to C_0\text{-mod}^{\heartsuit}    \end{equation} where $C_0$ denotes the coinvariant algebra for the Weyl group and $C_0\text{-mod}^{\heartsuit}$ denotes its abelian category of modules. Soergel then shows that the entirety of $\mathcal{O}_{0}$ can be recovered combinatorially from this quotient functor. 
    \subsection{Generalization to $G$-Categories} We now discuss our generalizations of the BGG category $\mathcal{O}$ and our generalization of $\mathbb{V}$. In fact, we generalize the construction of $\mathbb{V}$ to any $G$\textit{-category}, i.e. any DG category equipped with a module structure for the convolution monoidal category $\D(G)$ of $\D$-modules on $G$. 
    
    We set some notation: let $B$ denote a Borel subgroup of $G$ containing some fixed maximal torus $T \subseteq B$, and let $B^-$ denote the Borel subgroup containing $T$ in opposite position. Define $N := [B, B]$ and $N^- := [B^-,B^-]$.
    
    \subsubsection{Generalities on $G$-Categories}\label{Generalities on G Categories Subsubsection}In \cref{CategoricalPreliminaries}, we review the basic definitions and some basic results regarding $G$-categories. While some of these results may be known to experts, they are not available in the literature in the generality we need, so we prove them.
    
     For the purposes of this introduction we assume that, for a category $\C$ with an action of an algebraic group $H$, the reader is familiar with the construction of its \textit{twisted invariants}, which we define precisely in \cref{Character Sheaves and Twisted Invariants Subsubsection}, and its \textit{weak invariants} $\C^{H, w}$, i.e. the invariants as a $\QCoh(H)$-module category. We also give a more leisurely introduction to categorical representation theory in \cite[Section 1]{GannonClassificationOfNondegenerateGCategories}.
        \subsubsection{Non-Semisimple Variants of $\mathcal{O}$}\label{Non-Semisimple Variants of O Subsubsection}
    To motivate our definition for our variants of the BGG category $\mathcal{O}$, we first recall the Beilinson-Bernstein localization theorem \cite[Th\'eor\`eme Principal]{BB}, which states that the global sections functor induces a $t$-exact equivalence \begin{equation}\label{Beilinson-Bernstein Localization at Zero}       
   \D(G/B) \xrightarrow{\sim} \mathfrak{g}\text{-mod}_{\chi_0} \end{equation} of the category of $\D$-modules on the flag variety with the category of modules for $U\LG \otimes_{Z\LG}k$, and a $t$-exact equivalence \begin{equation}\label{Monodromic Beilinson-Bernstein Localization at Zero}\D(G)^{B\text{-mon}} \xrightarrow{\sim} \mathfrak{g}\text{-mod}_{\widehat{\chi_0}}\end{equation}of $B$-monodromic $\D$-modules (in the sense of \cref{Monodromic Objects Definition}) with the category of $U\LG$-modules of generalized central character zero. 
   
   Now, since the $N$-equivariant $\D$-modules on $B\backslash G$ are equivalently the $B$-monodromic $\D$-modules on $B\backslash G$, from \labelcref{Monodromic Beilinson-Bernstein Localization at Zero} one can deduce a $t$-exact equivalence of categories \begin{equation}\label{Beilinson-Bernstein In Particular Can Be Said to Identify NEquiv Dmods on G mod B with BGG O}\D(G/N)^B \simeq \D(B \backslash G)^N \simeq \D(B\backslash G)^{B\text{-mon}}\simeq \D(G)^{B\text{-mon}, B} \simeq \mathfrak{g}\text{-mod}_{\widehat{\chi_0}}^B\end{equation} which, at the abelian categorical level, identifies the category of compact $N$-equivariant $\D$-modules on the flag variety with the BGG category $\mathcal{O}_0$. 
   
   \begin{Remark}
       The equivalences of \labelcref{Beilinson-Bernstein Localization at Zero}, \labelcref{Monodromic Beilinson-Bernstein Localization at Zero}, and \labelcref{Beilinson-Bernstein In Particular Can Be Said to Identify NEquiv Dmods on G mod B with BGG O} are stated at the derived, ind-completed level, but of course the analogous equivalences also hold for the finite length abelian categories. For example, one can derive the equivalences of  \labelcref{Beilinson-Bernstein In Particular Can Be Said to Identify NEquiv Dmods on G mod B with BGG O} from the associated abelian categorical equivalence because the categories in \labelcref{Beilinson-Bernstein In Particular Can Be Said to Identify NEquiv Dmods on G mod B with BGG O}, which are a priori defined as invariance categories, are the derived categories of their hearts, which we prove in \cref{Twisted D-Modules are Derived Category of Heart Section}. 
   \end{Remark}
   
   Motivated by the above, we make the following definitions: 
 
    \begin{Definition}\label{Variants of Category O Definition}
        We define the \textit{geometric category} $\widetilde{\mathcal{O}}$ as the category \[\newGeometricCategoryO := \D(G/N)^{N, (T, w)} \simeq \D(N\backslash G/N)^{T, w}\] the \textit{algebraic category} $\widetilde{\mathcal{O}}$ as the category \[\newAlgebraicCategoryO := \D(G/N)^{G, w} \simeq \LG\text{-mod}^{N}\] and refer to the category \[\mathcal{O}_{\text{univ}} := \D(G/N)^{G \times T, w} \simeq \LG\text{-mod}^{N, (T, w)}\] as the \textit{universal category} $\mathcal{O}$.
    \end{Definition}

     The categories of \cref{Variants of Category O Definition} give natural generalizations of the BGG category $\mathcal{O}$. For example, as explained in \cite{KalSaf}, the universal category $\mathcal{O}$ can be identified with the category of $(U\LG, U\LT)$-bimodules whose diagonal $U\LB$-module structure integrates to a representation of $B$. Additionally, the Bernstein-Lunts theorem \cite{BernsteinLuntzEquivariantSheavesandFunctors}, whose higher categorical version is proved in \cite[Lemma A.18.1]{Ras2}, gives that the algebraic category $\widetilde{\mathcal{O}}$ is the derived category of the abelian category of those $U\LG$-modules whose restricted $U\LN$-module structure integrates to a representation of $N$. 
     
     The categories of \cref{Variants of Category O Definition} also give families of categories whose fiber at some $k$-point recovers the usual BGG category $\mathcal{O}_0$. For example, since $\newAlgebraicCategoryO$ admits an action of the category of $Z\LG$-mod, by \cite{GaitsgorySheavesofCategoriesandtheNotionof1Affineness} we can equivalently view $\newAlgebraicCategoryO$ as a category over $\Spec(Z\LG)$, and the fiber of this category at central character zero $\chi_0 \in \Spec(Z\LG)$ is the category $\LG\text{-mod}_{\chi_0}^{N}$. From this category, one can recover the usual category $\mathcal{O}_0$ by taking the compact objects in the heart of the standard $t$-structure. The categories $\mathcal{O}_{\text{univ}}$, respectively $\newGeometricCategoryO$, similarly admit a module structure for the category $\QCoh(\LTd) = \text{Sym}(\LT)\text{-mod}$, respectively $\D(T) \simeq \QCoh(\LTd/\characterlatticeforT)$ where $\characterlatticeforT$ is the lattice of characters for $T$.
    
    \subsubsection{Analogue of $\mathbb{V}$}\label{Analogue of V Subsubsection}Fix a generic character $\psi: N^- \to \mathbb{G}_a$. To demonstrate that \labelcref{Definition of mathbbV in the Introduction} is the special case of a general functor of $G$-categories, we first recall a well-known fact we prove for the sake of completeness below: \begin{Proposition}\label{SupportOfWhittakerSheaves}
The restriction functor gives a $t$-exact equivalence of right $T$-categories $$\D(G/N)^{N^-, \psi} \xrightarrow{\sim} \D(N^-B/N)^{N^-, \psi} \simeq \D(T).$$ 
\end{Proposition} Using this equivalence and the equivalence \labelcref{Beilinson-Bernstein In Particular Can Be Said to Identify NEquiv Dmods on G mod B with BGG O}, we can view the functor \labelcref{Definition of mathbbV in the Introduction} as a functor \begin{equation}\label{Definition of Avpsi in Introduction Without Label}\C^N \xrightarrow{} \C^{N^{-}, \psi}\end{equation} of invariants of the $G$-category $\C := \D(G/B) \simeq \D(G/N)^T$. In fact, as a consequence of \cite{BezYun}, Soergel's functor $\mathbb{V}$ is equivalently an \textit{averaging functor}
\begin{equation}\label{Definition of Avpsi in Introduction}\Avpsi: \C^N \xrightarrow{} \C^{N^{-}, \psi}\end{equation} for the $G$-category $\C := \D(G/B)$, up to cohomological shift. 

This averaging functor \labelcref{Definition of Avpsi in Introduction} is defined for any $G$-category, as we review in \cref{Averaging Functors Subsubsection}, and is a natural generalization of Soergel's functor $\mathbb{V}$. For example, in \cref{The Functor Avpsi On Olambda Subsection}, we show that this averaging functor for the twisted invariants category $\C := \D(G/_{[\lambda]}B)$ gives a twisted version of Soergel's $\mathbb{V}$. If we instead substitute $\C = \D(G/N)^{T,w}$ into \labelcref{Definition of Avpsi in Introduction}, we recover a functor on $\newGeometricCategoryO$, and substituting $\C = \D(G)^{G, w} = \LG\text{-mod}$ we similarly recover a functor on $\newAlgebraicCategoryO$. Finally, via the forgetful functor $\mathcal{O}_{\text{univ}} \to \newAlgebraicCategoryO$, one obtains a functor on $\mathcal{O}_{\text{univ}}$. 

We will also use this averaging functor for a general $G$-category to construct a quotient category analogous to the codomain of the quotient functor of \labelcref{Nondegenerate Quotient in Introduction Functor}, which we call a \textit{nondegenerate} $G$\textit{-category}, see \cref{Nondegenerate G Categories Introduction Subsubsection}. 

\begin{Remark}
In the setting of categorical representation theory, all functors such as \labelcref{Definition of Avpsi in Introduction} are a priori defined in the setting of DG categories, which are in particular stable $\infty$-categories in the sense of \cite{LuHTT}, \cite{LuHA}. However, we also show that if $\C$ admits a reasonable $t$-structure (in a sense we make precise in \cref{G Category with Compactly Generated Compatible tStructure Means Avpsi Is T Exact Up to Shift}) then this functor is $t$-exact, so in this case (which includes all of the categories in \cref{Variants of Category O Definition}) one obtains an exact functor of abelian categories by taking hearts.
\end{Remark}

\subsubsection{Highest Weights for Categorical Representations at Field-Valued Points}
Although essentially all of our results hold for general $G$-categories, we spend much of this paper studying the categories $\D(G/N)$ and $\D(G/N)^{T,w}$ because they are \lq universal cases\rq{} of $G$-categories in a sense we now make precise. Notice that tensoring with the bimodule category $\D(G/N)$ defines a functor \[G\text{-mod(DGCat)} \to \mathcal{H}_N\text{-mod(DGCat)}\] where $\mathcal{H}_N := \Hn$, and, similarly, tensoring with the bimodule category $\D(G/N)^{T,w}$ defines a functor \[G\text{-mod(DGCat)} \to  \mathcal{H}_{N, (T,w)}\text{-mod(DGCat)}\] where $ \mathcal{H}_{N, (T,w)} := \D(N\backslash G/N)^{T \times T, w}$. As we recall in \cref{Inv=Coinv}, these functors are equivalently given by the functors $\C \mapsto \C^N$ and $\C \mapsto \C^{N, (T,w)}$. The main result of \cite{BZGO} states that these bimodule categories are \lq universal cases\rq{} of $G$-categories in the sense that these functors are Morita equivalences: 
    
\begin{Theorem}\label{BZGO} \cite{BZGO} Let $\C$ be a category with a $G$-action. Then the functors $\C \mapsto \C^N$, $\C \mapsto \C^{N, (T,w)}$ are conservative. Equivalently, the canonical functors \[\D(G/N)\otimes_{\mathcal{H}_N}\C^N \to \C\] and \[\D(G/N)^{T, w}\otimes_{\mathcal{H}_{N, (T,w)}}\C^{N, (T,w)} \to \C\] are equivalences.
\end{Theorem}

In contrast, however, the functor $\C \mapsto \C^B$ is \textit{not} an equivalence of categories. This can be understood conceptually using the language of categorical representation theory: for any $G$-category $\C$, the category $\C^{N, (T, w)}$ can in particular be viewed as a module category for the monoidal category $\IndCoh(\LTd)$. Using this module structure, we have that \[\C^B \simeq \C^{N, (T, w)} \otimes_{\IndCoh(\LTd)}\mathrm{Vect},\] or, in other words, that $\C^B$ can be obtained by tensoring $\C^{N, (T, w)}$ with the $\IndCoh(\LTd)$-module category $\mathrm{Vect}$ obtained by pullback by the inclusion $\Spec(k) \to \LTd$ of the origin. 

This conceptual framework suggests a natural alternative: to determine whether a functor $F: \C \to \D$ is an equivalence, one cannot simply check that the functor $F^{N, (T, w)}$ is an equivalence of categories \lq at\rq{} the origin of $\LTd$, but one must instead check that \begin{equation}\label{Functor at Every Field Valued Point}F^{N, (T, w)} \otimes_{\IndCoh(\LTd)}L\mathrm{-mod}:    \C^{N, (T, w)} \otimes_{\IndCoh(\LTd)}L\mathrm{-mod} \to \D^{N, (T, w)} \otimes_{\IndCoh(\LTd)}L\mathrm{-mod}\end{equation} at every field-valued point $\Spec(L) \to \LTd$ of $\LTd$. In proving our main theorems, we also prove that this replacement indeed holds, at least provided that our functor admits a continuous adjoint:

\begin{Theorem}\label{Can Check Equivalence on Each field-valued Point for Groups}
 If $F$ is a functor of $G$-categories which admits a continuous adjoint, then $F$ is fully faithful (respectively, an equivalence) if and only if the induced functor \labelcref{Functor at Every Field Valued Point} is fully faithful (respectively, an equivalence) for all field-valued points $\lambda \in \LTd(L)$.
 
\end{Theorem} 

\begin{Remark}
    An earlier result due to Gaitsgory-Lurie \cite[Theorem 2.5.7]{Ber} also gives an \lq algebraic analogue\rq{} of \cref{BZGO}: it states that for any $G$-category, the functor $\C \mapsto \C^{G,w}$ is conservative. 
    
    Although we will not use this result outside of this remark, this result and \cref{BZGO} provide motivation for studying the categories of \cref{Variants of Category O Definition} apart from their connections to the BGG category $\mathcal{O}$: tensoring with them provides a \lq translation\rq{} between various equivalent incarnations of the 2-category of $G$-categories. 
\end{Remark}

    \subsection{Statement of Main Results}\label{Statement of Main Result Subsection} We now state our main results both in a universal case and in their extensions to general $G$-categories.
    \subsubsection{Kernel in Universal Case}\label{Kernel in Universal Case Subsubsection}As we recall above, one of the key features of Soergel's $\mathbb{V}$ is an explicit description of its kernel. In our setting, the kernel of the functor \labelcref{Definition of Avpsi in Introduction} applied to any invariant category, such as those of \cref{Variants of Category O Definition}, can be understood through the kernel of the \lq universal case\rq{} \begin{equation}\label{Definition of Avpsi in Introduction In Special Case of DG}
   \Avpsi: \D(G/N) \to \D(G/_{\psi}N^-) \end{equation} or, in other words, the case where $\C$ is the right $G$-category $\D(G)$. More precisely, an object $\F$ in an \lq invariance\rq{} category of $\D(G/N)$ such as those of \cref{Variants of Category O Definition} lies in the kernel of $\Avpsi$ if and only if $\text{oblv}(\F) \in \D(G/N)$ lies in the kernel of $\Avpsi$. This is due to the fact that the functor \labelcref{Definition of Avpsi in Introduction In Special Case of DG} is itself left $G$-invariant and the fact that the functor which forgets the equivariance is conservative. 
   
   However, $\D(G/N)$ is not finite length, and so a different type of analysis is required to understand the kernel of the functor \labelcref{Definition of Avpsi in Introduction In Special Case of DG}. Moreover, a choice of cohomological normalization occurs, even near central character zero: $C_0$ is essentially never a regular ring, so that, at the derived level, there are two reasonable choices for the DG variant of modules over $C_0$, namely, the category of ind-coherent sheaves or the category of quasi-coherent (or ind-perfect) sheaves. 
    
    We study the full subcategory of $\D(G/N)$ generated by the \textit{eventually coconnective} objects in the kernel, which we denote by $\D(G/N)_{\text{deg}}$ and, as we show in \cite{GannonClassificationOfNondegenerateGCategories}, corresponds to the choice of ind-coherent sheaves. One can show that the functor $\Avpsi$ of \labelcref{Definition of Avpsi in Introduction} is $t$-exact up to cohomological shift (\cref{BBMShiftedLeftAdjointIsExact}) so one can recover the entire kernel of this functor through the eventually coconnective objects in the kernel. Our first main result, which we also state as \cref{Definition of D(G/N)deg}, shows that $\D(G/N)_{\text{deg}}$ admits residual symmetries other than the $G$-action:

    \begin{Theorem}
        The subcategory $\D(G/N)_{\text{deg}}$ is closed under the action of $\mathcal{H}_N$.
    \end{Theorem}
    
    With this in mind, our main result (\cref{Simply Connected Version of Nondegenerate Subcategory of D(G/N) is the HN-subcategory Gen'd By Non-antidominant Simples}), which we state in the case where $G$ is simply connected for the ease of exposition, can be stated as follows: 
\begin{Theorem}\label{Explicit Description of Eventually Connective Kernel Objects}
If $G$ is simply connected, the category $\D(G/N)_{\text{deg}}$ is the full right $\mathcal{H}_N$-subcategory of $\D(G/N)$ generated by the right $[P, P]$-monodromic objects for every rank-one parabolic $P$.
\end{Theorem}

\subsubsection{Nondegenerate $G$-Categories}\label{Nondegenerate G Categories Introduction Subsubsection}
We continue to assume that $G$ is simply connected. By \cref{BZGO}, $\C^N \simeq 0$ only if $\C \simeq 0$. This fact, along with \cref{Explicit Description of Eventually Connective Kernel Objects}, motivates the following definition, which is a main object of study in \cite{GannonClassificationOfNondegenerateGCategories}:

\begin{Definition}\label{Nondegenerate G Category Definition for Simply Connected Group}
    We say that a $G$-category $\C$ is \textit{nondegenerate} if $\C^{[P, P]} \simeq 0$ for any rank-one parabolic $P$. 
\end{Definition}

\cref{Explicit Description of Eventually Connective Kernel Objects} can also be used to study an arbitrary $G$-category: given any $G$-category $\C$, one can construct its \textit{nondegenerate quotient} $\C_{\text{nondeg}}^N$ which is obtained by quotienting $\C^N$ by the full $\mathcal{H}_N$-subcategory of objects generated by $\C^{[P, P]\text{-mon}}$ under colimits. This quotient is a localization in the sense of \cref{Intro to t-Structures on Quotient Categories} and, as we will see below in \cref{Weyl Group Action on Nondegenerate Quotient Subsubsection}, also admits an action of the Weyl group $W := N_G(T)/T$. 

If $\C$ is moreover equipped with a reasonable $t$-structure (in a sense we make precise in \cref{G Category with Compactly Generated Compatible tStructure Means Avpsi Is T Exact Up to Shift}) then this quotient functor is $t$-exact and so an exact Serre quotient functor is defined for the heart of such categories. In the companion paper \cite{GannonClassificationOfNondegenerateGCategories}, we study the class of nondegenerate $G$-categories in much more detail and give a coherent description of nondegenerate $G$-categories via a \lq Mellin transform.\rq{}

\subsubsection{Weyl Group Action on Nondegenerate Quotient}\label{Weyl Group Action on Nondegenerate Quotient Subsubsection}Let $G$ again be an arbitrary reductive group. If $\C := \D(G/_{\psi}N^-)$, then it turns out that $\C$ is nondegenerate (which follows from \cref{No Degenerate Whittaker Objects in D(G/N)}) although there exist nondegenerate $G$-categories whose Whittaker invariants vanish. \cref{SupportOfWhittakerSheaves} implies that $\C^N$ admits a nontrivial action of the Weyl group. In general, if a DG category is not given as $\D$-modules on some space, it is difficult to construct a group action on it, as a group action on a DG category in particular contains an infinite amount of higher coherence data, see for example \cite[Section 26.4]{RomanovWilliamsonLanglandsCorrespondenceandBezrukavnikovsEquivalence}. However, our methods give a similar action of the Weyl group on $\C^N$ for any nondegenerate $G$-category: 

\begin{Theorem}\label{Nondegenerate G-Categories Have N Invariants Admit Weyl Group Action}
    If $\C$ is any nondegenerate $G$-category, $\C^{N}$ admits an action of the Weyl group which is functorial in $\C$ and such that if $\C := \D(G/_{\psi}N^-)$, the isomorphism of \cref{SupportOfWhittakerSheaves} upgrades to an isomorphism of categories with an action of $T \rtimes W$.
\end{Theorem}

This result in particular says the nondegenerate quotient of any $G$-category admits an action of the Weyl group. In the universal case where $\C^N \simeq \D(G/N)_{\text{nondeg}}$, the $W$-action on $\C^N$ is given by the \textit{Gelfand-Graev action}. We make this more precise below and recall the definition of the Gelfand-Graev action (of endomorphisms of the global differential operators on $G/N$) in \cref{Gelfand-Graev Action Construction}.

\subsection{Outline of Paper} In \cref{CategoricalPreliminaries}, we survey our conventions and prove foundational results in categorical representation theory--in particular, we prove \cref{Can Check Equivalence on Each field-valued Point for Groups}. In \cref{Nondegenerate Categories and the Weyl Group Action Section}, after making the necessary definitions, we prove \cref{Nondegenerate G-Categories Have N Invariants Admit Weyl Group Action} in a universal case and use this to deduce it for a general nondegenerate $G$-category. From this, we prove \cref{Explicit Description of Eventually Connective Kernel Objects} in \cref{Computations in BGG Category O Section}. There is one appendix, \cref{Twisted D-Modules are Derived Category of Heart Section}, which proves that some invariance categories we use are the derived categories of their hearts, and is a key technical tool in passing between abelian categories and DG categories.

\subsection{Acknowledgments} I am especially grateful to my advisor, Sam Raskin, who originally suggested the notion of nondegenerate $G$-categories to me. I would also like to thank Rok Gregoric for patiently explaining the details of the theory of $\infty$-categories to me. Moreover, I would like to thank David Ben-Zvi, Tsao-Hsien Chen, Victor Ginzburg, Sam Gunningham, Natalie Hollenbaugh, Gus Lonergan, Kendric Schefers, Brian Shin, Germ\'an Stefanich, Harold Williams, and Yixian Wu for many interesting and useful conversations. Finally, I would like to thank the anonymous referees for their careful reading of this manuscript, as well as for many helpful comments which improved the readability of this paper.
 
This project was completed while I was a graduate student at the University of Texas at Austin, and I would also like to thank everyone there for contributing to such an excellent environment.  
\section{Categorical Preliminaries}\label{CategoricalPreliminaries}
\subsection{Conventions}
Unless otherwise stated, all of our conventions regarding DG categories and derived algebraic geometry will follow \cite{GaRoI}. We highlight the main conventions we use here. 

\subsubsection{Categorical Conventions}\label{Categorical Conventions Subsubsection}
We work over a field $k$ of characteristic zero and categories will, by default, be DG categories in the sense of Chapter 1, Section 10 of \cite{GaRoI}. In other words, in the notation of \cite{GaRoI}, we are by default working in the category $\text{DGCat}_{\text{cont}}$. By definition, the objects of this category are DG categories, which are presentable stable $\infty$-categories $\C$ which are equipped with the data exhibiting $\C$ as a module category for the monoidal stable $\infty$-category Vect of $k$-vector spaces, and the functors between them are by definition \textit{continuous}, i.e. they preserve all colimits, and in particular are \textit{exact}, by assumption, meaning that they preserve cofiber sequences.\footnote{On the other hand, functors need not be $t$-exact, since not all DG categories even come equipped with a $t$-structure.} Occasionally, to emphasize our ground field, we will write $\DGCatContk$ for $\text{DGCat}_{\text{cont}}$. 

We always highlight when the categories involved are not DG categories. One prominent example will occur when, in the notation of \cite{GaRoI}, we are working in $\text{DGCat}^{\text{non-cocmpl}}$. In this case, we will say that we are working with \textit{not necessarily cocomplete} DG categories. Furthermore, when we work with a DG category $\C$ equipped with a $t$-structure, we regard the categories $\C^{\leq m}$ and $\C^{\geq m}$ for each integer $m$ as ordinary $\infty$-categories, and the \textit{eventually coconnective subcategory} is $\C^+ := \cup_{n \in \mathbb{N}} \C^{\geq -n}$ as a not necessarily cocomplete DG category. (Observe also that we use cohomological indexing for our $t$-structure.)

We will say a functor $F: \C \to \D$ of stable $\infty$-categories \textit{reflects} the $t$-structure if $F(X) \in \D^{\leq 0}$ if and only if $X \in \C^{\leq 0}$ and $F(X) \in \D^{\geq 0}$ if and only if $X \in \C^{\geq 0}$. In particular, any conservative $t$-exact functor reflects the $t$-structure. 

Furthermore, for any $\F, \mathcal{G}$ in a DG category $\C$, we let the notation $\uHom_{\C}(\F, \mathcal{G})$ denote the internal mapping object in Vect, see Chapter 1, Section 10.3.7 in \cite{GaRoI}, and similarly for $\uEnd_{\C}(\F) := \uHom_{\C}(\F, \F)$, and reserve the notation $\text{Hom}_{\C}(\F, \mathcal{G})$ for the underlying space of maps. When the underlying category is clear, we also use the notation $\uHom(\F, \mathcal{G})$ or  $\uEnd(\F)$. Finally, given two DG categories $\C, \D$, we let $\uHom(\C, \D)$ denote the DG category of maps between $\C$ and $\D$ \cite[Chapter 1, Section 10.3.6]{GaRoI}. 

\subsubsection{Other DAG Conventions}\label{Other DAG Conventions Subsubsection}
Since our field $k$ has characteristic 0, our analogue of affine schemes $\text{Sch}^{\text{aff}}$ is defined so that $\text{Sch}^{\text{aff,op}} := \text{ComAlg}(\text{Vect}^{\leq 0})$, where the right-hand side denotes the category of commutative algebra objects in the category of connective vector spaces, see \cite[Chapter 2, Section 1]{GaRoI}. We will define objects such as $\LTd\sslash\Wext$ as objects in the category of \textit{prestacks}, i.e. the $(\infty, 1)$ category of functors $\text{Fun}(\text{Sch}^{\text{aff,op}}, \text{Spc})$, where $\text{Spc}$ denotes the $(\infty, 1)$-category of spaces. By field, we mean a classical field, which is in particular an object in $\text{ComAlg}(\text{Vect}^{\heartsuit})$. 

We also follow the convention of \cite{GaRoI} of writing $\D(X) := \IndCoh(X_{dR})$ for the DG category of right $\D$-modules on a locally almost of finite type prestack $X$. A \textit{sheaf} on some prestack $X$ will mean an object of $\QCoh(X)$ or $\IndCoh(X)$ (when the latter is defined), and context will always dictate which category we are referring to. In particular, a given sheaf need not lie in the heart of the $t$-structure. 

\subsubsection{Representation Theoretic Notation}\label{Representation Theoretic Notation Section} We fix once and for all a split reductive algebraic group $G$ with Lie algebra $\LG$, and we additionally fix a choice of Borel $B$, maximal torus $T$ with Weyl group $W := N_G(T)/T$, and associated character lattice $\characterlatticeforT$. Set $N := [B, B]$ and $N^- := [B^-, B^-]$, and let $\Phi$ be the induced root system with root lattice $\mathbb{Z}\Phi$. We set $\Lambda := \{\mu \in \characterlatticeforT \otimes_{\mathbb{Z}} \mathbb{Q} : \langle \mu, \alpha^{\vee} \rangle \in \mathbb{Z} \text{ for all }\alpha \in \Phi\}$. We also view $\Lambda$ and $\characterlatticeforT$ as ind-closed subschemes of $\LTd$. 

We also fix a pinning for $G$, which we record as a group homomorphism $\psi: N^- \to \mathbb{G}_a$, which factors through $N^{-}/[N^-, N^-]$ and is \textit{nondegenerate} (by definition of pinning) in the sense that its restriction to any one-dimensional subscheme of $\text{Lie}(N^{-}/[N^-, N^-])$ corresponding to a root vector for a simple root is nontrivial. By \textit{reflection} of a vector space, we mean a diagonalizable linear endomorphism of the vector space fixing some hyperplane and having $-1$ as an eigenvalue, so that the square of any reflection is the identity. 

\subsection{Groups Acting on Categories}
In this subsection, we will review some basic properties of groups acting on categories. For further surveys, see, for example, \cite{RasTutorial} or \cite[Section 1.1]{CGYJacquet}. 

\subsubsection{Definitions}\label{Definitions for Groups Acting on Categories}
Let $H$ be any affine algebraic group, and let $\D(H)$ denote the category of (right) $\D$-modules on $H$. More generally, we let $\tilde{H}$ denote any group ind-scheme, and let $\D(\tilde{H})$ denote the category of $\D$-modules on $\tilde{H}$. The category $\D(\tilde{H})$ obtains a co-monoidal structure: the co-multiplication map is explicitly given by the composite \[\D(\tilde{H}) \xrightarrow{m^!} \D(\tilde{H} \times \tilde{H}) \xleftarrow{\sim} \D(\tilde{H}) \otimes \D(\tilde{H})\] given by pullback by the multiplication map and the inverse of the equivalence $\boxtimes$, which is an equivalence by \cite[Corollary 10.3.6]{GaiIndCoh}. Completely analogously, the category $\IndCoh(\tilde{H})$ obtains a co-monoidal structure.
    
\begin{Definition}\label{ActionDefinition}
If $\C$ is equipped with the structure of a co-module for the co-monoidal category $\D(\tilde{H})$ (respectively $\IndCoh(\tilde{H})$), we say that the category $\C$ has a \textit{strong} (respectively \textit{weak}) \textit{action} of $\tilde{H}$. If $\C$ is a category with a strong action of $\tilde{H}$, we will also say that $\C$ is an $\tilde{H}$-category.
\end{Definition}

Since the forgetful functor $\D(\tilde{H}) \to \IndCoh(\tilde{H})$ is given by pullback by the group homomorphism $\tilde{H} \to \tilde{H}_{\mathrm{dR}}$ (see \cite{GaiRozCrystals}) this forgetful functor is co-monoidal. In particular, any category with a strong $\tilde{H}$-action admits a weak $\tilde{H}$-action.
\begin{Remark}
In fact, the notion of an action of any group object of the category of prestacks can be defined using the notion of \textit{sheaves of categories}, see \cite{GaiUnivConstr}. 
\end{Remark}

\begin{Remark}
One of the major interests in studying groups acting on categories came from the strong action of loop groups on categories, see \cite[Remark 1.3.1]{Ber}. In particular, a strong action of $H$ on a category is often referred to as an \textit{action} of $H$ on the category, and we will continue this practice below.
\end{Remark}

\begin{Remark}\label{RemarkOnModulesVsComodules}
We may equivalently define an action of $H$ on $\C$ as data realizing $\C$ as a module for the monoidal category $\D(H)$ under convolution, see \cite[Section 1.1.1]{CGYJacquet}. We will occasionally use this perspective below as well. 
\end{Remark}

\begin{Remark}\label{Orbits of Field-Valued Points Make Sense if Field-Valued Points Are Sets}
Assume we have an action of $\tilde{H}$ on some prestack $\mathcal{X}$ such that $\tilde{H}(\fieldpossiblydifferentfromgroundfield)$ and $\mathcal{X}(\fieldpossiblydifferentfromgroundfield)$ are discrete spaces (i.e. sets) for some field $K$. (For example, this occurs for the action of the extended affine Weyl group $\Wext$ on $\LTd$). Then the notions of \textit{orbit} and \textit{stabilizer} with respect to $K$ are defined in the classical sense. 
\end{Remark}

\subsubsection{Invariants}\label{Invariants Subsubsection}
If $\C, \D$ are categories with an action of $\tilde{H}$, we let $\uHom_{\D(\tilde{H})}(\C, \D)$ denote the DG category of $\D(\tilde{H})$-linear maps between them. 

\begin{Definition}
Given a category $\C$ with an $\tilde{H}$-action, we define its \textit{invariants} as the category $\C^{\tilde{H}} := \uHom_{\D(\tilde{H})}(\text{Vect}, \C)$, where Vect acquires a trivial $\tilde{H}$-structure, and we similarly define the \textit{weak} invariants as $\C^{\tilde{H}, w} := \uHom_{\IndCoh(\tilde{H})}(\text{Vect}, \text{oblv}(\C))$.
\end{Definition}

\begin{Proposition}\label{Invariants of G Acting on Sheaves on X is Invariants on X Mod G}
Let $X$ be any ind-scheme, and assume $\mathscr{G}$ is an algebraic group (in particular, a classical scheme) or  group ind-scheme whose underlying prestack is a discrete set of points acting on $X$. Then
\begin{enumerate}
    \item The category $\IndCoh(X)$ canonically acquires a weak action of $\mathscr{G}$ such that the canonical functor $\IndCoh(X/\mathscr{G}) \xrightarrow{q} \IndCoh(X)$ induced by the quotient $X \xrightarrow{q} X/\mathscr{G}$ lifts to an equivalence $\IndCoh(X/\mathscr{G}) \xrightarrow{\sim} \IndCoh(X)^{\mathscr{G}, w}$.
     \item The category $\D(X)$ canonically acquires an action of $\mathscr{G}$ such that the canonical functor $\D(X/\mathscr{G}) \to \D(X)$ induced by the quotient map lifts to an equivalence $\D(X/\mathscr{G}) \xrightarrow{\sim} \D(X)^{\mathscr{G}}$.
\end{enumerate}
\end{Proposition}

\begin{proof}
Recall that the objects $\mathscr{G}^n \times X$ admit maps between them which form a simplicial object $\mathscr{G}^{\bullet} \times X$ in PreStk whose geometric realization gives the quotient $X/\mathscr{G}$ after sheafification; we review this and give explicit maps in \cite[Section 2.1.1]{GannonDescentToTheCoarseQuotientForPseudoreflectionAndAffineWeylGroups}. We may identify the category $\IndCoh(X/\mathscr{G})$ with the totalization of the corresponding  cosimplicial object given by $\IndCoh(\mathscr{G}^{\bullet} \times X)$ using the fact that $\IndCoh$ satisfies fppf descent by \cite[Corollary 10.4.5]{GaiIndCoh} when $\mathscr{G}$ is an algebraic group and using the fact that $\IndCoh$ satisfies ind-proper descent  when $\mathscr{G}$ is a discrete set of points by \cite[Chapter 3, Proposition 3.3.3]{GaRoII}. Similarly, by definition, the category $\IndCoh(X)^{\mathscr{G}}$ may be identified with the totalization of the cosimplicial object given by $\IndCoh(\mathscr{G}^{\bullet}) \otimes \IndCoh(X)$. Furthermore, we have a map of cosimplicial categories $\boxtimes: \IndCoh(\mathscr{G}^{\bullet}) \otimes \IndCoh(X) \to \IndCoh(\mathscr{G}^{\bullet} \times X)$, and, for any fixed $i$, the associated functor $\boxtimes: \IndCoh(\mathscr{G}^{i}) \otimes \IndCoh(X) \to \IndCoh(\mathscr{G}^{i} \times X)$ is an equivalence if $\IndCoh(\mathscr{G}^{i})$ is dualizable, see \cite[Corollary 10.3.6]{GaiIndCoh}. 

However, if $\indsch$ denotes any ind-inf-scheme, such as $\mathscr{G}^i$ or $(\mathscr{G}^i)_{dR} \simeq (\mathscr{G}_{dR})^i$, then the category $\IndCoh(\indsch)$ is in fact self-dual, see \cite[Chapter 3, Section 6.2.3]{GaRoII}. Therefore, this equivalence holds for $\IndCoh$, and, furthermore, the analogous proof shows for the associated categories of $\D$-modules since $\IndCoh(X_{dR}) \simeq \D(X)$ and the functor $(-)_{dR}$ commutes with limits and colimits, see \cite[Lemma 1.1.4]{GaiRozCrystals}. 
\end{proof}

\begin{Definition}\label{Monodromic Objects Definition}
If $\C$ is a category with an $H$-action, we define the category of \textit{monodromic objects} $\C^{H\text{-Mon}}$ as the full subcategory of $\C$ generated by the essential image of the forgetful functor $\C^{H} \xrightarrow{\text{oblv}^H} \C$. 
\end{Definition}

\begin{Remark}
In general, the functor $\text{oblv}^H$ need not be fully faithful unless $H$ is unipotent, see \cref{Unipotent Implies Oblv FF}. 
\end{Remark}

Using \cref{RemarkOnModulesVsComodules}, we can similarly define the \textit{coinvariants} of $H$ acting on $\C$ as $\C_{H} := \text{Vect} \otimes_{\D(H)} \C$. 

\begin{Theorem}\label{Inv=Coinv} We have the following: 
\begin{enumerate}
    \item (\cite[Appendix B]{GaiWhit}, \cite[Corollary 3.1.5]{GaiWhit}) If $\C$ is a category with an $H$-action, the forgetful functor $\text{oblv}: \C^H \to \C$ admits a continuous right adjoint, denoted $\text{Av}_*^H$. This functor induces an equivalence $\C_H \xrightarrow{\sim} \C^H$.
    \item (Gaitsgory-Lurie, \cite[Theorem 2.5.7]{Ber}) If $\C$ is a category with a weak $H$-action, the forgetful functor $\text{oblv}^{H, w}: \C^{H,w} \to \C$ admits a continuous right adjoint, denoted $\text{Av}_*^{H, w}$. This functor induces an equivalence $\C_{H, w} \xrightarrow{\sim} \C^{H, w}$.
\end{enumerate}
\end{Theorem}

This theorem also holds in the case of an infinite discrete group:

\begin{Proposition}
\label{Inv=Coinv for Group Ind-Schemes Which Are Discrete Sets of Points}
Assume $\tilde{\Lambda}$ is a group ind-scheme whose underlying set of points is discrete, and that $\C$ is a category with an action of $\tilde{\Lambda}$. Then the canonical map $\C_{\tilde{\Lambda}} \to \C^{\tilde{\Lambda}}$ is an equivalence. 
\end{Proposition}

\begin{proof}
Since the canonical map $\tilde{\Lambda} \xrightarrow{} \tilde{\Lambda}_{dR}$ is an equivalence, we need only show this for the weak action of $\tilde{\Lambda}$. However, since $\tilde{\Lambda}$ is ind-proper, the associated action maps are ind-proper, and therefore have left adjoints given by IndCoh pushforward (see Chapter 3 of \cite{GaRoII}). Therefore, our claim follows by \cite[Chapter 1, Corollary 2.5.7]{GaRoI}. 
\end{proof}

\begin{Remark}\label{UniversalCaseRemark}
The equivalences of invariants and coinvariants of \cref{Inv=Coinv} and \cref{Inv=Coinv for Group Ind-Schemes Which Are Discrete Sets of Points} have the following trivial but conceptually important consequence. Let $M$ be any closed subgroup of an affine algebraic group $H$ or some discrete group, and assume $\C$ is some category with an $H$-action. Then we have a canonical equivalence:
\raggedbottom
\[\uHom_{H}(\D(H/M), \C) \simeq \C^M\]

\noindent for which an $F \in \uHom_{H}(\D(H/M), \C)$ is canonically isomorphic to the functor $F(\delta_{1M}) \star^{H} -$, where $F(\delta_{1M}) \in C^H$. This is because:
\raggedbottom
\[\uHom_{H}(\D(H/M), \C) \simeq \uHom_{H}(\D(H)^M, \C) \simeq \uHom_{H}(\D(H)_M, \C)  \simeq \C^M\]

\noindent where the last step uses the explicit description of coinvariants as a colimit, as well as the dualizability of $\D(H)$. In particular, suppose $F(\delta_{1M})$ is contained in some $H$-subcategory of $\C$, say $\C'$. Then the entire essential image of $F$ is also contained in $\C'$. Of course, this entire discussion also applies to weak invariants and weak actions. 
\end{Remark}

\subsubsection{Character Sheaves and Twisted Invariants}\label{Character Sheaves and Twisted Invariants Subsubsection}
We briefly review the ideas of \textit{twisted invariants} of groups acting on categories in this section and \cref{Whittaker Invariance Section}, see, for example, \cite[Section 2]{CamDhi} for a more thorough treatment. Let $\mathcal{L}$ be some character sheaf on $H$. Then we may define the \textbf{twisted invariants} $\C^{H, \mathcal{L}}$ associated to $\mathcal{L}$. 

\begin{Remark}\label{Twisted Inv=Twisted Coinv}
The natural analogue to \cref{Inv=Coinv}, i.e. $\C_{H, \mathcal{L}} \xrightarrow{\sim} \C^{H, \mathcal{L}}$, holds for twisted invariants of any $H$ acting on a category, see \cite[Section 2.4]{BeraldoLoopGroupActionsonCategoriesandWhittakerInvariants}. In view of this fact, we will often abuse notation by passing interchangeably between these two categories.
\end{Remark}

Now fix any field $L$ and let $[\lambda]$ denote some $L$-point of $\LTd/\characterlatticeforT$. This gives rise to a monoidal functor \[\D(T) \simeq \IndCoh(\LTd)^{\characterlatticeforT} \xleftarrow{\sim}\IndCoh(\LTd/\characterlatticeforT) \to \text{Vect}_L\] via the Mellin transform (whose higher categorical version is reviewed in \cite[Appendix A]{GannonClassificationOfNondegenerateGCategories}), the equivalence of categories of \cref{Invariants of G Acting on Sheaves on X is Invariants on X Mod G}(1), and the pullback functor. By \cref{RemarkOnModulesVsComodules}, this in turn gives rise to a co-monoidal functor $\text{Vect}_L \to \D(T)$, which sends the one-dimensional vector space to a character sheaf $\mathcal{L}_{[\lambda]}$. Using this character sheaf, one can similarly define the twisted invariants \[\C^{T, \mathcal{L}_{[\lambda]}} := \C \otimes_{\D(T)} \mathrm{Vect}_L\] of any category $\C$ with a $T$-action or more generally the twisted invariants \[\C^{B, \mathcal{L}_{[\lambda]}} := (\C^N)^{T, \mathcal{L}_{[\lambda]}}\] for any category with a $B$-action. 

This category admits an alternative description which we now record:

\begin{Proposition}\label{Invariants for Character Sheaves Is Well Defined}
For any category $\C$ with an action of $T$, choose some lift $\lambda \in \LTd(L)$ of $[\lambda]$. Then the canonical functor $\C^{T,w} \otimes_{\IndCoh(\LTd)} \text{Vect}_L \to \C^{T, \mathcal{L}_{[\lambda]}}$ is an equivalence. 
\end{Proposition}

\begin{proof}
This follows from the isomorphism
\raggedbottom
\[\C^{T,w} \otimes_{\IndCoh(\LTd)} \text{Vect}_L \simeq \C \otimes_{\IndCoh(\LTd/\characterlatticeforT)} \IndCoh(\LTd) \otimes_{\IndCoh(\LTd)} \text{Vect}_L\] \[\simeq \C \otimes_{\IndCoh(\LTd/\characterlatticeforT)} \text{Vect}_L\]
\noindent since the $\mathcal{L}_{[\lambda]}$-invariants and coinvariants agree. 
\end{proof}

In view of \cref{Invariants for Character Sheaves Is Well Defined}, we will also use the notation $\C^{B_{\lambda}} := \C^{B, \mathcal{L}_{[\lambda]}}$. Furthermore, in view of \cref{Invariants of G Acting on Sheaves on X is Invariants on X Mod G}, we will also use the notation $\D(X/_{\lambda}B) := \D(X)^{B_{\lambda}}$ if $X$ is a scheme with a $G$-action. 
\begin{Remark}\label{Definition of Invariants Given by a Character}
We recall that to any character $H \xrightarrow{\chi} \mathbb{G}_a$, one can pull back the exponential $\D$-module to create a character sheaf $\mathcal{L}_{\chi}$ on $H$, see Section 2.4 of \cite{BeraldoLoopGroupActionsonCategoriesandWhittakerInvariants}. We shift this character sheaf so that it has unique nonzero cohomological degree $-\text{dim}(H)$, and denote the corresponding invariant category by $\C^{H, \chi}$. 
\end{Remark}

\begin{Remark}\label{Left vs. Right Disclaimer}
Recall that, given any algebraic group $H$, the \textit{opposite} algebraic group is the scheme $H$ with multiplication $h_1\cdot h_2 := h_2h_1$ and the same identity and inverse map as $H$. Many constructions in this paper will involve studying the natural action of $G \times G_{\mathrm{op}}$ on the category $\D(G)$, and studying the invariants for the $G' \times G'_{\mathrm{op}}$ (and, more generally, twisted invariants) for some closed subgroup scheme $G'$ of $G$. To ease notational clutter, we will denote this category of invariants by $\D(G)^{G' \times G'}$ rather than $\D(G)^{G' \times G'_{\mathrm{op}}}$. More generally, if $\alpha: G' \to \G_a$ is some character of $G'$, we will use the notation $\D(G)^{G' \times G', \alpha \times \alpha}$ rather than the notation $\D(G)^{G' \times G'_{\mathrm{op}}, \alpha \times \alpha}$ or the notation $\D(G)^{G' \times G', \alpha \times -\alpha}$ given by viewing the $G'_{\mathrm{op}}$ action as a $G'$-action by the inverse map. 

In view of \cref{Invariants of G Acting on Sheaves on X is Invariants on X Mod G}, if $X$ is a scheme with an action of a group $G'$, we will also use the notation $\D(X/_{\alpha}G') := \D(X)^{G', \alpha}$.
\end{Remark}

\subsubsection{Recovering $G$-Categories From Their Invariants}
We now prove \cref{Can Check Equivalence on Each field-valued Point for Groups}. Notice that if $F$ is a functor of $\D(G)$-categories (respectively, $\IndCoh(\LTd)$-categories), then any adjoint is automatically a functor of $\D(G)$-categories (respectively, $\IndCoh(\LTd)$-categories) by \cref{Adjoint Functor is Automatically G Equivariant} (respectively, the fact that $\IndCoh(\LTd) \simeq \Symt\text{-Mod}$ is rigid monoidal, see \cite[Chapter 1, Section 9]{GaRoI}) and the fact that a functor with an adjoint is an equivalence if and only if the functor and its adjoint are fully faithful. Therefore, \cref{Can Check Equivalence on Each field-valued Point for Groups} follows from the following proposition which we prove below.  
\begin{Proposition}\label{Can Check Fully Faithfulness on Each Field-Valued Point for LTd Module Categories}
    If $F: \C \to \D$ is a functor of $\IndCoh(\LTd)$-module categories which admits a continuous adjoint, then $F$ is fully faithful if and only if the induced functor $F \otimes_{\IndCoh(\LTd)} \text{id}_{\text{Vect}_L}$ is fully faithful for every field-valued point $\lambda \in \LTd(L)$.
\end{Proposition}
\begin{Remark}
We note in passing that there is a symmetric monoidal equivalence $\Upsilon_{\LTd}: \QCoh(\LTd) \xrightarrow{\sim} \IndCoh(\LTd)$ \cite[Chapter 6]{GaRoI} since $\LTd$ is smooth. In particular, we may replace $\IndCoh(\LTd)$ with $\QCoh(\LTd)$ in \cref{Can Check Fully Faithfulness on Each Field-Valued Point for LTd Module Categories}.
\end{Remark}

We will prove \cref{Can Check Fully Faithfulness on Each Field-Valued Point for LTd Module Categories} after showing the following lemma: 

\begin{Lemma}\label{Zero if and only if Zero at All Field Valued Points}
Let $A$ be a classical Noetherian commutative $k$-algebra. An object $\F$ of an $\IndCoh(\Spec(A))$-module category $\C$ is zero if and only if $\F \otimes p_{*}^{\IndCoh}(K) \simeq 0$ for all field-valued points $p: A \to K$.
\end{Lemma}

\begin{proof}
We first claim that the category $\IndCoh(\Spec(A))$ is generated by the set of elements of the form $i_{B, *}^{\IndCoh}(K_B)$ for $i_B: \Spec(B) \to \Spec(A)$ the integral closed subschemes of $\Spec(A)$ where $K_B$ is the field of fractions of $B$; here, we lightly abuse and use the terms $K_B$ and $\Xi(K_B)$ interchangeably. 

To see this, assume $\F \in \IndCoh(\Spec(A))$ is some nonzero object, and let $i: \Spec(B) \xhookrightarrow{} \Spec(A)$ denote the smallest closed subscheme such that $\mathcal{G} := i^!(\F)$ is nonzero. Since $X := \Spec(B)$ admits a proper surjective map from the disjoint union of the reduced subschemes of irreducible components, by \cite[Proposition 8.1.2]{GaiIndCoh} we have that $B$ is necessarily an integral domain. Notice that if $z$ is some closed subscheme of $X$ whose dimension is less than that of $X$ and $j$ is its complement, then we must have that the canonical map $\mathcal{G} \xrightarrow{} j_*^{\IndCoh}j^!(\mathcal{G})$ is an isomorphism, since otherwise $z^!(\F)$ would be nonzero by using \cite[Proposition 6.1.3(d)]{GaRoI} and then \cite[Proposition 6.1.3(c)]{GaRoI}. In particular, we may choose some nonzero $f \in B$ such that the image of the open embedding $j: \Spec(B[\frac{1}{f}]) \to \Spec(B)$ is contained in the smooth locus and such that $\mathcal{G} \xrightarrow{\sim} j_*(j^!(\mathcal{G}))$. 

Let $\mathcal{G}' := j^!(\mathcal{G})$ and $R := B[\frac{1}{f}]$. Since $\mathcal{G}' \simeq \Xi(M)$ for some $M \in \QCoh(\Spec(R)) = R\text{-mod}$ because $R$ is smooth, we see that \[\uHom_{\IndCoh(\Spec(R))}(R, \mathcal{G}') \simeq \uHom_{\QCoh(\Spec(R))}(R, M)\] is nonzero since $\mathcal{G}$, and thus $\mathcal{G}'$, is nonzero. However, notice that we have a cofiber sequence \[R \xrightarrow{g} R \to i_{g, *}(R/g)\] in $\QCoh(\Spec(R))$, and moreover we have that \[\uHom_{\QCoh(\Spec(R))}(i_{g, *}(R/g), M) \xrightarrow{\sim} \uHom_{\IndCoh(\Spec(R))}(\Xi(i_{g, *}(R/g)), \Xi(M))\] \[\simeq \uHom_{\IndCoh(\Spec(R))}(i_{g, *}^{\IndCoh}(\Xi(R/g)), \mathcal{G}')\] \[\simeq \uHom_{\IndCoh(\Spec(R))}(\Xi(R/g), i_{g}^{!}(\mathcal{G}')) \simeq 0\] from the fact that $\Xi$ is fully faithful, the compatibility of $\Xi$ with pushforward, and by assumption on $\mathcal{G}'$, respectively. Therefore we have that the natural multiplication map $\uHom_{\QCoh(\Spec(R))}(R, M) \xrightarrow{g} \uHom_{\QCoh(\Spec(R))}(R, M)$ is an isomorphism for any $g$ and so we see that \[\uHom_{\QCoh(\Spec(R))}(R[\frac{1}{g}], M) \xrightarrow{\sim} \uHom_{\QCoh(\Spec(R))}(R, M)\] since it can be computed as a limit of isomorphisms. Taking the limit over all nonzero $g \in R$, we finally see that \[\uHom_{\QCoh(\Spec(R))}(K_R, M) \xrightarrow{\sim} \uHom_{\QCoh(\Spec(R))}(R, M)\] where $K_R$ is the field of fractions of $R$. Thus since $R$ generates $R$-mod, we have \[0 \not\simeq \uHom_{\QCoh(\Spec(R))}(K_R, M) \xrightarrow{\sim} \uHom_{\IndCoh(\Spec(R))}(\Xi(K_R), \Xi(M))\] \[\xrightarrow{\sim} \uHom_{\IndCoh(\Spec(R))}(j_*^{\IndCoh}(\Xi(K_R)), j_*^{\IndCoh}(\mathcal{G}'))\] since $\Xi$ and $j_*^{\IndCoh}$ are fully faithful, and so \[0 \not\simeq \uHom_{\IndCoh(\Spec(R))}(K_B, \mathcal{G}) \xleftarrow{\sim} \uHom_{\IndCoh(\Spec(A))}(i_*^{\IndCoh}(K_B), \mathcal{F})\] by the fact that $K_B = K_R$ and the definition of $\mathcal{G}'$, and by adjunction, respectively. This shows the first claim; our full claim now follows from the fact that $\omega_{\Spec(A)}$ can be written as the colimit of such field-valued points. 
\end{proof}


\begin{proof}[Proof of \cref{Can Check Fully Faithfulness on Each Field-Valued Point for LTd Module Categories}]
Now assume $F: \C \to \D$ is a functor of $\QCoh(\LTd)$-module categories which admits an adjoint. We assume that $F$ admits a continuous right adjoint $R: \D \to \C$; a dual proof is valid if $F$ admits a (necessarily continuous) left adjoint. To show that $F$ is fully faithful, it suffices to show that the unit map $\text{id} \to RF$ is an equivalence. This is a natural transformation of functors, and therefore it is a natural isomorphism if and only if the induced map $C \to RF(C)$ is an equivalence for every $C \in \C$.

The map $C \to RF(C)$ is an equivalence if and only if its cofiber $K$ vanishes. By \cref{Zero if and only if Zero at All Field Valued Points}, we may check that this cofiber vanishes at all points. However, since $R$ is necessarily $\Symt$-linear by the rigid monoidality of $\Symt$-mod, and so one can check that for any field-valued $\lambda \in \LTd(L)$, $K \otimes L \in \C \otimes_{\IndCoh(\LTd)} \text{Vect}_L$ vanishes if and only if the associated unit map $\text{id}_{\C}\otimes \text{id}(C \otimes L) \to RF \otimes \text{id}(C \otimes L)$ is an equivalence, which is true since $RF \otimes \text{id} \simeq (R \otimes \text{id})(F \otimes \text{id})$ and $(F \otimes \text{id})$ is fully faithful by assumption. 
\end{proof}

We now derive the following corollary of \cref{Can Check Equivalence on Each field-valued Point for Groups}. To state it, we will use the notational shorthand $\mathrm{Av}_*^{B_{\lambda}}$ for the composite \[\C \xrightarrow{\AvN} \C^N \simeq \C^N \otimes_{\IndCoh(\LTd/\characterlatticeforT)} \IndCoh(\LTd/\characterlatticeforT) \xrightarrow{\mathrm{id}_\C \otimes [\lambda]^!} \C^N \otimes_{\IndCoh(\LTd/\characterlatticeforT)} L\mathrm{-mod} \simeq \C^{B, \mathcal{L}_{[\lambda]}}\] for any $G$-category $\C$ and any field-valued point $[\lambda] \in (\LTd/\characterlatticeforT)(L)$.

\begin{Corollary}\label{Any Nonzero Object in a G-Category has a field-valued point where averaging there won't vanish}
If $\C$ denotes some DG category with an action of $G$ and $\F \in \C$ is nonzero, then there exists some field-valued point $\lambda$ for which $\text{Av}_{\ast}^{B_{\lambda}}(\F)$ is nonzero.
\end{Corollary}

\begin{proof}
Consider the map given by including the zero category into the full $G$-subcategory $\D$ generated by $\F$. By \cref{Can Check Equivalence on Each field-valued Point for Groups}, because this functor is not an equivalence, there exists some field-valued $\lambda$ for which $\D^{B, \lambda}$ is nonzero. Therefore, the forgetful functor $\text{oblv}: \D^{B, \lambda} \to \D$ is nonzero and therefore its adjoint is nonzero as well. 
\end{proof}

The following lemma gives a related way to recover a category with an action of some torus from its weak invariants: 

\begin{Lemma}\label{Character Lattice of T Invariants Is C}
If $\C$ is a category with a $T$-action, then $(\C^{T, w})^{\characterlatticeforT} \simeq \C$.
\end{Lemma}

\begin{proof}
Note that there is a $T$-equivariant equivalence $\C \simeq \D(T) \otimes_{\D(T)} \C$. Taking the $(T, w)$-invariance, we see $\C^{T, w} \simeq \IndCoh(\LTd) \otimes_{\D(T)} \C$. Finally, taking $\characterlatticeforT$-invariants on both sides, we see that 
\raggedbottom
\[(\C^{T, w})^{\characterlatticeforT} \simeq (\IndCoh(\LTd) \otimes_{\D(T)} \C )^{\characterlatticeforT} \xleftarrow{\sim} \IndCoh(\LTd)^{\characterlatticeforT} \otimes_{\D(T)} \C \] using the fact that $\D(T)$-coinvariants equals $\D(T)$-invariants and thus commutes with limits. Therefore, since \[\IndCoh(\LTd)^{\characterlatticeforT} \otimes_{\D(T)} \C \xleftarrow{\sim} \IndCoh(\LTd/\characterlatticeforT) \otimes_{\D(T)} \C \simeq \D(T) \otimes_{\D(T)} \C \simeq \C\] by \cref{Invariants of G Acting on Sheaves on X is Invariants on X Mod G} and the Mellin transform, our claim follows.
\end{proof}

\subsubsection{Rigidity and Semi-Rigidity of Categories Related to $\D(G)$}
We recall that, as defined in \cite{BenZviNadlerCharacterTheoryofComplexGroup}, a monoidal category is \textit{semi-rigid} if it admits a set of compact generators which are dualizable on both the right and left. We also recall that a compactly generated monoidal category is \textit{rigid} in the sense of \cite[Chapter 1, Section 9]{GaRoI} if and only if it is semi-rigid and the monoidal unit is compact, see \cite[Proposition 3.3]{BenZviNadlerCharacterTheoryofComplexGroup}. In this section, we survey and prove results regarding rigid and semi-rigid monoidality of categories related to $\D(G)$, such as those in \cref{BZGO}. We first record the following more general result of Gaitsgory:

\begin{Theorem}(\cite[Lemma D.4.4]{GaiWhit})\label{Adjoint Functor is Automatically G Equivariant}
Let $\mathscr{G}$ denote any placid group ind-scheme. Then any datum of a lax or oplax equivariance on a functor of $\D(\mathscr{G})$-module categories is automatically strict. In particular, any adjoint of a $\D(\mathscr{G})$-equivariant functor acquires a canonical datum of $\D(\mathscr{G})$-equivariance. 
\end{Theorem}

Furthermore, we have the following theorem of Beraldo:

\begin{Theorem}\label{Harish-Chandra Category is Rigid Monoidal}(\cite[Proposition 2.3.11]{BeraldoLoopGroupActionsonCategoriesandWhittakerInvariants})
For any affine algebraic group $H$ of finite type, the Harish-Chandra category $\D(H)^{H \times H, w}$ is rigid monoidal.
\end{Theorem}

The \lq rigid\rq{} version of the following observation was used in \cite{BeraldoLoopGroupActionsonCategoriesandWhittakerInvariants} to prove \cref{Harish-Chandra Category is Rigid Monoidal}; we record it and the semi-rigid variant for future use: 

\newcommand{\Amathcal}{\mathcal{A}}
\newcommand{\Bmathcal}{\mathcal{B}}
\begin{Lemma}\label{Formal Lemma on Semi-Rigidity}
Assume that $\Amathcal$ is a compactly generated, monoidal category, $\Bmathcal$ is a monoidal category, and $F: \mathcal{A} \to \mathcal{B}$ is a monoidal functor of monoidal categories which admits a continuous, conservative right adjoint. Then if $\mathcal{A}$ is semi-rigid (respectively, rigid), then so too is $\mathcal{B}$.
\end{Lemma}

\begin{proof}
Assume $\Amathcal$ is semi-rigid. Then, by definition of semi-rigidity, $\Amathcal$ admits a set of compact generators which are both left and right dualizable, say by the collection of objects $A_{\alpha} \in \Amathcal$. Because $F$ is a left adjoint with continuous right adjoint, the objects $F(A_{\alpha})$ are compact in $\Bmathcal$, and because the right adjoint is conservative, this collection generates. In particular, $\Bmathcal$ is compactly generated. 

Since functors with continuous right adjoints preserve compact objects and monoidal functors preserve the monoidal unit $\mathbf{1}_{\Amathcal}$, if the monoidal unit of $\mathcal{A}$ is compact, then so too is the monoidal unit $\mathbf{1}_{\Bmathcal} \simeq F(\mathbf{1}_{\Amathcal})$ of $\Bmathcal$. Since a compactly generated monoidal category is rigid if and only if it is semi-rigid and has a compact monoidal unit (\cite[Proposition 3.3]{BenZviNadlerCharacterTheoryofComplexGroup}), it remains to show the semi-rigid version of this claim.

We have seen that the objects $F(A_{\alpha})$ are compact generators of $\Bmathcal$. Let $A^L_{\alpha}$ denote the left dual to $A_{\alpha}$. We claim $F(A^L_{\alpha})$ is the left dual to $F(A_{\alpha})$. To see this, note that, by definition of dualizability, there exists a coevaluation map $\mathbf{1}_{\Amathcal} \xrightarrow{u} A^L_{\alpha} \star_{\Amathcal} A_{\alpha}$ and a counit map $A^L_{\alpha} \star_{\Amathcal} A_{\alpha} \to \mathbf{1}_{\Amathcal}$ such that the compositions are equivalent to their respective identity maps (see \cite[Definition 2.4]{BenZviNadlerCharacterTheoryofComplexGroup}). Using the monoidality of $F$, one can check, for example, that the map 
\raggedbottom
\[\mathbf{1}_{\Bmathcal} \simeq F(\mathbf{1}_{\Amathcal}) \xrightarrow{F(u)} F(A^L_{\alpha} \star_{\Amathcal} A_{\alpha}) \simeq F(A^L_{\alpha}) \star_{\Bmathcal} F(A_{\alpha})\]

\noindent where the last equivalence uses the monoidality of $F$, gives the coevaluation map of the duality datum. An identical proof (with the roles of left and right reversed) shows that $F(A_{\alpha})$ is also right dualizable, and so $F(A_{\alpha})$ form a set of compact generators of $\Bmathcal$ which are both left and right dualizable, and so $\Bmathcal$ is semi-rigid monoidal. 
\end{proof}

We now prove a related result: 

\begin{Proposition}\label{Tw Hecke Category is Rigid Monoidal and Compactly Generated}
The category $\HNTw$ is rigid monoidal and compactly generated.
\end{Proposition}

We prove this after first recalling the following result of Ben-Zvi and Nadler:

\begin{Theorem} (\cite[Theorem 6.2]{BenZviNadlerCharacterTheoryofComplexGroup})\label{IndCoh on Double Hecke Quotient is Semi-Rigid}
The category $\IndCoh(B\backslash G/B)$ is semi-rigid. 
\end{Theorem}

\begin{Remark}
In \cite{BenZviNadlerCharacterTheoryofComplexGroup}, \cref{IndCoh on Double Hecke Quotient is Semi-Rigid} is originally stated for the category of (right) $\D$-modules $\D(B\backslash G/B)$. However, the entire proof of this theorem goes through for the category $\IndCoh(B \backslash G/B) \simeq \IndCoh(BB \times_{BG} BB)$, with the standard modifications that the Verdier duality functor $\mathbb{D}$ is replaced with the Serre duality functor on $\IndCoh(B\backslash G/B)$, defined in \cite[Section 4.4]{DrinfeldGaitsgoryOnSomeFinitenessQuestionsforAlgebraicStacks}, and the constant sheaf for $\D$-modules is replaced with the constant sheaf for $\IndCoh$.

In particular, note that the same proper map $p: BB \to BG$ and diagonal map $BB \to BB \times BB$ are used in both the $\D$-module and $\IndCoh$ cases. The analogue of the semi-rigid monoidality of $\D(X) = \D(BB)$ in \cite[Section 6]{BenZviNadlerCharacterTheoryofComplexGroup} is replaced with the \textit{rigid} monoidality of the compactly generated category $\IndCoh(BB) \simeq \text{Rep}(B)$. 
\end{Remark}

\begin{proof}[Proof of \cref{Tw Hecke Category is Rigid Monoidal and Compactly Generated}]
\noindent Let $H := N_{dR}T$, so that $H_{dR} = B_{dR}$. Recall that pushforward along
\begin{equation}\label{ind map}\IndCoh(N\backslash G/N) \xrightarrow{} \D(N \backslash G/N)\end{equation} given by the IndCoh pushforward along the natural quotient map $N\backslash G/N \to (N\backslash G/N)_{dR} \simeq N_{dR}\backslash G_{dR}/G_{dR}$ admits a continuous right adjoint, see \cite{GaiRozCrystals}. Letting $q: B\backslash G/B \to H \backslash G_{dR}/H$, we see that, under the identifications of \cref{Invariants of G Acting on Sheaves on X is Invariants on X Mod G} for the group $T \times T$, we may identify $q_*^{\IndCoh}$ with the weak $(T \times T)$-invariants of the functor \labelcref{ind map}. We therefore see that the functor $q_*^{\IndCoh}$ admits a continuous right adjoint. 

We also note that $q_*^{\IndCoh}$ is monoidal, because, for example, it is equivalently the functor given by pushforward by the Cartesian product of the quotient maps $BB \times_{BG} BB \to BH \times_{BG_{dR}} BH$. Moreover the associated right adjoint $q^!$ is conservative--for example, pulling back by the cover $s: G/B \to B\backslash G/B$, we can identify $s^!q^!$ with the composite of two forgetful functors, one which forgets the $\D$-module structure to the $\IndCoh$ module structure and one which forgets the $N, (T,w)$-equivariance. Therefore, $q_*^{\IndCoh}$ falls into the setup of \cref{Formal Lemma on Semi-Rigidity}, and furthermore we have seen that $\IndCoh(B\backslash G/B)$ is semi-rigid by \cref{IndCoh on Double Hecke Quotient is Semi-Rigid}. Therefore, using \cref{Formal Lemma on Semi-Rigidity}, we see $\HNTw$ is semi-rigid.

Because $\HNTw$ is semi-rigid, it is in particular compactly generated. A compactly generated monoidal category is rigid if and only if it is semi-rigid and the monoidal unit is compact (\cite[Proposition 3.3]{BenZviNadlerCharacterTheoryofComplexGroup}), so it remains to verify that the monoidal unit is compact. Note that we have a functor $\D(T)^{T \times T, w} \xhookrightarrow{i_*^{\IndCoh}} \HNTw$ induced by the composite $T \cong N\backslash B/N \xhookrightarrow{i} N\backslash G/N$ which is monoidal. Furthermore, it admits a continuous right adjoint--the fact that $i^!: \D(N \backslash G/N) \to \D(T)$ is $T \times T$-linear follows from \cref{Adjoint Functor is Automatically G Equivariant}. Therefore, the monoidal unit of $\HNTw$ is mapped to by a compact object in $\D(T)^{T \times T, w}$ (namely, the monoidal unit, which is compact by \cref{Harish-Chandra Category is Rigid Monoidal}) via a functor admitting a continuous right adjoint, and thus is compact. 
\end{proof}

Finally, we prove a result analogous to \cref{Adjoint Functor is Automatically G Equivariant}:

\begin{Corollary}\label{Continuous Right Adjoint on N Invariants Implies HN Linear} Assume $F: \C \to \D$ is a map of $G$-categories and denote by $F^N: \C^N \to \D^N$ the associated functor on the $N$-invariant subcategory. Assume $F^N$ admits a continuous right adjoint (respectively, left adjoint). Then $F$ itself admits a continuous, $G$-equivariant right adjoint $R: \D \to \C$ (respectively, $G$-equivariant left adjoint). In particular, the adjoint of $F^N$ is $\mathcal{H}_N$-linear.
\end{Corollary}

\begin{proof}
The induced map $F^N: \C^N \to \D^N$ is necessarily $T$-equivariant, and so by \cref{Adjoint Functor is Automatically G Equivariant}, the adjoint functor is also $T$-equivariant. In particular, applying $(-)^{T,w}$ to this adjoint pair, we see that the functor $F^{N, (T,w)}: \C^{N, (T,w)} \to \D^{N, (T,w)}$ admits a continuous right adjoint, and therefore so too does the functor $\text{id} \otimes F: \D(G/N)^{T,w} \otimes_{\text{Vect}} \C^{N, (T,w)} \to \D(G/N)^{T,w} \otimes_{\text{Vect}} \D^{N, (T,w)}$. Furthermore, by \cref{BZGO}, the following diagram commutes:
\raggedbottom
\begin{equation*}
  \xymatrix@R+2em@C+2em{
  \D(G/N)^{T,w} \otimes_{\text{Vect}} \C^{N, (T,w)} \ar[r]^{\text{id} \otimes F^{N, (T,w)}}\ar[d]^{\text{act}} & \D(G/N)^{T,w} \otimes_{\text{Vect}} \D^{N, (T,w)} \ar[d]^{\text{act}}\\
    \C \ar[r]^{F}  & \D 
  }
 \end{equation*}
 
 \noindent where the vertical arrows are given by the action maps. The category $\HNTw$ is rigid monoidal (\cref{Tw Hecke Category is Rigid Monoidal and Compactly Generated}) and so the action map here admits a continuous right adjoint, see \cite[Chapter 1, Lemma 9.3.3]{GaRoI}. However, the right adjoint to the action map is also conservative. Therefore, we may check whether the right adjoint to $F$ commutes with colimits after applying the right adjoint to the action map on $\C$. However, the diagram given by the right adjoints of the above diagram commutes, and the right adjoints given by the upper and rightmost arrows are both continuous. Therefore, we see that the right adjoint $R$ to $F$ commutes with colimits. Since $F$ is $G$-equivariant, its adjoint is $G$-equivariant by \cref{Adjoint Functor is Automatically G Equivariant}. Therefore, $R^N$ will be $\HN$-linear, as desired. 
  Now assume $F^N$ admits a left adjoint. Recall the \textit{horocycle functor} $\text{hc}$: for a category $\mathcal{C}$ with an action of $G \times G$, it is defined as the composite of the functors \[\C^{\Delta G} \xrightarrow{\mathrm{oblv}}\C^{\Delta B} \xrightarrow{\AvN} \C^{(N \times N)\Delta T}\] where for any group $H$ we use the notation $\Delta H$ for the diagonal copy of $H$ inside $H \times H$. Since $\text{hc}$ is defined as a composite of the forgetful functor and averaging functor, we see that the following diagram commutes:
\raggedbottom
\begin{equation*}
  \xymatrix@R+2em@C+2em{
   \D(G) \otimes_G \C \ar[r]^{F} \ar[d]^{\text{hc}} & \D(G) \otimes_G \D \ar[d]^{\text{hc}} \\
   \D(G/N) \otimes_T \C^N \ar[r]^{\text{id} \otimes F^N} & \D(G/N) \otimes_T \D^N
  }
\end{equation*}

It is a standard result that the functor $\text{hc}$ admits a left adjoint ch, see, for example, \cite[Section 2.6.1]{BZG} for an alternate description of this functor and its left adjoint. Moreover, following \cite[Section 2.6.1]{BZG} and \cite[Section 4.6]{RaskinAffineBBLocalization}, the functor $\text{hc}$ is conservative, since, for example, the composite $\text{hc}^L\circ \text{hc} \simeq \text{ch}\circ \text{hc}$ can be identified with the functor which convolves with the Springer sheaf, a sheaf which has the sheaf $\delta_{1G} \in \D(G/G)$ as a direct summand. Therefore, to see that $F$ commutes with limits, it suffices to show that $\text{hc}F$ commutes with limits, and by the commutativity of the diagram, this follows since $\text{hc}$ and $\text{id} \otimes F^N$ are both right adjoints. 
\end{proof}

\subsubsection{Completeness of $t$-Structures}
We now recall the following definition for later use:

\begin{Definition}\label{Left and right-complete t-structure Definition}
Let $\C$ be a category equipped with a $t$-structure. We say an object $\F \in \C$ is its \textit{right-completion} (respectively, is its \textit{left-completion}) with respect to the $t$-structure if the natural map $\text{colim}_n\tau^{\leq n}\F \to \F$ (respectively, the natural map $\F \to \text{lim}_m\tau^{\geq m}\F)$ is an equivalence. A $t$-structure on a stable or DG category $\C$ is said to be \textit{right-complete} (respectively, \textit{left-complete}) if all of its objects are their right-completions (respectively, their left-completions). 
\end{Definition}

\begin{Lemma}\label{Conservative t Exact Functor to right-complete Implies right-complete and Dual Statement}
Let $\C$ and $\D$ be DG categories equipped with $t$-structures, and let $F: \C \to \D$ be a (continuous) functor between them which is $t$-exact and conservative. 
\begin{enumerate}
    \item If $\D$ is right-complete with respect to its $t$-structure, then $\C$ is also right-complete with respect to its $t$-structure. 
    \item If $\D$ is left-complete with respect to its $t$-structure, then if $F$ commutes with (small) limits, $\C$ is also left-complete with respect to its $t$-structure. 
\end{enumerate}

\end{Lemma}

\begin{proof}
Assume $C \in \C$, and consider the map $C \xleftarrow{\phi} \text{colim}_n \iota^{\leq n}\tau^{\leq n}(C)$ in $\C$, where $\iota^{\leq n}$ denotes the inclusion. By the conservativity of $F$, it suffices to show that $F(\phi)$ is an equivalence. However, we have:
\raggedbottom
\[F(\text{colim}_n\iota^{\leq n}\tau^{\leq n}(C)) \simeq \text{colim}_nF(\iota^{\leq n}\tau^{\leq n}(C))\] \[\simeq\text{colim}_n\iota^{\leq n}\tau^{\leq n}(F(\iota^{\leq n}\tau^{\leq n}(C))) \simeq \text{colim}_n \iota^{\leq n}\tau^{\leq n}(F(C))\]

\noindent where the first equivalence uses the continuity of $F$, the second step uses the right $t$-exactness of $F$, and the third step uses the left $t$-exactness of $F$. By the right-completeness of $\D$, this composite map is an equivalence. However, we can identify this composite map with $F(\phi)$, so we see that $F(\phi)$ is an isomorphism, as desired. 
An argument dual to the above gives (2). Specifically, if $C \in \C$, then we have equivalences:
\raggedbottom
\[F(\text{lim}_m\iota^{\geq m}\tau^{\geq m}(C)) \simeq \text{lim}_mF(\iota^{\geq m}\tau^{\geq m}(C)) \] \[\simeq \text{lim}_m\iota^{\geq m}\tau^{\geq m}F(\iota^{\geq m}\tau^{\geq m}(C)) \simeq \text{lim}_m\iota^{\geq m}\tau^{\geq m}F(C)\]

\noindent where the first step uses the assumption that $F$ commutes with small limits, the second step uses the fact that $F$ is left $t$-exact, and the third uses the right $t$-exactness of $F$. Therefore, as above, the conservativity of $F$ gives that the $t$-structure on $\C$ is left-complete provided the $t$-structure on $\D$ is. 
\end{proof}

\subsubsection{t-Structures on $H$-Categories} \label{t-structures on G Categories}
For later use, we will also recall the $t$-structure on categories with a group action. We first recall a result of \cite{GaiFun}:

\begin{Lemma} (\cite[Lemma 4.1.3]{GaiFun})\label{GaitsgoryExactnessLemma}
Assume $\C, \D_1$, and $\D_2$ are DG categories equipped with $t$-structures and we are given a (continuous) functor $F: \D_1 \to \D_2$. Then the categories $\C \otimes \D_i$ each inherit a $t$-structure where the objects of $(\C \otimes \D_i)^{\leq 0}$ are generated under colimits by objects of the form $c \otimes d$ where $c \in \C^{\leq 0}$ and $d \in \D^{\leq 0}$. Furthermore:
\begin{enumerate}
    \item If $F: \D_1 \to \D_2$ is right $t$-exact, then so too is the functor $\text{id}_{\C} \otimes F: \C \otimes \D_1 \to \C \otimes \D_2$.  
    \item If $F: \D_1 \to \D_2$ is left $t$-exact, then so too is the functor $\text{id}_{\C} \otimes F: \C \otimes \D_1 \to \C \otimes \D_2$, if the $t$-structure on $\C$ is compactly generated, i.e. $\C^{\leq 0}$ is generated under colimits by the objects of $\C^{\leq 0}$ which are compact in $\C$.  
\end{enumerate}
\end{Lemma}

Now, as above, we let $H$ be any affine algebraic group. We note that, in particular, $\IndCoh(H)$ and $\D(H)$ satisfy the hypotheses of \cref{GaitsgoryExactnessLemma} since, for example, each category is the derived category of its heart since $H$ is smooth--see \cite[Proposition 4.7.3]{GaiRozCrystals} for the $\D(H)$ case. We make the following definition: 

\begin{Definition} \cite[Appendix B.4]{Ras2}
We say that a $t$-structure on an $H$-category $\C$ is \textit{compatible} with the $H$-action if the coaction map $\text{coact}[-\text{dim}(H)]: \C \to \D(H) \otimes \C$ is $t$-exact.
\end{Definition}

\begin{Example}\label{If Smooth Group Acts on Scheme Induced Action on Sheaves is Compatible with t-Structure}
Assume $H$ acts on some scheme $X$. Then the action map $a: H \times X \to X$ is smooth (since there is an isomorphism $H \times X \cong H \times X$ which takes the action map to the projection map, which is smooth), and so in particular the coaction map $a^![-\text{dim}(H)]$ is $t$-exact. 
\end{Example}

Let $\C$ denote a DG category with an action of an algebraic group $H$, and let $\chi: H \to \mathbb{G}_a$ be any character. Recall the construction $\C^{H, \chi}$ of \cref{Definition of Invariants Given by a Character}. If $\C$ is equipped with a $t$-structure compatible with the action of $H$, then, following \cite[Appendix B.1-B.4]{Ras2}, we can equip $\C^{H, \chi}$ with a unique $t$-structure such that the forgetful functor is $t$-exact. 

\begin{Proposition}\label{t-Exact Implies Invariants Are}
If $F: \C \to \D$ is a $t$-exact functor of $G$-categories and $H \leq G$ is some closed subgroup of $G$, then the induced functor $\tilde{F}: \C^{H, \chi} \to \D^{H, \chi}$ is $t$-exact.
\end{Proposition}

\begin{proof}
By assumption, the following diagram canonically commutes:

\begin{equation*}
  \xymatrix@R+2em@C+2em{
   \C^{H, \chi} \ar[r]^{\tilde{F}} \ar[d]^{\text{oblv}} & \D^{H, \chi} \ar[d]^{\text{oblv}}\\
   \C \ar[r]^{F} & \D
  }
 \end{equation*}
 
\noindent where oblv denotes the forgetful functor. Assume $\F \in \C^{H, \chi, \geq 0}$. Then we have $\text{oblv}(\F) \in \C^{\geq 0}$ and so by the commutativity of the diagram above, we see that $\text{oblv}(\tilde{F}(\F)) \in \D^{\geq 0}$. In particular, we see that $\text{oblv}(\tau^{< 0}(\F)) \simeq 0$ by the $t$-exactness of oblv. Since oblv is conservative, we see that $\tau^{< 0}(\F) \simeq 0$, and therefore $\tilde{F}(\F) \in \D^{H, \chi, \geq 0}$. Repeating this same argument replacing the coconnective categories with the respective connective categories (and replacing $\tau^{< 0}$ with $\tau^{>0}$), an identical argument shows $\tilde{F}$ preserves connective objects and thus $\tilde{F}$ is $t$-exact.
\end{proof}
    
\begin{Definition}\label{Definition of Bounded Cohomological Dimension of Category and Cohomological Amplitude of Functor} Let $\C$ be a category equipped with a $t$-structure.
\begin{enumerate}
    \item We say an object $C \in \C$ is \textit{cohomologically bounded} if there exists $m, n \in \mathbb{Z}$ such that $\tau^{> n}C$ and $\tau^{< m}C$ vanish, so that $C \simeq \tau^{\leq n}\tau^{\geq m}C$. 
    \item We say $\C$ has \textit{bounded cohomological dimension} if there exists some $n > 0$ such that if $X \in \C^{\geq n}$ and $Y \in \C^{\leq 0}$, then the space of maps $\text{Hom}_{\C}(X, Y)$ is connected. 
    \item Let $\D$ be a category equipped with a $t$-structure, and let $F: \C \to \D$ be a functor between them. We say that $F$ has \textit{cohomological amplitude in} $[m,n]$ for some $m, n \in \mathbb{Z}$ if for all $C \in \C^{\heartsuit}$, the cohomology of $F(C)$ given by the $t$-structure is concentrated in degrees $[m, n]$, and furthermore we say it has \textit{bounded cohomological amplitude} if it has cohomological amplitude in $[m,n]$ for some integers $m$ and $n$. 
\end{enumerate}
\end{Definition}

\begin{Proposition} Let $F: \C \to \D$ be a functor of DG categories (or any stable $\infty$-categories) equipped with $t$-structures. 
\begin{enumerate}
    \item Assume further that the $t$-structure on $\C$ is right-complete as in \cref{Left and right-complete t-structure Definition} and the $t$-structure on $\D$ is compatible with filtered colimits. If $F$ maps $\C^{\heartsuit}$ to $\D^{\geq m}$ (which in particular occurs for some $m$ if $F$ has bounded cohomological amplitude), then $F$ maps $\C^{\geq 0}$ to $\D^{\geq m}$. 
    \item Further assume that $F$ commutes with cofiltered limits. Then if the $t$-structure on $\C$ is left-complete and the $t$-structure on $\D$ is compatible with cofiltered limits, then if $F$ maps $\C^{\heartsuit}$ to $\D^{\leq n}$, then $F$ maps $\C^{\leq 0}$ to $\D^{\leq n}$.
\end{enumerate}
\end{Proposition}

\begin{proof}
We first claim that $F$ sends $\C^{[0, q]} := \C^{\geq 0} \cap \C^{\leq q}$ to $\D^{\geq m}$. To see this, we induct on $q$, noting that the base case follows from definition of having cohomological amplitude in $[m, n]$. By the inductive step, note that any $\F \in \C^{[0, q]}$ for $q > 0$ admits a cofiber sequence $\tau^{\leq q-1}\F \to \F \to H^q(\F)[-q]$, where we omit the inclusion functors given by the $t$-structure. The inductive hypothesis implies that $F(\tau^{\leq q-1}\F) \in \D^{\geq m}$ and the fact that $F$ is exact implies that $F(H^q(\F)[-q]) \simeq F(H^q(\F))[-q]$ lies in $\D^{[m + q, n + q]} \subseteq \D^{\geq m}$. Now, for a general $\F \in \C^{\geq 0}$, by right-completeness we may write $\F \xleftarrow{\sim} \text{colim}_n\tau^{\leq n}(\F)$. Apply $F$, which commutes with colimits, to see that $F(\F) \xleftarrow{\sim} \text{colim}_nF(\tau^{\leq n}(\F))$. Therefore, by the above, we see that $F(\F)$ is a filtered colimit of objects of $\D^{\geq m}$, and therefore, since the $t$-structure on $\D$ is compatible with filtered colimits, we see that $F(\F) \in \D^{\geq m}$. 

The dual proof holds to prove claim (2), where one similarly first inducts on $q$ to show that $\C^{[-q, 0]} \in \\D^{\leq n}$ and uses left-completeness of $\C$, the fact that $F$ commutes with cofiltered limits, and the fact that $\D^{\leq 0}$ is closed under cofiltered limits to show the general claim.
\end{proof}

We now state two results in Appendix B of \cite{Ras2}:

\begin{Proposition}\label{CohomologicalAmplitudesForGActingOnC} Assume we are given a category $\C$ which is additionally equipped with an action of $H$ compatible with the $t$-structure, and let $\chi: H \to \mathbb{G}_a$ be any character.
 \begin{enumerate}
    \item (\cite[Lemma B.4.1]{Ras2}) The right adjoint $\text{Av}_*^{H, \chi}: \C \to \C^{H, \chi}$ to the forgetful functor has cohomological amplitude in $[0, \text{dim}(H)]$. 
    \item (\cite[Lemma B.6.1]{Ras2}) Assume $\C$ has a compactly generated $t$-structure. Then the partially defined left adjoint to the forgetful functor, denoted $\text{Av}_!^{H, \chi}: \C \to \C^{H, \chi}$, has cohomological amplitude $[-\text{dim}(H), 0]$. 
\end{enumerate}
\end{Proposition}

\newcommand{\charsheaf}{\mathcal{L}_{\chi}}
\begin{proof}
We show the methods of \cite[Appendix B]{Ras2} adapt to show the more general character case as in (1). Consider the underlying sheaf of our character sheaf, $\charsheaf \in \D(H)$, associated to $\chi$. By our conventions above, $\charsheaf$ is in cohomological degree $-\text{dim}(H) = -1 + (-\text{dim}(H) + 1)$. Since we may identify the averaging functor with the functor $m_{*, \text{dR}}(- \boxtimes \charsheaf[2\text{dim}(H)])$, the fact that $m$ is affine implies that the pushforward $m_{*, dR}$ is right $t$-exact, and so in particular this functor will have cohomology only in degrees bounded by $\text{dim}(H)$. Since $\text{Av}_*^{H, \chi}$ is a right adjoint to a $t$-exact functor, it is in particular left $t$-exact and thus, combining these results, we obtain (1) when $\C = \D(H)$.  

To demonstrate the argument for general $\C$, note that, by assumption, the functor $\text{coact}[-\text{dim}(H)]: \C \xrightarrow{} \D(H) \otimes \C$ is $t$-exact and therefore, in particular, by \cref{t-Exact Implies Invariants Are} we obtain that the induced functor on invariants $\C^{H, \chi} \xrightarrow{} \D(H)^{H, \chi} \otimes \C$ is $t$-exact. We also have that the following diagram commutes:

\begin{equation*}
  \xymatrix@R+2em@C+2em{
   \C \ar[r]^{}\ar[d]^{\text{Av}_*^{H, \chi}} & \D(H) \otimes \C \ar[d]^{\text{Av}_*^{H, \chi} \otimes \text{id}}\\
    \C^{H, \chi} \ar[r]^{}  & \D(H)^{H, \chi} \otimes \C 
  }
 \end{equation*}
 
 \noindent where the horizontal arrows are the shifted coaction maps as above. Since the coaction maps are conservative (for example, one can pull back by the identity morphism $\ast \to H$ to obtain the identity map) and $t$-exact, we obtain our claim by \cref{GaitsgoryExactnessLemma}.
\end{proof}

\subsection{Whittaker Invariants}\label{Whittaker Invariance Section}
We now recall the notion of the Whittaker invariants of a category, discussed above in \cref{Analogue of V Subsubsection}.

\subsubsection{Whittaker Invariants and Unipotent Groups}
Recall we have fixed a pinned reductive group in \cref{Representation Theoretic Notation Section}; in particular, we have fixed an additive character $\psi: N^- \to \mathbb{G}_a$ which is nondegenerate. 

\begin{Definition}
Given a category $\C$ with a $G$-action, we define its \textit{Whittaker invariants} as the category $\uHom_{\D(N^{-})}(\text{Vect}_{\psi}, \C)$.
\end{Definition}
We record one result on the twisted invariance of \textit{unipotent groups}:

\begin{Lemma}\label{Unipotent Implies Oblv FF} \cite[3.3.5]{Ber}
    The forgetful functors $\C^{N^{-}, \psi} \xrightarrow{\text{oblv}} \C$ and $\C^{N} \xrightarrow{\text{oblv}} \C$ are fully faithful. 
\end{Lemma}

\begin{Remark}
Strictly speaking, \cite{Ber} above does not discuss the case $\C^{N^{-}, \psi} \xrightarrow{\text{oblv}} \C$. However, after replacing the constant sheaf on $N^{-}$ with the sheaf $\psi$ gives identical results, largely due to the unipotence of $N^{-}$.
\end{Remark}
    
Because of these results, we will view $C^N$ and $C^{N^{-}, \psi}$ as full subcategories of $\C$. 

\newcommand{\wdot}{\dot{w}}
\newcommand{\Gaminusalpha}{\mathbb{G}_a^{-\alpha}}
\subsubsection{A Universal Example of a Whittaker Category}\label{Proof of Support of Whittaker Sheaves} By \cref{BZGO}, the category $\D(G/N)^{N^{-}, \psi}$ provides a universal example of the Whittaker invariants of a $G$-category. We compute this explicitly by showing $\D(G/N)^{N^{-}, \psi} \xrightarrow{\sim} \D(T)$, i.e. by proving \cref{SupportOfWhittakerSheaves}. We prove \cref{SupportOfWhittakerSheaves} after proving the following lemma:

\begin{Lemma}\label{No Simultaneously Monodromic and Twisted Monodromic Ga Objects}
For any category $\C$ with an action of $\mathbb{G}_a$, if $\psi: \mathbb{G}_a \to \mathbb{G}_a$ is a nontrivial character, then $\C^{\mathbb{G}_a, \psi\text{-mon}} \cap \C^{\mathbb{G}_a\text{-mon}} \simeq 0$.
\end{Lemma}

\begin{proof}
If $\F \in \C^{\mathbb{G}_a\text{-mon}}$, there exists some isomorphism $\text{act}^!(\F) \cong \F \boxtimes \omega_{\mathbb{G}_a}$. In this case, if $t: \mathbb{G}_a \to \ast$ is the terminal map, then we see that 
\raggedbottom
\[\F \simeq \F \otimes k \simeq (\text{id}_{\C} \otimes t_{*, dR})(\F \boxtimes \omega_{\mathbb{G}_a})[-2] \cong (\text{id}_{\C} \otimes t_{*, dR})(\text{act}^!(\F))[-2]\]

Similarly, if $\F \in \C^{\mathbb{G}_a, \psi\text{-mon}}$, there is an isomorphism $(\text{act}^!(\F)) \cong \F \boxtimes \mathcal{L}_{\psi}$, where $\mathcal{L}_{\psi}$ is the shifted exponential $\D$-module on $\mathbb{G}_a \simeq \mathbb{A}^1$. Thus, continuing the chain of isomorphisms above, we see that if $\F$ is also $\psi$-monodromic, then $\F$ is equivalent to 
\raggedbottom
\[(\text{id}_{\C} \otimes t_{*, dR})(\F \boxtimes \mathcal{L}_{\psi}[-2]) \simeq \F \otimes t_{*, dR}(\mathcal{L}_{\psi})[-2] \simeq 0\]

\noindent since de Rham pushforward to a point is given by de Rham cohomology (up to shift), which vanishes on the exponential $\D$-module on $\mathbb{A}^1$. 
\end{proof}

\begin{Example}\label{NoMonodromicWhittakerForVect}
Consider the category Vect with the trivial $G$-action where $G$ is not its maximal torus. Then $\text{Vect}^{N^{-}, \psi} \simeq 0$ since every object is $N^-$-monodromic.
\end{Example}

\begin{proof}[Proof of \cref{SupportOfWhittakerSheaves}]
Fix some $\F \in\D(G/N)^{N^{-}, \psi}$, and let $j: N^-B/N \xhookrightarrow{} G/N$ denote the open embedding. We wish to show $\F \xrightarrow{\sim} j_{*, dR}j^!(\F)$. To do this, it suffices to show that the restriction to the complementary closed subset vanishes. In turn, by backwards induction on the length of the Weyl group element, it suffices to show that the restriction of $\F$ to each Schubert cell $N^-\wdot B/N$ vanishes, where $\wdot \in N_G(T)$ is some arbitrarily chosen lift of a given element $w \in W$ where $w \neq 1$. 

We claim that, furthermore, we have $\D(N^-\wdot B/N)^{N^{-}, \psi} \simeq 0$ if $w \neq 1$. To see this, note that because $\psi: N^- \to \mathbb{G}_a$ induces a map of group schemes $\tilde{\psi}: N^-/[N^-, N^-] \to \mathbb{G}_a$, we may equivalently show the identity 
\raggedbottom
\[\D([N^-, N^-]\backslash N^-\wdot B/N)^{N^{-}/[N^-, N^-], \tilde{\psi}} \simeq 0.\]

We may write $N^-/[N^-, N^-] \simeq \prod_{\alpha} \Gaminusalpha$ as a product of copies of the additive group $\mathbb{G}_a$, where $\alpha$ varies over the simple roots. Since $w \neq 1$, there exists some simple root $\alpha$ such that $\wdot^{-1} \Gaminusalpha\wdot \subseteq N$. We therefore see that this action of $\Gaminusalpha$ is trivial, since if $m$ is some $R$-point of $[N^-, N^-]\backslash N^-$ and $b$ is some $R$-point of $B/N$, then if $x \in \Gaminusalpha(R)$ then
\raggedbottom
\[x(m\wdot b) = mx\wdot b = m\wdot (\wdot^{-1}x\wdot)b = m\wdot b\]

\noindent in $G/N$, where the last step follows because $(\wdot^{-1}x\wdot) \in N$ and uses the fact that the action of $N$ on $B/N$ is trivial. We therefore see that the $\Gaminusalpha$ action on $\D([N^-, N^-]\backslash N^-\wdot B/N)$ is trivial. In particular, all objects are monodromic with respect to this $\Gaminusalpha$, and so no nonzero objects are monodromic with respect to some nontrivial character by \cref{No Simultaneously Monodromic and Twisted Monodromic Ga Objects}. 
\end{proof} 

In fact, \cref{SupportOfWhittakerSheaves} can be extended more generally, which we record for later use: 

    \begin{Proposition}\label{NoMonodromicWhittaker}
    Let $\C$ be some category with an action of a reductive group $G$, and let $\alpha$ be the negative of a simple root with associated parabolic subgroup $P_{\a}$. If $\F \in \C^{N^-, \psi}$ is $Q_{\a} := [P_{\a}, P_{\a}]$-monodromic, $\F \simeq 0$. 
    \end{Proposition}
    
    \newcommand{\SLalpha}{M_s}
    \newcommand{\Avstar}{\text{Av}_*}
    \begin{proof}
    Let $s \in W$ denote the simple reflection associated to $\alpha$, and consider the Levi decomposition $P_{\a} = U_{w_0s} \rtimes L_s$. Then we see that $Q_{\alpha}$ contains a connected simple closed algebraic subgroup $M_s$ of rank-one. Note that the averaging functor $\C \xrightarrow{\Avstar^{Q_{\a}}} \C^{Q_{\a}}$ is conservative on the subcategory of $Q_{\a}$-monodromic objects. In particular, the averaging with respect to the $M_s$-action is conservative on this subcategory.
    
    We consider the restriction of the averaging functor $\Avstar^{\SLalpha}$ to the Whittaker subcategory, which is in particular given by the composite:
    \raggedbottom
    \[\C^{N^{-}, \psi} \xhookrightarrow{\text{oblv}} \C^{\G_a^{\alpha}, \omega_{\text{exp}}} \xhookrightarrow{\text{oblv}^{\alpha}} \C \xrightarrow{\Avstar^{\SLalpha}} \C^{\SLalpha}\]
    
    \noindent where the two leftmost functors are fully faithful by \cref{Unipotent Implies Oblv FF}. However, by \cref{UniversalCaseRemark} we have that the composite functor $\Avstar^{\SLalpha}\text{oblv}^{\alpha}$ is given by convolution with an object in $\D(\SLalpha/\SLalpha)^{\G^{\alpha}_a, \omega_{\text{exp}}} \simeq 0$. Therefore, if an object of $\C$ lies in the subcategory $\C^{N^{-}, \psi}$ and the subcategory of $Q_{\a}$-monodromic objects of $\C$, it must be the zero object. 
    \end{proof}
    
\subsubsection{Averaging Functors}\label{Averaging Functors Subsubsection}
Let $\C$ be a category with a $G$-action, and consider the partially defined left adjoint \cite[Appendix A]{DrinfeldGaitsgoryOnATheoremofBraden} to the forgetful functor $\text{oblv}: \C^{N^{-}, \psi} \to \C$, which we denote by $\Avpsi$. We will refer to this functor as the \textit{Whittaker averaging functor} and, in the case of a category with two commuting $G$-actions (such as $\D(G)$), we will use the terms \textit{left Whittaker averaging} and \textit{right Whittaker averaging} to emphasize which action we are averaging with respect to. 

Let $d$ denote the dimension of $N^-$. Then, following \cite[Corollary 3.3.8]{Ber}, we obtain a canonical map $\text{Av}_!^{\psi} \to \text{Av}_*^{\psi}[2d]$. With this, we can now review one of the main results of \cite{BBM}; see \cite[Section 2.7.1]{Ras2}, the proof of \cite[Corollary 7.3.1]{Ras2}, and \cite[p. 14]{Ras4} for further discussion:
    
    \begin{Theorem}\label{BBMAdjointTheorem}
    The canonical map $\Avpsi \to \text{Av}_*^{\psi}[2d]$ is an equivalence when restricted to the full subcategory $\C^N$. In particular, the functor $\AvN: \C \to \C^N$ admits a left adjoint when restricted to $\C^{N^{-}, \psi}$. 
    \end{Theorem}

\begin{Corollary}\label{BBMShiftedLeftAdjointIsExact}
In the notation of \cref{BBMAdjointTheorem}, if $\C$ has a $G$-action compatible with the $t$-structure, then the functor $\Avpsi[-\text{dim}(N)] \simeq \text{Av}_*^{\psi}[\text{dim}(N)]$ is $t$-exact.
\end{Corollary}

\begin{proof}
We have that $\Avpsi[-\text{dim}(N)] \simeq \text{Av}_*^{\psi}[\text{dim}(N)]$ simultaneously has cohomological amplitude in $[0, \text{dim}(N)]$ and in $[-\text{dim}(N), 0]$ by \cref{CohomologicalAmplitudesForGActingOnC}(1) and \cref{CohomologicalAmplitudesForGActingOnC}(2), respectively. 
\end{proof}

As a consequence of \cref{BBMShiftedLeftAdjointIsExact}, we see that the right adjoint functor $\AvN[\text{dim}(N)]$ is left $t$-exact. We also recall the following result of Ginzburg which also gives right $t$-exactness:

\begin{Theorem}\label{Ginzburgt-Exactness} \cite[Theorem 1.5.4]{Gin}
The functor $\AvN[\text{dim}(N)]: \D(G)^{N^{-}, \psi} \to \D(G)^N$ is $t$-exact.
\end{Theorem}

\begin{Remark}
To compare the above with the notation of \cite{Gin}, we have $\Omega_{N} \simeq \omega_{N}[-\text{dim}(N)] \simeq \underline{k}_N[\text{dim}(N)]$, where $\underline{k}_N$ denotes the constant sheaf. (The notation $\Omega_{N}$ will not be used outside this remark.)
\end{Remark}

\begin{Corollary}\label{G Category with Compactly Generated Compatible tStructure Means Avpsi Is T Exact Up to Shift}
If $\C$ is a $G$-category equipped with a compactly generated $t$-structure compatible with the $G$-action, the functor 
\raggedbottom
\[\C^{N^{-}, \psi} \xhookrightarrow{} \C \xrightarrow{\AvNshifted} \C^N\]

\noindent is $t$-exact.
\end{Corollary}

\begin{proof}
When $\C = \D(G)$, this result follows from \cref{Ginzburgt-Exactness}. For general $\C$ equipped with a compactly generated $t$-structure compatible with the $G$-action, this follows since the diagram

\begin{equation*}
  \xymatrix@R+2em@C+2em{
   \C^{N^{-}, \psi} \ar[r]^{\text{oblv}} \ar[d]  & \C \ar[r]^{\AvNshifted} \ar[d] &  \C^N \ar[d]\\
    \D(G)^{N^{-}, \psi} \otimes \C \ar[r]^{\text{oblv} \otimes \text{id}_{\C}}  & \D(G) \otimes \C \ar[r]^{\AvNshifted \otimes \text{id}_{\C}} & \D(G)^N \otimes \C
  }
 \end{equation*}
 
\noindent commutes, where the vertical arrows are the coaction maps, which are conservative and $t$-exact (up to shift), as in the proof of \cref{CohomologicalAmplitudesForGActingOnC}.  
\end{proof}

\subsection{Other Categorical Preliminaries}
\subsubsection{A Base Change Lemma}\label{A Base Change Lemma Section}
As in \cite{GaRoI}, we let $\text{1-Cat}^{\text{St,cocmpl}}_{\text{cont}}$ denote the category of presentable stable $\infty$-categories whose functors commute with all colimits. Assume $\Aone$ is a commutative algebra object in $\text{1-Cat}^{\text{St,cocmpl}}_{\text{cont}}$. Assume $\algobj$ is an associative algebra object of $\Aone$ and $\mathcal{L}$ is a commutative algebra object of $\Aone$. We have a symmetric monoidal functor $\text{ind}_{\mathcal{L}}: \Aone \to \Atwo$ given by the left adjoint to the forgetful functor \cite[Corollary 4.2.4.8]{LuHA}, where the symmetric monoidal structure on $\Atwo$ is given by \cite[Theorem 4.5.2.1]{LuHA}, and the symmetric monoidality of the functor is given by \cite[Theorem 4.5.3.1]{LuHA}. In particular, $\newalgobj := \text{ind}_{\mathcal{L}}(\algobj)$ is an associative algebra object of $\Atwo$, and we obtain an induced functor $\widetilde{\text{ind}}_{\mathcal{L}}: \algobj\text{-mod}(\Aone) \to \newalgobj\text{-mod}(\Atwo)$. 

\begin{Proposition}\label{Categorical Extension of Scalars Proposition}
There is an equivalence of categories making the following diagram canonically commute:
    \begin{equation*}
  \xymatrix@R+2em@C+2em{
\algobj\text{-mod}(\Aone) \ar[d]_{\widetilde{\text{ind}}_{\mathcal{L}}}  \ar[dr]^{\newalgobj \otimes_{\algobj}(-)}  \\
\newalgobj\text{-mod}(\mathcal{L}\text{-mod}(\Aone)) & \newalgobj\text{-mod}(\Aone) \ar[l]_{\sim}
  }
 \end{equation*}
 
\noindent where the bottom arrow is the left adjoint to the forgetful functor induced by $\mathcal{L}\text{-mod}(\Aone) \xrightarrow{\text{oblv}} \Aone$.
\end{Proposition}

\begin{proof}
The arrows are all the left adjoints to the respective forgetful functors, so the diagram commutes. The fact that the bottom functor is an equivalence is a direct consequence of \cite[Proposition 8.5.4]{GaRoI}, where, in their notation, $\mathcal{A}_1, \mathcal{A}_2, \mathbf{M}_1, \mathbf{M}_2$, and $\mathbf{A}$ correspond to our $\algobj, \mathcal{L}\text{-mod}, \Aone, \mathcal{L}\text{-mod}$, and $\mathbf{A}$, respectively. 
\end{proof}
\newcommand{\indkL}{\text{ind}_k^L}
Let $L/k$ be an extension of (classical) fields. Then we can set $\Aone := \DGCatContk$ and $\mathcal{L} := L\text{-mod}$, so that $\mathcal{L}\text{-mod}(\DGCatContk)$ can be identified with the associated category of DG categories over $L$, $\DGCatContL$. Let $\indkL: \DGCatContk \to \DGCatContL$ denote the functor $\text{ind}_{\mathcal{L}}$ as above. 

\begin{Lemma}\label{D-modules Base Change wrt Extension of Scalars}
Let $X$ be any $k$-scheme, and let $X_L := X \times_{\text{Spec}(k)} \text{Spec}(L)$ denote the base change. 
\begin{enumerate}
    \item There are canonical equivalences $\indkL(\IndCoh(X)) \simeq \IndCoh(X_L)$ and $\indkL(\D(X)) \simeq \D(X_L)$.
    \item If $X$ is equipped with an action of an affine algebraic group $H$ and $\chi: H \to \mathbb{G}_a$ is any character, then $\indkL(\D(X)^{H,w}) \simeq \D(X_L)^{H_L, w}$ and $\indkL(\D(X)^{H, \chi}) \simeq \D(X_L)^{H_L, \chi_L}$, where $\chi_L := \chi \times_{\text{Spec}(k)} \text{Spec}(L)$. 
\end{enumerate}
\end{Lemma}

\newcommand{\Symkt}{\text{Sym}_k(\LT)}
\newcommand{\SymLt}{\text{Sym}_L(\LT)}
\newcommand{\Symktmod}{\text{Sym}_k(\LT)\text{-mod}}
\newcommand{\SymLtmod}{\text{Sym}_L(\LT)\text{-mod}}
\newcommand{\Vectk}{\text{Vect}_k}
\newcommand{\VectL}{\text{Vect}_L}

\begin{proof}
Both claims of (1) are general properties of $\IndCoh$. For example, we can show this via the fact that $\IndCoh$ on any ind-inf-scheme (such as $X$ or $X_{dR}$) is dualizable (in fact, it is self-dual by \cite[Chapter 3, Section 6.2.3]{GaRoII}) and therefore the tensor product of the category of ind-coherent sheaves on ind-inf-schemes is IndCoh of the product of the ind-inf-schemes \cite[Corollary 10.3.6]{GaiIndCoh}. 

We show the first claim of (2); the second follows by an identical argument. We have that $\D(X)^{H, w} \xleftarrow{\sim} \D(X)_{H, w}$ by \cref{Inv=Coinv}. Write $\D(X)_{H, w}$ as a colimit of maps of objects of the form $\IndCoh(X_{dR}) \otimes \IndCoh(H)^{\otimes i}$. Since $\indkL$ is a left adjoint and therefore commutes with colimits, so we see that $\indkL(\D(X)^{H,w})$ is a colimit of objects of the form $\indkL(\IndCoh(X_{dR}) \otimes \IndCoh(H)^{\otimes i})$. Using the symmetric monoidality of $\indkL$ above and \cref{D-modules Base Change wrt Extension of Scalars}(1), we see that this colimit diagram is precisely the diagram whose colimit is $\D(X_L)_{H_L, w}$. Applying \cref{Inv=Coinv} once again, we obtain our desired claim. 
\end{proof}

\begin{Corollary}
For any map $\Symkt \to L$, there is a canonical equivalence
\raggedbottom
\[\D(X)^{T, w} \otimes_{\Symkt\text{-mod}} \text{Vect}_L \simeq \D(X_L)^{T_L, w} \otimes_{\SymLt\text{-mod}} \text{Vect}_L\]

\noindent of objects in $\DGCatContL$.
\end{Corollary}

\begin{proof}
Note that
\raggedbottom
\[\D(X)^{T,w} \otimes_{\Symktmod} \VectL \simeq \text{ind}^{\SymLt}_{\Symkt}(\D(X)^{T,w}) \otimes_{\SymLtmod} \VectL\]

\noindent and so, by applying \cref{Categorical Extension of Scalars Proposition} and then applying \cref{D-modules Base Change wrt Extension of Scalars} we see that the above category is equivalent to
\raggedbottom
\[\indkL(\D(X)^{T,w}) \otimes_{\SymLtmod} \VectL \simeq \D(X_L)^{T_L, w} \otimes_{\SymLtmod} \VectL\]

\noindent of objects of $\DGCatContL$. 
\end{proof}

\subsubsection{Categorical Lemmas from Adjunctions}
The following lemma will later be used to determine the essential image of a fully faithful functor. 

\begin{Lemma}\label{FullyFaithfulAdjointProperty}
Assume we have an adjoint pair $(L, R)$ where $L: \C \to \D$. Let $u$ (respectively, $c$) denote the unit (respectively, counit) of this adjunction such that $R$ is fully faithful. Then we have:
\begin{enumerate}
    \item For any $X \in \C$, $L(u(X))$ is an isomorphism. 
    \item The essential image of $R$ is precisely those objects $C \in \C$ for which the unit map $C \to RL(C)$ is an equivalence. 
\end{enumerate}
\end{Lemma}

\begin{proof}
General properties of adjunction give that the composite
\raggedbottom
\[L(X) \xrightarrow{L(u(X))} LRL(X) \xrightarrow{c(R(X))} L(X)\]

\noindent can be identified with the identity map. However, since $R$ is fully faithful, the counit is an equivalence, and therefore so is $L(u(X))$. 

We now show (2). If the unit map is an equivalence, then $C \xrightarrow{\sim} R(L(C))$. Conversely, assume that $C \in \C$ is in the essential image. Then $C$ is isomorphic to $R(D)$ for some $D \in \D$. We then obtain the composite
\raggedbottom
\[R(D) \xrightarrow{u(R(D))} RLR(D) \xrightarrow{R(c(D))} R(D)\]

\noindent can be identified with the identity, and the rightmost arrow is an isomorphism since $R$ is fully faithful. Therefore, $u(R(D))$ is an isomorphism, and thus so too is the map $u(C)$. 
\end{proof}

\subsubsection{A Lemma on Compact Objects}
We will repeatedly use the following lemma on compact objects of a subcategory of a compactly generated category, so we recall it here:

\begin{Lemma}\cite[Lemma 2.2]{NeemanTheConnectionBetweenKTheoryLocalizationTheoremAndSmashingSubcategories}\label{Subcategory Generated By Compact Objects Implies Inclusion Preserves Compact Objects}\label{Subcategory Generated by Compacts of Big Category Has Continuous Right Adjoint}
Assume $\D$ is a compactly generated DG category and $R$ is a subset of compact objects of $\D^{c}$ closed under cohomological shifts (i.e. suspensions). Let $\mathcal{R}$ denote the full subcategory of $\D$ generated under colimits by $R$. Then the inclusion functor $\mathcal{R} \xhookrightarrow{} \D$ preserves compact objects and in particular admits a continuous right adjoint. 
\end{Lemma}

\subsubsection{A Lemma on AB5 Categories}
Recall that an abelian category $\mathcal{A}$ is said to be \textit{AB5} if it is closed under colimits and filtered colimits of short exact sequences exist and are exact. We now present a lemma which is likely well-known, but we were unable to locate a reference for. We thank Rok Gregoric for explaining its proof to us:

\begin{Lemma}\label{Colim in AB5 is Colim of Image}
Fix an AB5 abelian category $\mathcal{A}$ and some filtered colimit $\F := \text{colim}_i \F_i$ in $\mathcal{A}$. Let $\phi_i: \F_i \to \F$ denote the structure maps, and let $\tilde{\F}_i$ denote the image of $\phi_i$ for each $i$. Then the canonical filtered colimit $\text{colim}_i\tilde{\F}_i$ is equivalent to $\F$, i.e. $\F$ may be written as an increasing union of the images of the $\phi_i$. 
\end{Lemma}

\begin{proof}
We may realize the image $\tilde{\F}_i$ as the limit of the two arrows $\F \rightrightarrows \F \coprod_{\F_i} \F$ which are given by the universal property of the coproduct. We then see that 
\raggedbottom
\[\text{colim}_i \tilde{\F}_i \cong \text{colim}_i\text{lim}(\F \rightrightarrows \F \coprod_{\F_i} \F) \] \[\cong \text{lim}(\text{colim}_i\F \rightrightarrows \text{colim}_i(\F \coprod_{\F_i} \F)) \cong \text{lim}(\F \rightrightarrows \F \coprod_{\text{colim}_i(\F_i)} \F)\]

\noindent and since $\text{colim}_i\F_i \cong \F$ by assumption, we see that this expression is equivalently given by the image of the identity, i.e. $\F$. Here, the commutation of limits and filtered colimits is given by the axiom of AB5 categories because, by assumption, the colimit functor is left exact, and left exact functors preserve all finite limits. 
\end{proof}

\subsubsection{t-Structures on Quotient Categories}\label{Intro to t-Structures on Quotient Categories}
\newcommand{\Ccirc}{\mathring{\C}}
\newcommand{\tildeR}{\tilde{R}}
\newcommand{\tildeL}{\tilde{L}}

We set the following notation for this subsection: Assume $I_*: \C_0 \xhookrightarrow{} \C$ is a fully faithful embedding in $\DGCatContk$ which admits a (continuous) right adjoint $I^!$. Furthermore, assume $\C$ is equipped with an accessible $t$-structure (i.e. the $\infty$-category $\C^{\geq 0}$ is accessible) which is compatible with filtered colimits, i.e. $\C^{\geq 0}$ is closed under filtered colimits. Furthermore assume the essential image of $I_*$ is closed under truncation functors (and so, in particular, $\C_0$ admits a unique $t$-structure so that $I_*$ is $t$-exact) and the essential image of $I_*^{\heartsuit}: \C_0^{\heartsuit} \xhookrightarrow{} \C^{\heartsuit}$ is closed under subobjects. Let $\Ccirc := \text{ker}(\C \xrightarrow{I^!} \C_0)$. Note that the inclusion functor $J_*: \Ccirc \xhookrightarrow{} \C$ admits a (necessarily continuous) left adjoint $J^!$ because $I^!$ preserves small limits.

By definition, the functor $J^!$ is a \text{localization}:

\begin{Definition}\label{LocalizationDefinition}
A functor with a fully faithful right adjoint is called a \textit{localization}.  
\end{Definition}
    
When clear from context, we will refer to $J^!$ as \textit{the quotient functor}.

\begin{Example}\label{LocalizationExample}
If $j: U \xhookrightarrow{} X$ is an open embedding of algebraic varieties, the functor $j^!: \D(X) \to \D(U)$ is a localization. 
\end{Example}

We will now recall the following general result of \cite{RaskinAffineBBLocalization}.

\begin{Proposition}\cite[Section 10.2]{RaskinAffineBBLocalization}\label{Raskin t-ExactCatProp}
 There is a unique $t$-structure on $\Ccirc$ such that $J^!$ is $t$-exact, uniquely determined by setting $\Ccirc^{\geq 0}$ as the full subcategory of objects $\F \in \Ccirc$ for which $J_*(\F) \in \C^{\geq 0}$. 
\end{Proposition}

Note that, in particular, this $t$-structure is accessible, since the category $\Ccirc^{\geq 0}$ is the fiber product of accessible categories $\Ccirc \times_{\C} \C^{\geq 0}$ and therefore is accessible by \cite[Section 5.4.6]{LuHTT}. 

\begin{Corollary}\label{Quotient Category Identification on Eventually Coconnective Subcategories}
The inclusion functor $J_*$ identifies $\Ccirc^+$ with the kernel of $I^!|_{\mathcal{C}^+}$, and the inclusion functor $I_*$ identifies $\C_0^+$ with the kernel of $J^!|_{\C^+}$.
\end{Corollary}

\begin{proof}
Because $I_*$ is $t$-exact, its right adjoint $I^!$ is left $t$-exact and thus preserves the eventually coconnective subcategory. In particular, if $\F \in \C^+$ then $J_*(\F) \in \Ccirc^+$, and $I^!(J_*(\F)) \simeq 0$. Conversely if $X \in \C^+$ lies in the kernel of $I^!$ then $J^!(X) \in \Ccirc^+$ by the $t$-exactness of $J^!$. Similarly if $Z \in \C_0^+$ then $I_*(Z) \in \C^+$ by the $t$-exactness of $I_*$ and if $X \in \C^+$ lies in the kernel of $J^!$ then $I^!(X) \in \C_0^{+}$ since $I^!$ is left $t$-exact, being the right adjoint to a $t$-exact functor $I_*$. 
\end{proof}

\begin{Lemma}\label{t-Structure on ccirc is compatible with filtered colimits}
The $t$-structure on $\Ccirc$ is compatible with filtered colimits.
\end{Lemma}

\begin{proof}
Assume we are given a filtered diagram $I$ in $\Ccirc^{\geq 0}$. Let $X := \text{colim}_{i \in I}X_i$ denote the colimit of this diagram. Then we see that the continuity of $J_*$ gives $J_*(X) \simeq \text{colim}_i(J_*(X_i))$, and therefore since the $t$-structure on $\C$ is compatible with filtered colimits by assumption, we see $\text{colim}_i(J_*(X_i)) \in \C^{\geq 0}$ and so, by \cref{Raskin t-ExactCatProp} we have $X \in \Ccirc^{\geq 0}$, and so the $t$-structure is compatible with filtered colimits by definition. 
%
\end{proof}

\begin{Corollary}\label{t-Structure on Quotient of Right-Complete Category is Right-Complete}
If the $t$-structure on $\C$ is right-complete, then the $t$-structure on $\Ccirc$ is right-complete.
\end{Corollary}

\begin{proof}
Let $\F \in \Ccirc$, and let $\phi^{\F}: \F \to \text{colim}(\tau^{\leq n}\F)$ denote the canonical map. Using the fact that the $t$-structure on $\Ccirc$ is compatible with filtered colimits, we see that for all $n \in \mathbb{Z}$ the map $\tau^{\leq n}(\phi^{\F})$ is an equivalence. Therefore, if we denote by $\mathcal{K}$ the fiber of $\phi^{\F}$, we see $\tau^{\leq n}(\mathcal{K}) \simeq 0$ for all $n$. Thus $J_*(\mathcal{K}) \in \cap_n C^{\geq n} \simeq 0$, so $\phi^{\F}$ is an equivalence, as required. 
\end{proof}

Although $J_*$ is not right $t$-exact in general, we do have the following statement:

\begin{Lemma}\label{Any leq 0 of quotient category is quotient of leq 0 object}
If $X \in \Ccirc^{\leq 0}$, then there is a canonical equivalence $X \simeq J^!(\tau^{\leq 0}J_*(X))$. 
\end{Lemma}

\begin{proof}
For such an $X$, we obtain a cofiber sequence $\tau^{\leq 0}J_*X \to J_*X \to \tau^{> 0}J_*X$. Applying $J^!$ to this cofiber sequence (which is continuous and thus exact by construction), we obtain the cofiber sequence 
\raggedbottom
\begin{equation}\label{Cofiber Sequence for Object of Quotient Category}J^!(\tau^{\leq 0}J_*X) \to J^!J_*X \to J^!(\tau^{> 0}X)\end{equation}

\noindent and furthermore we see that the unit map $X \xrightarrow{\sim} J_*J^!(X)$ identifies the middle term of \labelcref{Cofiber Sequence for Object of Quotient Category} with $X$. The left and middle terms of  \labelcref{Cofiber Sequence for Object of Quotient Category} thus both lie in $\Ccirc^{\leq 0}$ while the rightmost term lies in $\Ccirc^{> 0}$. Thus the rightmost term in  \labelcref{Cofiber Sequence for Object of Quotient Category} vanishes and so the composite 
\raggedbottom
\[J^!(\tau^{\leq 0}J_*X) \xrightarrow{\sim} J^!J_*X \xleftarrow{\sim} X\]

\noindent gives our desired equivalence. 
\end{proof}

Finally, we record the interaction with the above discussion when our underlying category DGCat is replaced with $G$-categories:

\begin{Proposition}\label{Quotient of G Category with Compatible t-Structure Action by G Subcategory Has G Action Compatible with t-Structure}
Assume $G$ acts on $\C$ compatibly with the $t$-structure and the subcategory $\C_0$ is closed under the $G$-action. Then $\Ccirc$ acquires a $G$-action and moreover $G$ acts compatibly with the $t$-structure. 
\end{Proposition}

\begin{proof}
First note that the fact that $I_*$ is a functor of $G$-categories gives that $I^!$ is a functor of $G$-categories by \cref{Adjoint Functor is Automatically G Equivariant}. Therefore we see that $\Ccirc := \text{ker}(I^!)$ acquires a $G$-action. We have that $J_*$ is $G$-equivariant and so $J^!$ is as well by \cref{Adjoint Functor is Automatically G Equivariant}. Let $c, c_0$ (respectively) denote the coaction map of $G$ on $\C, \Ccirc$ (respectively) shifted by $[-\text{dim}(G)]$ so that $c$ is by assumption $t$-exact. 

We now wish to show $c_0$ is $t$-exact if $c$ is. To see this, first note if $X \in \Ccirc^{\geq 0}$ then $(\text{id}_G \otimes J_*)c_0(X) \simeq c(J_*(X))$ by the $G$-equivariance of $J_*$ and so $c_0(X) \in (\D(G) \otimes \Ccirc)^{\geq 0}$ by \cref{GaitsgoryExactnessLemma}. Now assume $X \in \Ccirc^{\leq 0}$. By \cref{Any leq 0 of quotient category is quotient of leq 0 object}, we may write $X \simeq J^!(\tau^{\leq 0}J_*(X))$. Since $c$ is $t$-exact and $G$-equivariant, we can use this description of $X$ to obtain equivalences
\raggedbottom
\[c_0(X) \simeq (\text{id}_{\D(G)} \otimes J^!)c_0\tau^{\leq 0}J_*(X)\]\[ \simeq (\text{id}_{\D(G)} \otimes J^!)\tau^{\leq 0}c_0J_*(X) \simeq (\text{id}_{\D(G)} \otimes J^!)\tau^{\leq 0}(\text{id}_{\D(G)} \otimes J_*)(c_0(X))\] 

\noindent so that in particular $c_0(X) \simeq (\text{id}_{\D(G)} \otimes J^!)(Y)$ for some $Y \in (\D(G) \otimes \C)^{\leq 0}$. Therefore, $c_0(X) \in (\D(G) \otimes \Ccirc)^{\leq 0}$ by the $t$-exactness of $\text{id}_{\D(G)} \otimes J^!$ given by \cref{GaitsgoryExactnessLemma}. 
\end{proof}

\subsubsection{t-Structures on Quotient Categories Given by the Eventually Coconnective Kernel}
Assume $\C$ is a DG category with a $t$-structure compatible with filtered colimits, and let $L: \C \to \D$ be a $t$-exact functor of stable $\infty$-categories equipped with $t$-structures. Set $\C_0$ to denote the full subcategory generated under filtered colimits by eventually coconnective objects in the kernel of $L$, i.e. the full subcategory generated by objects $X \in \C^+$ for which $L(X) \simeq 0$. By $t$-exactness of $L$, $\C_0$ is closed under the truncation functors $\tau^{\leq 0}$ and $\tau^{\geq 0}$, where the latter also uses that the $t$-structure on $\C$ is compatible with filtered colimits. The $t$-exactness of $L$ also gives that the essential image of $\C_0^{\heartsuit} \xhookrightarrow{} \C^{\heartsuit}$ is closed under subobjects. We equip the full subcategory $\Ccirc \xhookrightarrow{J_*} \C$ with the $t$-structure so that the quotient functor $J^!$ is $t$-exact as in \cref{Raskin t-ExactCatProp}. Let $\tildeL := L\circ J_*$.  

\begin{Proposition}
The functor $\tildeL$ is $t$-exact.
\end{Proposition}

\begin{proof}
By definition, $\tildeL$ is the composite of two left $t$-exact functors, and therefore is left $t$-exact. Now assume that $X \in \Ccirc^{< 0}$. Then, because $X \xrightarrow{\sim} J^!J_*(X)$, we have that $\tau^{\geq 0}(J_*(X))$ lies in the kernel of $J^!$, since $J^!$ is $t$-exact, and thus $\tau^{\geq 0}(J_*(X)) \in \C_0$. By definition of $\C_0$, we see that $L(\tau^{\geq 0}(J_*(X))) \simeq 0$. Therefore we see that $\tau^{\geq 0}\tildeL(X) \simeq \tau^{\geq 0}LJ_*(X) \simeq L(\tau^{\geq 0}J_*(X)) \simeq 0$, so $\tildeL$ is also right $t$-exact.
\end{proof}

Now, with the above notation, further assume that $\D$ is a DG category and that $L$ is a functor in $\DGCatContk$ which admits a (continuous) right adjoint $R: \D \to \C$. Because $I^!R$ is a right adjoint, one can easily verify that $I^!R \simeq 0$, and so the canonical map $R \to J_*J^!R$ is an equivalence. Let $\tildeR$ denote the functor $J^!R$. 

\begin{Proposition}\label{Left Adjoint of Quotient Functor}
In the above notation, $\tildeR$ is right adjoint to $\tildeL$.
\end{Proposition}

\begin{proof}
We have equivalences, for all $X \in \Ccirc$ and $Y \in \D$:
\raggedbottom
\[\Hom_{\D}(\tildeL(X), Y) \simeq \Hom_{\D}(\tildeL J^!J_*(X), Y) \simeq \Hom_{\D}(LJ_*(X), Y) \simeq \Hom_{\C}(J_*(X), R(Y))\]

\noindent via the fully faithfulness of $J_*$, the definition of $\tildeL$, and the adjoint property respectively. Since $R \xrightarrow{\sim} J_*J^!R$, the above expression is equivalent to:
\raggedbottom
\[\Hom_{\C}(J_*(X), J_*J^!R(Y)) \simeq \Hom_{\Ccirc}(X, J^!R(Y)) \simeq \Hom_{\Ccirc}(X, \tildeR(Y))\]

\noindent by the fully faithfulness of $J_*$ and the definition of $\tildeR$, respectively, which completes our proof.
\end{proof}

\newcommand{\LTdcompleted}{\LT^{\ast, \wedge}_{\lambda}}
\newcommand{\LTdmodquotientcompleted}{(\LTd/\characterlatticeforT)^{\wedge}_{[\lambda]}}
\newcommand{\quotientmapforcoarsequotient}{\overline{s}}
\newcommand{\quotientmaptostackquotient}{q}
\newcommand{\inclusionofcoinvariantalgebraintovectorspace}{\overline{x}'}
\newcommand{\inclusionofcoinvariantalgebraintovectorspaceINDCOHpullback}{\overline{x}'^{!}}
\newcommand{\inclusionofcoinvariantalgebraintovectorspaceQCOHpullback}{\overline{x}'^{*}}
\newcommand{\terminalmapfromC}{\alpha}
\newcommand{\terminalmapfromCmodassociatedstabilizer}{\dot{\alpha}}
\newcommand{\quotientmapforFINITEcoarsequotient}{\overline{s}_{\text{fin}}}
\newcommand{\mapfromstackquotientofWexttocoarsequotientofWext}{\tilde{\phi}}
\newcommand{\mapfromstackquotientofWexttocoarsequotientofWextATAPOINT}{\dot{\tilde{\phi}}}
\renewcommand{\G}{\mathbb{G}}
\section{Nondegenerate Categories and the Weyl Group Action}\label{Nondegenerate Categories and the Weyl Group Action Section}
In this section, we introduce the notion of a nondegenerate $G$-category and show it admits a Weyl group action. 
\newcommand{\coactbar}{\overline{\text{coact}}}
\newcommand{\DNdeg}{\mathcal{D}(N \backslash G/N)_{\text{deg}}}
\newcommand{\Ddeg}{\D(N \backslash G)_{\text{deg}}}
\raggedbottom
\subsection{Degenerate $G$-Categories}\label{Degenerate G-Categories Subsection}
\subsubsection{Definition of the Universal Degenerate Category} We make the following definitions: 
\begin{Definition}\label{Deg subcat of DNmodG and DnmodGmodN Definitions}With respect to the generic character $\psi$ we have fixed above:
\begin{enumerate}
    \item The \textit{degenerate subcategory} $\Ddeg$ of $\D(N\backslash G)$ is the full subcategory of $\D(N\backslash G)$ generated by eventually coconnective objects in the kernel of $\Avpsi: \D(N \backslash G) \to \D(N^{-}_{\psi}\backslash G)$. 
   \item The \textit{degenerate subcategory} $\DNdeg$ of $\D(N\backslash G/N)$ is the full subcategory of $\D(N\backslash G/N)$ generated by eventually coconnective objects in the kernel of the composite \[\D(N\backslash G/N) \xrightarrow{\sim} \D(G/N) \xrightarrow{\Avpsi} \D(G/N)^{N^-, \psi}\] of the pullback map and the averaging functor $\Avpsi$. 
     \end{enumerate}
\end{Definition}

\begin{Remark}
    We will see in \cref{Simply Connected Version of Nondegenerate Subcategory of D(G/N) is the HN-subcategory Gen'd By Non-antidominant Simples} that the degenerate subcategory does not, in fact, depend on the choice of generic character $\psi$.
\end{Remark}
The categories of \cref{Deg subcat of DNmodG and DnmodGmodN Definitions} are related as follows:
We now justify the use of the subscript \lq deg\rq{} in the notation \lq$\DNdeg$\rq{} as opposed to a subscript such as \lq left-deg\rq{}.

\begin{Proposition}\label{Deg is BiDeg}
The following categories are equivalent:
\begin{enumerate}
    \item The full subcategory of $\D(N\backslash G/N)$ generated by eventually coconnective objects in the kernel of the left Whittaker averaging functor $\Avpsi: \D(N\backslash G/N) \to \D(N^-_{\psi}\backslash G/N)$.
    \item The full subcategory of $\D(N\backslash G/N)$ generated by eventually coconnective objects in the kernel of the right Whittaker averaging functor $\text{Av}_!^{\psi}: \D(N\backslash G/N) \to \D(N\backslash G/_{\psi}N^-)$.
    \item The full subcategory of $\D(N\backslash G/N)$ generated by eventually coconnective objects in the kernel of the bi-Whittaker averaging functor $\text{Av}_!^{\psi \times \psi}: \D(N\backslash G/N) \to \Hpsi$.
\end{enumerate}
\end{Proposition}

We will prove this after showing the following lemma:

\begin{Lemma}\label{Mellin Dual of pistar is conservative on t}
The functor $\D(T) \xrightarrow{\sim} \D(N\backslash G/_{\psi}N^-) \xrightarrow{\Avpsi} \D(N^-_{\psi}\backslash G/_{\psi}N^-)$, where the first arrow is given by \cref{SupportOfWhittakerSheaves}, is conservative.
\end{Lemma}

\begin{proof}
We may equivalently show $\Avpsishifted$ is conservative. Because both functors $\AvNshifted$ appearing below average with respect to the right $N$-action and both functors $\Avpsishifted$ appearing below average with respect to the left action, the following diagram commutes
\raggedbottom
\begin{equation*}
  \xymatrix@R+2em@C+2em{
   \D(N^-_{\psi}\backslash G/N)  & \D(N \backslash G/N) \ar[l]_{\text{Av}_!^{\psi}[-\text{dim}(N)]} \\
   \Hpsi \ar[u]_{\AvNshifted}  & \D(N\backslash G/_{\psi}N^-) \ar[l]^{\Avpsishifted} \ar[u]_{\AvNshifted}
  }
 \end{equation*}
 
 \noindent and so to verify conservativity of $\Avpsishifted$ we may verify the composite given by $\Avpsishifted\AvNshifted$ as in the diagram is conservative. Both functors in the diagram are $t$-exact by \cref{BBMShiftedLeftAdjointIsExact} and \cref{Ginzburgt-Exactness}. Let $i: T \cong N \backslash B \xhookrightarrow{} N \backslash G$ denote the closed embedding. Since the category $\D(T)$ is the derived category of its heart \cite[Section 4.7]{GaiRozCrystals}, we may equivalently show that the composite $\Avpsishifted\AvNshifted$ on objects of the form $m_{*, dR}(\F \boxtimes \psi)$, where $m: N \backslash B \times N^- \xhookrightarrow{} G$ is the open embedding induced by multiplication. 
 
By the $t$-exactness of $\Avpsishifted$ given by \cref{BBMShiftedLeftAdjointIsExact}, we may verify conservativity for objects $\F \in \D(T)^{\heartsuit}$. If $\F \in \D(T)^{\heartsuit}$ then $\AvNshifted(m_{*, dR}(\F \boxtimes \psi))$ contains $i_{*, dR}(\F) \in \D(N\backslash G/N)^{\heartsuit}$ as a subobject, since the composite $\D(T) \simeq \D(N\backslash G/_{\psi}N^-) \xrightarrow{\AvN} \D(N \backslash G/N) \xrightarrow{i^!} \D(T)$ is the identity. We similarly see that $\Avpsishifted$ is the identity when restricted to $\D(T) \xrightarrow{i_{*, dR}} \D(N \backslash G/N)$ by the uniqueness of left adjoints. Therefore, we see that $\Avpsishifted\AvNshifted(\F)$ contains $i^!(\F) \simeq \F$ as a subobject, and therefore if $\F$ is nonzero, so too is $\AvNshifted\Avpsishifted(\F)$. Therefore $\Avpsishifted(\F) \simeq 0$. 
\end{proof}

\begin{proof}[Proof of \cref{Deg is BiDeg}]
We show the categories generated as in (1) and (3) are equivalent; a symmetric argument gives that the categories in (2) and (3) are equivalent. Since bi-Whittaker averaging can be realized as a composite of right and left Whittaker averaging, we see that the subcategory as in (1) is contained in the subcategory of (3). This factorization, along with the conservativity of the averaging functor $\text{Av}^{\psi}_!: \D(N^-_{\psi}\backslash G/N) \to \Hpsi$ in \cref{Mellin Dual of pistar is conservative on t}, gives that the category in (3) is contained in the category in (1).
\end{proof}

We now record some corollaries of the above results.

\begin{Corollary}\label{No Degenerate Whittaker Objects in D(G/N)}
We have an equivalence of categories $(\D(N\backslash G)_{\text{deg}})^{N^-, \psi} \simeq 0$.
\end{Corollary}

\begin{proof}
If $\F \in (\D(N\backslash G)_{\mathrm{deg}})^{N^-, \psi}$ then it can be written as a colimit of eventually coconnective objects in the kernel of the left Whittaker averaging functor $\Avpsi$. Since $\Avpsi$ commutes with colimits, we see that $\F \in \D(N\backslash G/_{\psi}N)^-$ lies in the kernel of $\Avpsi$. Therefore, by \cref{Mellin Dual of pistar is conservative on t}, $\F \simeq 0$. in particular lies in the kernel of $\Avpsi$. 
\end{proof}
\begin{Corollary}\label{Other Action Theorem}
The category $\DNdeg$ is closed under the left action of $\HN$. 
\end{Corollary}

\begin{proof}
The category $\HN$ is compactly generated, and therefore it suffices to show the action of the compact objects $\F \in \HN$ preserves the category $\DNdeg$. Since the $t$-structure on $\D(N \backslash G/N)$ is right-complete (see \cref{Left and right-complete t-structure Definition}), $\F$ is eventually connective. Moreover, the convolution action of $\HN$ on itself preserves the eventually connective subcategory, since the pullback by a smooth map is $t$-exact up to shift and the pushforward map has finite cohomological amplitude, see \cite[Proposition 1.5.29]{HTT}. 

Let $\mathcal{G}$ be a generator of $\DNdeg$, i.e. an eventually coconnective object in the kernel of left Whittaker averaging. Since each averaging functor is continuous, by \cref{Deg is BiDeg} we see that $\mathcal{G}$ is also in the kernel of the \textit{right} averaging functor $\text{Av}_!^{\psi}: \D(N \backslash G/N) \to \D(N \backslash G/_{\psi}N^-)$. Therefore since the right averaging functor is left $\HN$-equivariant, we see that $\F \star^N \mathcal{G}$ is an eventually coconnective object in the kernel of the right averaging functor, and therefore in $\DNdeg$. 
\end{proof}

\begin{Corollary}\label{Definition of D(G/N)deg}
Let $\D(N\backslash G)_{\text{deg}}$ denote the full subcategory generated by eventually coconnective objects in the kernel of the left Whittaker averaging functor $\Avpsi: \D(N\backslash G) \to \D(N^-_{\psi}\backslash G)$. Then $\D(N\backslash G)_{\text{deg}}$ is closed under the left action of $\HN$.
\end{Corollary}

\begin{proof}
Let $\D$ denote the full (left) $\HN$-subcategory of $\D(N\backslash G)$ generated by eventually coconnective objects in the kernel of left Whittaker averaging. This naturally acquires a right $G$-action, and we have a natural inclusion functor of right $G$-categories $\D(N \backslash G)_{\text{deg}} \xhookrightarrow{} \D$. Moreover, by \cref{BZGO}, we may check that this functor is an equivalence after applying the right invariants $(-)^{N}$. Therefore this result follows directly from \cref{Other Action Theorem}. 
\end{proof}

\renewcommand{\G}{\mathbb{G}}

\subsubsection{Definition of Degenerate $G$-Categories in General}
\begin{Definition} Let $\C$ denote a $G$-category. We say $\C$ is \textit{degenerate} if the canonical map $\D(N \backslash G)_{\text{deg}} \otimes_{G} \C \to \C^N$ is an equivalence.
\end{Definition}

\begin{Proposition}\label{Equivalence of Various Notions of Degeneracy}
A $G$-category $\C$ is degenerate if and only if for any Borel $B'$ with $N' := [B', B']$, the canonical map $\C^{N'} \xleftarrow{} \D(N' \backslash G)_{\text{deg}} \otimes_{G} \C$ is an equivalence.
\end{Proposition}

\begin{proof}
The \lq if\rq{} direction follows from taking $B' := B$. For the other direction, fix some $B', N'$ as in the statement of \cref{Equivalence of Various Notions of Degeneracy}. There exists some $g \in G(k)$ such that $gBg^{-1} = B'$ so that $gNg^{-1} = N'$. Furthermore, $g$ determines a character $\psi'$, defined as the composite $N^{'-}/[N^{'-}, N^{'-}] \xrightarrow{\text{Ad}_{g^{-1}}} N^-/[N^-, N^-] \to \G_a$, which is nondegenerate. Therefore, the following diagram commutes:
\raggedbottom
\begin{equation*}
  \xymatrix@R+2em@C+2em{
   \D(N^-_{\psi}\backslash G) \ar[d]^{\ell_g} & \D(N\backslash G) \ar[d]^{\ell_g}  \ar[l]_{\Avpsi}\\
   \D(N^{'-}_{\psi^{'}}\backslash G)   & \D(N'\backslash G)  \ar[l]_{\text{Av}_!^{\psi{'}}}
  }
 \end{equation*}
 
 \noindent where the vertical arrows are given by left multiplication by $g$, which is manifestly right $G$-equivariant. The fact that the action of $G$ on itself yields an action on $\D(G)$ compatible with the $t$-structure implies that $\text{Ad}_g$ induces an isomorphism $\D(N\backslash G)_{\text{deg}} \xrightarrow{\sim} \D(N'\backslash G)_{\text{deg}}$. Furthermore, the following diagram commutes
 \raggedbottom
 \begin{equation*}
  \xymatrix@R+2em@C+2em{
  \C^N \ar[d]^{:= } & \D(N\backslash G) \otimes_G \C \ar[d]^{\text{Ad}_g \otimes \text{id}}  \ar[l]_{\sim} & \ar[l]_{I_* \otimes \text{id}} \D(N\backslash G)_{\text{deg}} \otimes_G \C \ar[d]^{\text{Ad}_g \otimes \text{id}}\\
  \C^{N'}  & \D(N'\backslash G) \otimes_G \C \ar[l]_{\sim} & \ar[l]_{I'_* \otimes \text{id}} \D(N'\backslash G)_{\text{deg}} \otimes_G \C
  }
 \end{equation*}
 
\noindent where the maps $I_*$ and $I'_*$ are the inclusions and the leftmost horizontal arrows are the canonical maps, which are equivalences by \cref{Inv=Coinv}. We have that all three vertical maps are equivalences, and, assuming that $\C$ is degenerate, all maps in the top row of the diagram are equivalences. Therefore, all maps in the bottom row of the diagram are equivalences.
\end{proof}

\subsection{Nondegenerate $G$-Categories}

\subsubsection{Definition and Basic Properties}
Let $I_*: \D(G/N)_{\text{deg}} \xhookrightarrow{} \D(G/N)$ denote the inclusion functor. We first note:

\begin{Proposition}\label{Inclusion of Degenerate Category of D(G/N) Is Closed Under Truncation Functors}
The essential image of $I_*$ is closed under truncation functors.
\end{Proposition}

\begin{proof}
Assume $\F$ lies in the essential image of $I_*$, i.e. $\F$ is a colimit of eventually coconnective objects in the kernel of $\Avpsi$. By definition of $t$-structures, we obtain a cofiber sequence 
\raggedbottom
\[\tau^{\leq 0}\F \to \F \to \tau^{> 0}\F\]

\noindent where we omit the inclusion functors. Since $\Avpsi$ is $t$-exact up to shift by \cref{BBMShiftedLeftAdjointIsExact}, we see that $\tau^{> 0}\F$ is an eventually coconnective object in the kernel of $\Avpsi$, and thus lies in the essential image of $I_*$. Rotation gives a cofiber sequence
\raggedbottom
\[\F \to \tau^{> 0}\F \to \tau^{\leq 0}(\F)[1]\]
\noindent and so $\tau^{\leq 0}(\F)[1]$ is a colimit of objects in the essential image of $I_*$ and thus lies in the essential image of $I_*$. Finally, since the shift functor $[-1]$ is an equivalence of categories, it commutes with all colimits and so the essential image of $I_*$ is closed under all cohomological shifts. Therefore, we obtain that $\tau^{\leq 0}(\F)$ lies in the essential image of $I_*$, as required. 
\end{proof}

By \cref{Inclusion of Degenerate Category of D(G/N) Is Closed Under Truncation Functors}, we see that the inclusion $I_*$ falls into the setup of \cref{Intro to t-Structures on Quotient Categories}. We defer the proof of the following statement to \cref{Proof of Deg Inclusion Admits a Continuous Right Adjoint}.

\begin{Proposition}\label{Degenerate Inclusion Admits Continuous Right Adjoint}
The functor $I_*$ admits a continuous, $\HN$-equivariant right adjoint $I^!: \D(G/N) \to \D(G/N)_{\text{deg}}$, and the category $\D(G/N)_{\text{deg}}$ is compactly generated.
\end{Proposition}

We set $\D(G/N)_{\text{nondeg}}$ to denote the right orthogonal category to $\D(G/N)_{\text{deg}}$. Following \cref{Intro to t-Structures on Quotient Categories}, we let $J_*$ denote the inclusion of $\D(G/N)_{\text{nondeg}}$ into $\D(G/N)$. From \cref{Degenerate Inclusion Admits Continuous Right Adjoint}, we therefore obtain: 

\begin{Proposition}
The inclusion functor of the nondegenerate category $\D(G/N)_{\text{nondeg}}$ admits a $\HN$-equivariant left adjoint $J^!: \D(G/N) \to \D(G/N)_{\text{nondeg}}$, and there is a unique $t$-structure on $\D(G/N)_{\text{nondeg}}$ for which this quotient functor is $t$-exact.
\end{Proposition}

Note that the $\HN$-equivariance follows from \cref{Continuous Right Adjoint on N Invariants Implies HN Linear}. We now make the following definition: 

\begin{Definition}\label{NondegenerateDefinition}
We say a $G$-category $\C$ is \textit{nondegenerate} if $\D(N \backslash G)_{\text{nondeg}} \otimes_{G} \C \simeq \C^N$ (or, equivalently, $\D(N \backslash G)_{\text{deg}} \otimes_{G} \C \simeq 0$). 
\end{Definition}

We now record some first properties of our category $\D(N\backslash G)_{\text{nondeg}}$.

\begin{Proposition}\label{Whittaker Invariants Are Automatically Nondegenerate}
The canonical map $\D(N\backslash G)^{N^-, \psi} \to (\D(N\backslash G)_{\text{nondeg}})^{(N^-, \psi)}$ is an equivalence. 
\end{Proposition}

\begin{proof}
The map $I^!: \D(N\backslash G) \to \D(N\backslash G)_{\text{deg}}$ is manifestly right $G$-equivariant. Therefore we can identify $(\D(N\backslash G)_{\text{nondeg}})^{(N^-, \psi)}$ with the kernel of the induced functor on Whittaker invariants $I^{!, N^{-}, \psi}: \D(N\backslash G)^{N^-, \psi} \to \D(N\backslash G)_{\text{deg}}^{N^-, \psi}$. By \cref{No Degenerate Whittaker Objects in D(G/N)}, the codomain of this functor vanishes, and so all objects are nondegenerate.  
\end{proof}

\begin{Proposition}\label{Easier Classification of Kernel Properties}The functor $\AvN: \D(N^-_{\psi}\backslash G) \to \D(N\backslash G)$ canonically factors through the nondegenerate subcategory $D(N \backslash G)_{\text{nondeg}}$, i.e. we have an equivalence $\AvN \xrightarrow{\sim} J_*J^!\AvN$. The functor $J^!\AvNshifted$ is $t$-exact.
\end{Proposition}

\begin{proof}
This follows from the fact that, for any $\mathcal{G} \in \D(N \backslash G)^+_{\text{deg}}$, we have by adjunction that $\uHom_{\D(N \backslash G)}(\mathcal{G}, \AvN(\F)) \simeq \uHom_{\D(N^-_{\psi}\backslash G)}(\Avpsi(\mathcal{G}), \F) \simeq 0$, and such $\mathcal{G}$ generate $\D(N \backslash G)_{\text{deg}}$ by \cref{Definition of D(G/N)deg}. The $t$-exactness follows from \cref{Ginzburgt-Exactness} and the fact the quotient functor is $t$-exact, see \cref{Intro to t-Structures on Quotient Categories}. 
\end{proof}

For a rank-one parabolic $P_{\alpha}$ associated to a simple root $\alpha$, let $Q_{\alpha} := [P_{\alpha}, P_{\alpha}]$. 

\begin{Proposition}\label{Avpsi Vanishes on Qalpha Monodromic Objects}
The left Whittaker averaging functor $\Avpsi$ sends objects of $\D(N\backslash G)$ which are $Q_{\alpha}$-monodromic for some simple root $\alpha$ (i.e. objects which are equivalent to objects of the form $q^!(\mathcal{F})$ for $\mathcal{F} \in \D(Q_{\alpha}\backslash G)$ where $q: N\backslash G \to Q_{\alpha}\backslash G$ the quotient map) to zero. Furthermore, $\Avpsi$ vanishes on the full $G$-subcategory of $\D(N\backslash G)$ generated under shifts and colimits by the various $Q_{\alpha}$-monodromic objects.
\end{Proposition}

\begin{proof}
Let $\alpha$ denote a simple root. Since the composite of the $G$-functors $\text{Av}_*^{Q_{\alpha}^r}\AvN$ has integral kernel in the right $Q_{\alpha}$-monodromic objects of $\D(N^-_{\psi}\backslash G/N)$, we see that its integral kernel is zero by \cref{SupportOfWhittakerSheaves}. Therefore, its left adjoint $\Avpsi\text{oblv}^{Q_{\alpha}}$ vanishes. The latter claim follows from the continuity and exactness of the left adjoint $\Avpsi$.
\end{proof}

\subsubsection{Proof of \cref{Degenerate Inclusion Admits Continuous Right Adjoint}}\label{Proof of Deg Inclusion Admits a Continuous Right Adjoint}

In this section, we prove \cref{Degenerate Inclusion Admits Continuous Right Adjoint} after showing the following proposition: 

\begin{Proposition}\label{Degenerate Category is Compactly Generated}
The following categories are equivalent: 
\begin{enumerate}
    \item The full subcategory of $\mathcal{D}(G/N)$ generated under colimits by the compact objects in the kernel of $\Avpsi$.
    \item The full subcategory of $\mathcal{D}(G/N)$ generated under colimits by cohomologically bounded objects in the kernel of $\Avpsi$.
    \item The full subcategory of $\mathcal{D}(G/N)$ generated under colimits by eventually coconnective objects in the kernel of $\Avpsi$. 
\end{enumerate}
\end{Proposition}

\begin{proof}
The compact objects in $\D(G/N)$ are in particular cohomologically bounded \cite[Section 5.1.17]{DrinfeldGaitsgoryOnSomeFinitenessQuestionsforAlgebraicStacks}, and cohomologically bounded objects are in particular eventually coconnective. Now, assume we are given an eventually coconnective $\mathcal{F}$. Since the $t$-structure on $\D(G/N)$ is right-complete  (see \cref{Left and right-complete t-structure Definition}), we may write $\mathcal{F} \simeq \text{colim}\tau^{\leq n}\F$ as the colimit of objects which have nonzero cohomology in only finitely many degrees, and therefore we assume $\mathcal{F}$ has nonzero cohomology in only finitely many degrees. Furthermore, we have that $\Avpsi$ is $t$-exact, so $\Avpsi(\F)$ vanishes if and only if $\Avpsi$ takes each cohomology group to zero. Therefore, we may assume $\F$ is concentrated in a single cohomological degree, and, since $\Avpsi$ commutes with shifts, we may assume $\F \in \D(G/N)^{\heartsuit}$. However, $\D(G/N)^{\heartsuit}$ is a Grothendieck abelian category, and furthermore it is compactly generated, see \cite[Corollary 3.3.3]{GaiRozCrystals}. Furthermore, the compact objects are closed under subquotients (see \cite[Section 4.7.7]{GaiRozCrystals}) and so, by \cref{Colim in AB5 is Colim of Image} we see that every object is a union of compact subobjects. Since $\Avpsi$ is $t$-exact, we therefore see that this object $\F$ is a colimit of compact objects in the kernel of $\Avpsi$. 
\end{proof}

\begin{proof}[Proof of \cref{Degenerate Inclusion Admits Continuous Right Adjoint}]
The compact generation follows by \cref{Degenerate Category is Compactly Generated}. By \cref{Subcategory Generated by Compacts of Big Category Has Continuous Right Adjoint} with $R$ taken to be the compact objects concentrated in a single cohomological degree, we see further that the inclusion functor of the degenerate subcategory $\D(G/N)_{\text{deg}}$ preserves compact objects and therefore admits a continuous right adjoint. Finally, the right adjoint $I^!$ to the inclusion functor $I_*: \D(G/N)_{\text{deg}} \xhookrightarrow{} \D(G/N)$ is necessarily $\HN$-linear by \cref{Continuous Right Adjoint on N Invariants Implies HN Linear}. 
\end{proof}

\subsubsection{t-Structures on Universal Nondegenerate Categories}
We equip $\D(G/N)_{\text{nondeg}}$ and $\D(N\backslash G/N)_{\text{nondeg}}$ with $t$-structures as in \cref{Intro to t-Structures on Quotient Categories}. 

\begin{Remark}
A priori, we may equip $\D(N\backslash G/N)_{\text{nondeg}} \simeq (\D(G/N)_{\text{nondeg}})^N$ with a $t$-structure determined by the condition that the forgetful functor $(\D(G/N)_{\text{nondeg}})^N \xhookrightarrow{\text{oblv}^N} \D(G/N)_{\text{nondeg}}$ is $t$-exact. However, we claim these two $t$-structures are equivalent. To see this, assume we equip $(\D(G/N)_{\text{nondeg}})^N$ with the other $t$-structure. Since the map $J_*: \D(G/N)_{\text{nondeg}} \xhookrightarrow{} \D(G/N)$ is $G$-equivariant, we see that the following diagram canonically commutes
\raggedbottom
\begin{equation*}
  \xymatrix@R+2em@C+2em{
   \D(G/N)_{\text{nondeg}} \ar[r]^{J_*} & \D(G/N)  \\
  \D(G/N)_{\text{nondeg}}^N \ar[u]^{\text{oblv}} \ar[r]^{J_*} & \D(G/N)^N \ar[u]^{\text{oblv}}
  }
\end{equation*}

\noindent and, by assumption on the $t$-structure on $\D(G/N)^N_{\text{nondeg}}$ and $\D(G/N)^N$, the forgetful functors reflect the $t$-structure. Therefore, we see that the $t$-structure on $\D(G/N)_{\text{nondeg}}^N$ determined by the requirement that $\text{oblv}^N$ is $t$-exact also has the property that $X \in \D(G/N)_{\text{nondeg}}^{N, \geq 0}$ if and only if $J_*(X) \in \D(G/N)^{N, \geq 0}$, and thus these two $t$-structures agree. 
\end{Remark}

Since $T \times T$ acts on $\D(N\backslash G/N)_{\text{nondeg}}$ compatibly with the $t$-structure by \cref{Quotient of G Category with Compatible t-Structure Action by G Subcategory Has G Action Compatible with t-Structure}, we may equip $\HNTw_{\text{nondeg}}$ with a $t$-structure such that the forgetful functor is $t$-exact. 

\begin{Proposition}\label{Right Completeness of Universal Nondegenerate G-Categories}
The $t$-structures on the categories \[\D(G/N)_{\text{nondeg}}\text{, }(\D(G/N)_{\text{nondeg}})^N\text{, and }\HNTw_{\text{nondeg}}\] are right-complete.
\end{Proposition}

\begin{proof}
The $t$-structures on $\D(G/N)_{\text{nondeg}}$ and $(\D(G/N)_{\text{nondeg}})^N$ are right-complete by construction and \cref{t-Structure on Quotient of Right-Complete Category is Right-Complete}, since the $t$-structures on $\D(G/N)$ and $\D(N\backslash G/N)$ are right-complete. Here, the $t$-structure on $\D(N\backslash G/N)$ is right-complete because it admits a $t$-exact conservative functor $\text{oblv}^N: \D(N\backslash G/N) \xhookrightarrow{} \D(G/N)$ to a category for which the $t$-structure is right-complete. Since the forgetful functor is $t$-exact and conservative, it reflects the $t$-structure. Therefore, we also obtain the right-completeness of the $t$-structure on $\HNTw_{\text{nondeg}}$. 
\end{proof}

\subsection{The Weyl Group Action on Nondegenerate Categories}\label{WeylGroupAction} In \cref{WeylGroupAction}, we construct a Weyl group action on the $N$-invariants of a $G$-category $\C$ which is nondegenerate in the sense of \cref{NondegenerateDefinition}. We now briefly sketch the construction so as to give an overview of the content of \cref{WeylGroupAction}: this sketch will also highlight where we use the fact that $\C$ is nondegenerate.

By the definition of nondegenerate $G$-categories, we have $\C^N \simeq \D(N\backslash G)_{\mathrm{nondeg}} \otimes_G \C$, and so the construction of a $W$-action on $\C^N$ will reduce to constructing a $W$-action on $\D(N\backslash G)_{\mathrm{nondeg}}$ compatible with the $G$-action. 

A category closely related to $\D(N\backslash G)$ admits an action of $W$, namely the category of modules for the \textit{ordinary} ring of differential operators on $N\backslash G$, see \cref{Big Action on Modules for Ring of Differential Operators of G Mod N}. (We briefly review this action in \cref{Gelfand-Graev Action Construction}.) Moreover, one can show that there is a fully faithful functor \[\D(N\backslash G) \xhookrightarrow{} \operatorname{H}^0\Gamma(\D_{N\backslash G})\mathrm{-mod},\] see \cref{FullyFaithfulDModLemma}. However, since $N\backslash G$ is essentially never affine, this functor is essentially \textit{never} an equivalence. Moreover, even if $G = \mathrm{SL}_2$, the full subcategory $\D(N\backslash G)$ of $ \operatorname{H}^0\Gamma(\D_{N\backslash G})\mathrm{-mod}$ is \textit{not} preserved under the $W$-action on $\operatorname{H}^0\Gamma(\D_{N\backslash G})\mathrm{-mod}$. However, in \cref{Nondegeneracy is Closed Under Actions} we show that the full subcategory $\D(N\backslash G)_{\mathrm{nondeg}}$ of $\D(N\backslash G)$ \textit{is} preserved under this $W$-action. We obtain our desired $W$-action on $\D(N\backslash G)_{\mathrm{nondeg}}$ by restricting the action on $\operatorname{H}^0\Gamma(\D_{N\backslash G})\mathrm{-mod}$ to this subcategory. 

\begin{Remark} We now compare the above strategy to other papers in the literature:
\begin{enumerate}
\item The construction of a $W$-action on $\D(N\backslash G)_{\mathrm{nondeg}}$ can be viewed as an \lq inverse\rq{} to a result of \cite{BBP}. In our notation, the main result of \cite{BBP} describes the category $\operatorname{H}^0\Gamma(\D_{N\backslash G})\mathrm{-mod}^{\heartsuit}$ in terms of the localization functor $\operatorname{H}^0\Gamma(\D_{N\backslash G})\mathrm{-mod}^{\heartsuit} \to \D(N\backslash G)^{\heartsuit}$ and its translates under the Gelfand-Graev action. On the other hand, our goal is to describe a \lq large\rq{} subcategory of $\D(N\backslash G)$ which, when viewed as a subcategory of $\operatorname{H}^0\Gamma(\D_{N\backslash G})\mathrm{-mod}$, is preserved under the $W$-action.
\item The fact that the image of $\D(N\backslash G)_{\mathrm{nondeg}}$ in $\operatorname{H}^0\Gamma(\D_{N\backslash G})\mathrm{-mod}$ is preserved under the Gelfand-Graev action can be viewed as a categorical analogue of \cite[Proposition 5.1.4]{Gin}, which states that the open subset $T^*(N\backslash G) \times_{\mathfrak{g}^*} \mathfrak{g}^*_{\mathrm{reg}}$ of the affinization \[\overline{T^*(N\backslash G)} := \Spec(\operatorname{H}^0\Gamma(\O_{T^*(N\backslash G)}))\] of the cotangent bundle $T^*(N\backslash G)$ of $N\backslash G$ given by points of $T^*(N\backslash G)$ whose value under the moment map to $\mathfrak{g}^*$ is regular is preserved under the quasi-classical Gelfand-Graev action on $\overline{T^*(N\backslash G)}$.
\item We say a $G$-category $\C$ is \textit{generated by its Whittaker invariants} if the action map $\D(G/_{\psi}N) \otimes_{\Hpsi} \C^{N^-, \psi} \to \C$ is an equivalence. Using \cref{Whittaker Invariants Are Automatically Nondegenerate}, one can show that any $G$-category $\C$ generated by its Whittaker invariants is nondegenerate, so that, in particular, $\C^N$ admits a $W$-action. On the other hand, as we have already mentioned in \cref{Weyl Group Action on Nondegenerate Quotient Subsubsection}, there are nondegenerate $G$-categories $\C$ which are not generated by their Whittaker invariants. The interplay between nondegenerate $G$-categories and $G$-categories generated by their Whittaker invariants is studied in the companion paper \cite{GannonClassificationOfNondegenerateGCategories}.

\end{enumerate}
\end{Remark}

\subsubsection{Overview of the Gelfand-Graev Action on $   \operatorname{H}^0\Gamma(\D_{G/N})$}\label{Gelfand-Graev Action Construction}First, we recall the construction of the Gelfand-Graev action on the (ordinary) ring $\operatorname{H}^0\Gamma(\D_{G/N})$ of differential operators on $G/N$. We give the classical account here; an alternative construction is given for $G$ adjoint or simply connected in \cite{GiKa}.

First, assume $[G, G]$ is simply connected. We specify an action of $s \in W$ for every simple reflection $s$. For our fixed $s$, one can construct a certain rank-two symplectic vector bundle $V_s$ over $G/[P_s, P_s]$ such that the complement of the zero section is isomorphic to $G/N$. (Because we will use this construction in what follows, we will spell out this construction in more detail in \cref{Kazhdan-Laumon SFT Construction Reminder}.) From this vector bundle, one can naturally define a \textit{symplectic Fourier transformation} \[\operatorname{H}^0\Gamma(V_s) \xrightarrow{\sim} \operatorname{H}^0\Gamma(V_s)\] on the global differential operators on $V_s$ as in \cite[Section B.2]{GinRic}. Because $V_s$ is a smooth (and, in particular, normal) variety, the restriction map gives an isomorphism \[\operatorname{H}^0\Gamma(V_s) \xrightarrow{\sim} \operatorname{H}^0\Gamma(\D_{G/N})\] of ordinary $k$-algebras, and so we obtain an involution of the ring $\operatorname{H}^0\Gamma(\D_{G/N})$ labeled by a simple reflection $s$. These symplectic Fourier transforms satisfy the braid relations and thus induce a $W$-action on $\operatorname{H}^0\Gamma(\D_{G/N})$, see for example \cite[Lemma 3.2.1(1)]{GinRic}. 

For general $G$, the Gelfand-Graev action is defined by taking the Gelfand-Graev action on $\operatorname{H}^0\Gamma(\D_{\tilde{G}/\tilde{N}})$, where $\tilde{G}$ is some reductive group with $[\tilde{G}, \tilde{G}]$ simply connected admitting a finite central isogeny $\tilde{G} \to G$ with kernel $Z$, and taking $Z$-invariants. Essentially by construction, these automorphisms upgrade the $(G \times T)$-action on $\operatorname{H}^0\Gamma(\D_{G/N})$ to an action of $G \times (T \rtimes W)$, and are compatible with the comoment map $U\LG \otimes_{Z\LG} \Symt \to \operatorname{H}^0\Gamma(\D_{G/N})$ induced from the action of $G \times T$ on $G/N$.  In particular, we obtain a group action on the category of modules for $\operatorname{H}^0\Gamma(\D_{G/N})$ using the Harish-Chandra datum in the sense of \cite[Section 1.2]{Ras3}:

\begin{Corollary}\label{Big Action on Modules for Ring of Differential Operators of G Mod N}
The group $G \times (T \rtimes W)$ acts strongly on $\operatorname{H}^0\Gamma(\D_{G/N})$-Mod, where the $W$-action is given by the Gelfand-Graev action.
\end{Corollary}

\subsubsection{Fully Faithfulness of Forgetful Functor}  We now prove the following Lemma, which relates the category $\D(G/N)$ with the ordinary category of modules for the (classical) global differential operators on $G/N$:

\begin{Lemma}\label{FullyFaithfulDModLemma}
We have the following:
\begin{enumerate}
    \item The global sections functor induces an equivalence $\D(G/N) \xrightarrow{\sim} \GlobalDiffOp\text{-mod}$.
    \item The forgetful functor $\GlobalDiffOp\text{-mod} \xrightarrow{\text{oblv}} \ClGlobalDiffOp\text{-mod}$ induced by the ring map $\ClGlobalDiffOp \to \GlobalDiffOp$ is fully faithful.
\end{enumerate}
\end{Lemma}

\begin{Remark}\label{Universal Flag Variety is Quasiaffine}
It is known that $G/N$ is a locally closed subscheme of $\mathbb{A}^n_k$ for some $n \in \mathbb{Z}^{\geq 0}$, see, for example, \cite[Section 29.2]{RomanovWilliamsonLanglandsCorrespondenceandBezrukavnikovsEquivalence}. In particular, $G/N$ is quasi-affine, and so we may realize $G/N$ as an open subset of its affine closure $\overline{G/N}$. Furthermore, since $\overline{G/N}$ is known to be normal (which follows directly from the smoothness of $G$ and \cite[Theorem 18.4]{GrosshansAlgebraicHomogeneousSpacesandInvariantTheory}, for example), we have an equivalence $\text{H}^0(\D_{\overline{G/N}}) \cong \ClGlobalDiffOp$. However, this affine closure is not smooth if $G$ is semisimple and not a product of copies of $\SL_2$. In particular, the category $\D(\overline{G/N}) := \IndCoh(\overline{G/N}_{dR})$ need not be equivalent to the category of right modules for the ring $\ClGlobalDiffOp$, and we do not know if there is such an equivalence.  
\end{Remark}

\begin{proof}[Proof of \cref{FullyFaithfulDModLemma}]
To show (1), we note that since $G/N$ is quasi-affine (see \cref{Universal Flag Variety is Quasiaffine}), the global sections functor $\text{Hom}_{\D(G/N)}(\D_{G/N}, -)$ is conservative. Therefore, since the global sections functor also commutes with colimits, by Barr-Beck we have a canonical equivalence $\D(G/N) \xrightarrow{\sim} \Gamma(\D_{G/N})\text{-mod}$. 

Now, to show (2), consider the map of rings $\ClGlobalDiffOp \to \GlobalDiffOp$. We obtain a standard adjunction 

$$\GlobalDiffOp \otimes_{\ClGlobalDiffOp} (-) : \ClGlobalDiffOp\text{-Mod} \rightleftarrows \GlobalDiffOp\text{-Mod}:  \text{oblv}$$

\noindent and therefore it suffices to show that the counit of the adjunction is an equivalence on compact generators, i.e. the natural map
\raggedbottom
\[\GlobalDiffOp \otimes_{\ClGlobalDiffOp} \GlobalDiffOp \to \GlobalDiffOp\]

\noindent is an equivalence. Note this map is filtered, and so in particular to show this map is an isomorphism we may show the induced map on associated graded is an isomorphism. However, the associated graded of $\GlobalDiffOp$ is the (not necessarily classical) ring $A$ of global functions on the cotangent bundle $T^*(G/N)$, and the associated graded of $\ClGlobalDiffOp$ is the ring $\text{H}^0(A)$ by \cite[Corollary 6.3.1]{GinRic}. In particular, we wish to show that the natural map:
\raggedbottom
\[A \otimes_{\text{H}^0A} A \to A\]

\noindent is an equivalence. 

First, note the fact that $G/N$ is quasi-affine (\cref{Universal Flag Variety is Quasiaffine}) also implies that $T^*(G/N)$ is quasi-affine, since the map $T^*(G/N) \to G/N$ is affine (as $T^*(G/N)$ is a vector bundle as $G/N$ is smooth), affine morphisms are quasi-affine \cite[Lemma 29.13.3]{StacksProject}, and so the composite morphism $T^*(G/N) \to G/N \to \ast$ is quasi-affine since the composite of quasi-affine morphisms is quasi-affine \cite[Lemma 29.13.4]{StacksProject}. In particular, $T^*(G/N)$ is an affine open subset of $\text{Spec}(\text{H}^0(A))$ \cite[Lemma 29.13.3]{StacksProject}.  Let $j$ denote this open embedding. Since the terminal map $T^*(G/N) \to \ast$ is quasi-affine, by \cite[Chapter 3, Proposition 3.3.3]{GaRoI} we obtain an equivalence $\QCoh(T^*(G/N)) \xrightarrow{\sim} \QCoh(\text{Spec}(A)) := A\text{-Mod}$.  Then, in particular, our natural map may be identified with the counit of the adjunction
\raggedbottom
\[j^*j_*(A) \to A\]

\noindent applied to the compact generator $A \in A\text{-Mod}$. However, $j_*$ is fully faithful (see, for example, the proof of \cite[Chapter 3, Proposition 3.3.3]{GaRoI}) and so this map is an equivalence. 
\end{proof}

\subsubsection{The Gelfand-Graev Action and the Whittaker Subcategory}Recall that $\D(G/N)^{N^-, \psi}$ is a full subcategory of $\D(G/N)$, see \cref{Unipotent Implies Oblv FF}. We record one result of interest with respect to the Gelfand-Graev action, see \cite[Proposition 5.5.2]{Gin}:

\begin{Proposition}\label{WActionOnWhittaker}
The Gelfand-Graev functors on $\D(G/N)^{N^-, \psi}$ are compatible with the $(W, \cdot)$-action on $\D(T)$. In particular, this category is a full $W$-subcategory of $H^0(\Gamma_{\D_{G/N}})\text{-mod}$. 
\end{Proposition}
  
\begin{Remark}
Since the category $\D(G/N)^{N^-, \psi} \xrightarrow{\sim} \D(T)$ is the derived category of its heart, the result of \cref{WActionOnWhittaker} in \cite[Proposition 5.5.2]{Gin}, stated for the classical derived category $\text{ho}(\D(T))$, immediately gives $W$-equivariance of the equivalence of DG categories $\D(G/N)^{N^-, \psi} \xrightarrow{\sim} \D(T)$ as well. 
\end{Remark} 
 \subsubsection{Reduction to Simply Connected Derived Subgroup Case}\label{Reduction to tildeG Subsubsection}
Let $\tilde{G} := G_{sc} \times Z(G)^{\circ}$, where $G_{sc}$ denotes the simply connected cover of the semisimple group $[G, G]$ and $Z(G)^{\circ}$ denotes the connected component of the identity. The following result is well-known; however, we were unable to locate a reference for the exact formulation we will use, so we provide a proof here: 

\begin{Lemma}\label{Split Reductive Group Isogeny Lemma}
The canonical map $\tilde{G} \to G$ has kernel $Z$ for $Z$ a finite central subscheme, and thus induces an isomorphism $\tilde{G}/Z \xrightarrow{\sim} G$.  Furthermore, $Z(G)^{\circ}$ is a split torus, so any reductive group in particular admits a central isogeny from a group which is the product of a semisimple simply connected group and a split torus. 
\end{Lemma}
\begin{proof}
By \cite[Remark 19.30]{MilneAlgebraicGroupsBook}, we may write $G$ as the quotient $(G_{sc} \times Z(G))/\tilde{F}$, where $\tilde{F}$ is the finite group scheme given by the scheme-theoretic image of the map $Z(G') \to G_{sc} \times Z(G)$ defined by $\xi \mapsto (\xi, \xi^{-1})$.  Now let $F$ denote the scheme-theoretic image of the closed subscheme $Z(G_{sc}) \times_{Z(G)} Z(G)^{\circ}$ under the above map to $G' \times Z(G)$. Then the map $G_{sc} \times Z(G)^{\circ} \to (G' \times Z(G))/\tilde{F}$ induces an isomorphism $(G_{sc} \times Z(G)^{\circ})/F \xrightarrow{\sim} (G_{sc} \times Z(G))/\tilde{F}$ (see \cite[Theorem 5.82]{MilneAlgebraicGroupsBook} with $H := G_{sc} \times Z(G)^{\circ}$ and $N := \tilde{F}$). 

Let $\mathscr{T} := Z(G)^{\circ}$. Then $\mathscr{T}$ is smooth, since $k$ has characteristic zero, and connected by assumption. Furthermore, $\mathscr{T}$ is diagonalizable since $Z(G)$ is \cite[Proposition 21.8]{MilneAlgebraicGroupsBook}, and subgroups of diagonalizable groups are diagonalizable \cite[Theorem 12.9(c)]{MilneAlgebraicGroupsBook}. Therefore, we see that $\mathscr{T}$ is a smooth connected diagonalizable group. We therefore see that $\mathscr{T}$ is a split torus by \cite[Chapter 12, Section e]{MilneAlgebraicGroupsBook}. 
\end{proof}

We let $\tilde{G}$ be as above and let $\phi: \tilde{G} \to G$ denote the central isogeny of \cref{Split Reductive Group Isogeny Lemma}. Let $\tilde{B} := \phi^{-1}(B)$ and $\tilde{T} := \phi^{-1}(T)$. Furthermore, let $\tilde{N} := [\phi^{-1}(B), \phi^{-1}(B)]$ denotes the unipotent radical of $\tilde{B}$ (which is not necessarily the same group as $\phi^{-1}([B, B])$), and let $\tilde{N}^- := [\phi^{-1}(B^-), \phi^{-1}(B^-)]$. Given any parabolic subgroup $P$ of $G$, we set $\tilde{P} := \phi^{-1}(G)$ to be the corresponding parabolic subgroup of $\tilde{G}$, and, given a $P$, we set $\tilde{Q} := [\tilde{P}, \tilde{P}]$. If $P$ is the rank-one parabolic subgroup $P_{\alpha}$, we will also denote this subgroup by $\tilde{P}_{\alpha}$ and $\tilde{Q}_{\alpha} := [\tilde{P}_{\alpha}, \tilde{P}_{\alpha}]$.  This notation is justified by the following:

\begin{Proposition}\label{Reduction to Split Simply Connected SS Group times Split Torus} With the above notation, we have the following:
\begin{enumerate}
    \item \cite[Theorem 22.6]{BorelLinearAlgebraicGroups} The group $\tilde{B}$ is a Borel subgroup of $\tilde{G}$ containing a maximal torus $\tilde{T}$, and $\tilde{B}$ (respectively $\tilde{T}, \tilde{P}$) is the unique Borel (respectively, torus, parabolic subgroup) for which the image under $\phi$ is $B$ (respectively, $T$, $P$). Furthermore, $\phi$ induces an isomorphism of Weyl groups. 
    \item The group $\tilde{N}$ is the unipotent radical of $\tilde{B}$, and $\phi$ induces an isomorphism $\phi|_{\tilde{N}}: \tilde{N} \xrightarrow{\sim} N$. 
    \item The closed subgroup scheme $\tilde{N}$ factors through the subgroup scheme $G_{sc} \times 1$. 
    \item Let $\alpha: N \to \mathbb{G}_a$ be any character, and let $\tilde{\alpha}$ denote the additive character $\alpha\phi|_{\tilde{N}}$. Then the pullback functor $\D(G/_{\alpha}N) \xrightarrow{} \D(\tilde{G}/_{\tilde{\alpha}}\tilde{N})$ induces an equivalence of categories $\D(G/_{\alpha}N) \xrightarrow{\sim} \D(\tilde{G}/_{\tilde{\alpha}}\tilde{N})^Z$. 
\end{enumerate}
The analogous claims to these statements also hold when $\tilde{N}$ is replaced with $\tilde{N}^-$ and $\tilde{B}$ with $\tilde{B}^- := \phi^{-1}(B^-)$. 
\end{Proposition}

\begin{proof}
As noted above, the first claim is precisely \cite[Theorem 22.6]{BorelLinearAlgebraicGroups}, using the fact that central isogenies are in particular surjective. The first part of claim (2) follows by definition, and the second part follows from the fact the kernel of the map $\phi|_{\tilde{N}}$ in particular lies in the kernel of $\phi$, which contains only semisimple elements. Thus any element in the kernel of $\phi|_{\tilde{N}}$ is both semisimple and unipotent, and thus trivial. This also shows the third claim, since there are no nontrivial unipotent elements of a torus. 

Let $Z$ denote the kernel of $\phi$. Then, using the fact that $\phi$ induces an isomorphism $\tilde{G}/Z \xrightarrow{\sim} G$ and \cref{Inv=Coinv}, we obtain (4) via:
\raggedbottom 
\[\D(G/_{\alpha}N) \simeq \D(G)^{N, \alpha} \simeq \D((G_{sc} \times Z(G)^{\circ})/Z)^{\tilde{N}, \tilde{\alpha}}\] \[\simeq D(G_{sc} \times Z(G)^{\circ})^{(\tilde{N}\times Z, \tilde{N}, \tilde{\alpha} \times 1)} \simeq D(G_{sc}/_{\tilde{\alpha}}\tilde{N} \times Z(G)^{\circ})^{Z}\]

\noindent where the final step uses the fact that any element of $\tilde{N}$ lies in $G_{\mathrm{sc}} \times 1$. 
\end{proof}

\begin{Remark} The result of \cref{Reduction to Split Simply Connected SS Group times Split Torus} in particular implies that we have a canonical induced nondegenerate character given by the composite $\tilde{N}^- \xrightarrow{\phi|_{\tilde{N}^-}} N \xrightarrow{\psi} \mathbb{G}_a$. We will lightly abuse notation and also denote this character by $\psi$. 
\end{Remark}
\renewcommand{\E}{\mathcal{E}}
\renewcommand{\G}{\mathbb{G}}
\newcommand{\Gmalpha}{\mathbb{G}_m^{\alpha}}
\newcommand{\Gsmall}{G^{\langle s \rangle}}
    
\subsubsection{Symplectic Fourier Transform Construction and Reminders}\label{Kazhdan-Laumon SFT Construction Reminder}
In this section, we recall the construction of the symplectic Fourier transformations associated to each simple reflection $s$ and recall some basic properties of the various Fourier transformations which will be used later. Let $P_{\alpha}$ denote the associated rank-one parabolic subgroup, and let $\alpha: T \to \G_m$ denote the associated root. In what follows, we will also use the notation $P_s$ for $P_{\alpha}$ depending on whether we wish to emphasize that the parabolic subgroup is associated to a simple reflection or a root. We will use similar notation for other objects that depend on a single root such as $Q_{\alpha} := Q_s := [P_s, P_s]$.

Following \cite{KaLa} and \cite{Pol}, we can construct a rank-two symplectic vector bundle on $G/[P_s, P_s]$ \textit{in the case} $[G, G]$ \textit{is simply connected} as follows. 

Let $N_{w} := N \cap \dot{w}^{-1}\dot{w_0}N\dot{w_0}\dot{w}$, where, for each $w \in W$, $\dot{w}$ is some arbitrarily chosen lift of $w$. Choose a Levi decomposition of the parabolic $P_s \cong L_s \ltimes N_{w_0s}$ and set $Q_s := [P_s, P_s]$ and $M_s := [L_s, L_s]$. Our choice of $\psi$ in particular determines an isomorphism $\mathbb{G}_a \xrightarrow{\sim} N_s$, so we can identify $M_s \cong \SL_{2, \alpha}$. With this choice, set $V_s := G/N_{w_0s} \times^{\SL_{2, \alpha}} \mathbb{A}^2$. Then the projection map $V_s \to G/[P_s, P_s]$ can be upgraded to the structure of a rank-two symplectic vector bundle, and the vector bundle $V_s$ has complement to the zero section $G/Q_{\alpha} \xhookrightarrow{z} V_s$ given by the universal flag variety $G/N \xhookrightarrow{j} V_s$ \cite{KaLa}. Furthermore, the composite $\pi j$ may be identified with the quotient map $q: G/N \to G/Q_{\alpha}$. 

\begin{Example}\label{Example for SL2}
    If $G = \SL_2$ and $N$ is the group of strictly upper triangular matrices, then the first column gives an isomorphism $G/N \cong \A^2\setminus 0$, and $V_s \cong \A^2$ is its affine closure.
\end{Example}

Let \[\mathbb{F}: \D(V_s) \to \D(V_s)\] denote the symplectic Fourier transform associated to the symplectic vector bundle $V_s$ with respect to the symplectic form $\omega: V_s \xrightarrow{\sim} V_s^{\vee}$ and let \[F_s := j^!\mathbf{F}j_{*, dR}: \D(G/N) \to \D(G/N)\] denote its restriction to $\D(G/N)$. To emphasize which vector bundle $\mathcal{V}$ we are taking the Fourier transform with respect to, we will sometimes write $\mathbb{F}_{\mathcal{V}}: \D(\mathcal{V}) \to \D(\mathcal{V}^{\vee})$ for the usual Fourier transform. Since $F_s$ is a $G$-equivariant functor, by \cref{UniversalCaseRemark}, $F_s$ is naturally isomorphic to convolution with the sheaf $F_s(\delta_{1N})$. This sheaf is known as the \textit{Kazhdan-Laumon sheaf} for $s$, see \cite{KaLa}, \cite{Pol}.

Observe that, as a scheme with an action of $N$, the rank-two vector space $V_s$ of \cref{Example for SL2} contains a distinguished line: namely, the line which is invariant under the left $N$-action. We now define a similar line bundle $\mathcal{L}_s$ over $G/Q_s$ and a map of $G/Q_s$-schemes $\ell_s: \mathcal{L}_s \to V_s$ whose fiber at every $k$-point of $G/Q_s$ yields the inclusion of a one-dimensional vector space into a two dimensional vector space in an analogous fashion. To do this, let $\Gsmall$ denote the locally closed subscheme of $G$ given by those Bruhat cells labelled by those $w \in W$ which are minimal in their right $\langle s \rangle$ cosets. Note that the notation above give an identification $N_{w_0s}\backslash (B \cap Q_s) \cong B_{\alpha}$, a Borel subgroup of $\SL_{2, s}$. Set $\mathcal{L}_s := \Gsmall/N_{w_0s} \times^{B_{\alpha}} \A^1$. Then we similarly see that we have a map $\mathcal{L}_s \to G/[P_s, P_s]$ and the locally closed embedding $\Gsmall \to G$ induces our desired map $\ell_s: \mathcal{L}_s \to V_s$ over $G/[P_s, P_s]$. 

We record two observations in the following remark.

\begin{Remark}\label{Dual to Rank One Vector Bundle is Projection Remark} By light abuse of notation, let $\omega$ denote the isomorphism $V_s \xrightarrow{\sim} V_s^{\vee}$. Duality gives a canonical map $\ell_s^{\vee}: V_s^{\vee} \xrightarrow{} \mathcal{L}_s^{\vee}$. Furthermore, note that the composite $\ell_s^{\vee}\omega \ell_s$ vanishes. Therefore there exists an isomorphism of schemes $a$ over $G/Q_s$ making the following diagram commute:
\raggedbottom
\begin{equation*}
  \xymatrix@R+2em@C+2em{
V_s \ar[r]^{\omega}  \ar[d]^{p} & V_s^{\vee} \ar[d]^{\ell_s^{\vee}} \\
V_s/\mathbb{G}_a \ar[r]^{a} & \mathcal{L}_s^{\vee}
  }
 \end{equation*}
 
 \noindent where $a$ is an isomorphism since $V_s$ has rank-two. Dually, there is an isomorphism of schemes $\omega|_{\mathcal{L}_s}$ over $G/Q_s$ making the following diagram commute:
\raggedbottom
\begin{equation}\label{Dual Morphism to Ell_s is projection dual under symplectic identification}
  \xymatrix@R+2em@C+2em{
\mathcal{L}_s \ar[r]^{\omega|_{\mathcal{L}_s}}  \ar[d]^{\ell_s} & (V_s/\mathbb{G}_a)^{\vee} \ar[d]^{p^{\vee}} \\
V_s \ar[r]^{\omega} & V_s^{\vee}
  }
 \end{equation}
\end{Remark}

We now recall a result on the usual Fourier transform, which can be proven by taking Verdier dual versions of all the functors in the proof of \cite[Th\'eor\`em 1.2.2.4]{Lau}:

\begin{Proposition}\label{Functoriality of Fourier Transform}
Assume $f: E \to V$ is a morphism of vector bundles, where $E$ has rank $e$ and $V$ has rank $r$. Then there is an isomorphism of functors $\mathbb{F}_{V}(f_{*, dR}(-)) \simeq  f^{\vee, !}\mathbb{F}_E(-)[e - r]$, where $f^{\vee}: V^{\vee} \to E^{\vee}$ is the induced morphism. 
\end{Proposition}

\subsubsection{The Symplectic Fourier Transformations Preserve the Universal Nondegenerate Category}
As in \cref{Reduction to tildeG Subsubsection}, we assume $G$ is an arbitrary (split) reductive group. Using the notation of \cref{Reduction to tildeG Subsubsection} and \cref{Kazhdan-Laumon SFT Construction Reminder}, we show that the symplectic Fourier transformations of \cref{Kazhdan-Laumon SFT Construction Reminder} induce endofunctors on the nondegenerate and degenerate subcategories of $\D(G/N)$. We first show this claim for the nondegenerate subcategory in the special case when $[G, G]$ is simply connected; in other words, we show:

\begin{Lemma}\label{Nondeg and Zero Section of Vector Bundle Don't Talk}
Assume that $\F \in \D(\tilde{G}/\tilde{N})_{\text{nondeg}}$. Then the canonical map $\mathbf{F}j_{*, dR}\F \xrightarrow{} j_{*, dR}j^!\mathbf{F}j_{*, dR}\F$ is an isomorphism. 
\end{Lemma}

\begin{proof}
Let $\mathcal{G} \in \D(\tilde{G}/\tilde{Q}_{\alpha})$. We may equivalently show that $\uHom_{\D(V_s)}(z_{*, dR}(\mathcal{G}), \F) \simeq 0$. The category $\D(\tilde{G}/\tilde{Q}_{\alpha})$ is compactly generated \cite[Corollary 3.3.3]{GaiRozCrystals}, and therefore it suffices to show this vanishing when $\mathcal{G}$ is compact. We also have
\raggedbottom
\[\uHom_{\D(V_s)}(z_{*, dR}(\mathcal{G}), \mathbf{F}j_{*, dR}(\F)) \simeq \uHom_{\D(V_s)}(\mathbf{F}z_{*, dR}(\mathcal{G}), j_{*, dR}(\F))\] \[\simeq \uHom_{\D(V_s)}(\pi^!(\mathcal{G})[-2], j_{*, dR}(\F))\]

\noindent where $\pi: V_s \to \tilde{G}/\tilde{Q}_{\alpha}$ is the projection map, $z: \tilde{G}/\tilde{Q}_{\alpha} \xhookrightarrow V_s$ is the zero section as above, and the second equivalence follows by applying \cref{Functoriality of Fourier Transform} to the inclusion map of the zero section. Therefore, if $q: \tilde{G}/\tilde{N} \to \tilde{G}/\tilde{Q}_{\alpha}$ is the quotient map, we see that the above expression is equivalent to $\uHom_{\D(\tilde{G}/\tilde{N})}(q^!(\mathcal{G})[-2], \F)$. However, the compact objects of $\D(\tilde{G}/\tilde{Q}_{\alpha})$ are cohomologically bounded (this holds if $\tilde{G}/\tilde{Q}_{\alpha}$ is replaced with any smooth, classical scheme $X$ and is an immediate consequence of the $t$-exactness of the composite $\text{oblv}: \D(X) \to \IndCoh(X) \xrightarrow{\sim} \QCoh(X)$ of conservative functors, \cite[Proposition 4.2.11]{GaiRozCrystals}) and so $q^!(\mathcal{G})$ is in particular eventually coconnective since $q^!$ is $t$-exact up to shift. Furthermore, $q^!(\mathcal{G})[-2]$ is in the kernel of $\Avpsi$ by \cref{Avpsi Vanishes on Qalpha Monodromic Objects}, so $q^!(\mathcal{G})[-2]$ is a generator of $\D(\tilde{G}/\tilde{N})_{\text{deg}}$. Therefore, $\uHom_{\D(\tilde{G}/\tilde{N})}(q^!(\mathcal{G})[-2], \F) \simeq 0$. 
\end{proof}

First, note that the symplectic Fourier transform is $t$-exact \cite[Theorem 1.2.1(iii)]{KaLa}, so the action of $\langle s \rangle$ on $\D(V_s)$ reduces to the action of $\langle s \rangle$ on the abelian subcategory $\D(V_s)^{\heartsuit}$, since $\D(V_s)$ is the derived category of its heart. This action is computed in \cite[Section 4]{Pol}. Statement (1) in the following proposition may be viewed as a derived, $D$-module variant \cite[Lemma 6.1.1]{Pol}:

\newcommand{\DNQalphanondeg}{\D(G/N)_{Q_{\alpha}\text{-nondeg}}}
\newcommand{\DNQalphainvariant}{\D(G/N)_{Q_{\alpha}\text{-nondeg}}^{\langle s \rangle}}
\newcommand{\DNQalphascnondeg}{\D(\tilde{G}/\tilde{N})_{\tilde{Q}_{\alpha}\text{-nondeg}}}
\newcommand{\DNQalphascinvariant}{\D(\tilde{G}/\tilde{N})_{\tilde{Q}_{\alpha}\text{-nondeg}}^{\langle s \rangle}}
\begin{Proposition}\label{Avpsi factors through order two coinvariants} We temporarily use the notation $\DNQalphascnondeg$ to denote the full subcategory of objects $\F \in \D(\tilde{G}/\tilde{N})$ for which $q_{*, dR}(\F)$ vanishes, which we may equivalently view as a quotient category by \cref{Intro to t-Structures on Quotient Categories}.
 \begin{enumerate}
\item The category $\DNQalphascnondeg \xhookrightarrow{j_{*, dR}} \D(V_s)$ is closed under the action of the order two group $\langle s \rangle$. 
 \item The functor $\AvN$ lifts to a functor $\AvN:\D(\tilde{G}/_{\psi}\tilde{N}^-) \to \DNQalphascinvariant$. 
    \item The functor $\Avpsi: \D(\tilde{G}/\tilde{N}) \to \D(\tilde{G}/_{\psi}\tilde{N}^-)$ has the property that, for all $\F \in \D(\tilde{G}/\tilde{N})$, there is a canonical isomorphism $\Avpsi(\F) \simeq \Avpsi(F_{s}(\F))$. 
    
\end{enumerate}
\end{Proposition}

\begin{proof}
Consider the full subcategory of $\D(V_s)$ generated by elements in the essential image of $z_{*, dR}$ and $\pi^!$. Since $\pi^![-2] \simeq \mathbf{F} z_{*, dR}$ (\cref{Functoriality of Fourier Transform}), this subcategory is closed under the symplectic Fourier transform. The right orthogonal complement can be identified with those objects $\F \in \D(V_s)$ for which $z^!(\F) \simeq 0$ and $\pi_{*, dR}(\F) \simeq 0$. The first condition is equivalent to the condition that the canonical map $\F \xrightarrow{} j_{*, dR}(j^!(\F))$ is an equivalence, and therefore this subcategory may be identified with the full subcategory of those $\F' \in \D(\tilde{G}/\tilde{N})$ such that $\pi_{*, dR}(j_{*, dR}\F') \simeq q_{*, dR}(\F')$ vanishes, i.e. $\DNQalphascnondeg$. 

For claim (2), note that $\DNQalphascnondeg$ is a full $\tilde{G}$-subcategory by construction, and therefore it is a full $\tilde{G} \times \langle s \rangle$ subcategory. Therefore the functor $\AvN$ is given by an integral kernel in $\D(\tilde{N}^-_{\psi}\backslash \tilde{G}/\tilde{N}) \simeq \D(\tilde{T})$. Direct computation shows that the kernel is given by $\delta_1 \in \D(\tilde{T})$ up to cohomological shift. 
Since the isomorphism $\D(\tilde{N}^-_{\psi}\backslash \tilde{G}/\tilde{N}) \simeq \D(\tilde{T})$ is compatible with the Gelfand-Graev action (see \cref{WActionOnWhittaker}), we see that we may lift this kernel to an object of $\D(\tilde{N}^-_{\psi}\backslash \tilde{G}/\tilde{N})^{\langle s \rangle}$, and therefore we have obtained our lift. 

Claim (3) follows since the left adjoint $\Avpsi: \D(\tilde{G}/\tilde{N}) \to \D(\tilde{G}/_{\psi}\tilde{N}^-)$ factors through $\DNQalphascnondeg$ by \cref{Avpsi Vanishes on Qalpha Monodromic Objects} and the adjoint is necessarily $\tilde{G} \times \langle s_{\alpha} \rangle$ linear by \cref{Adjoint Functor is Automatically G Equivariant} and (2), so that in particular $\Avpsi$ factors through the coinvariant category. Equivalently, by \cref{Inv=Coinv}, we have that $\Avpsi$ factors through the category of invariants $\DNQalphascinvariant$.  
\end{proof}

Next, observe that, since the functor $F_s: \D(\tilde{G}/\tilde{N}) \to \D(\tilde{G}/\tilde{N})$ is equivariant with respect to the $G$-action, it is in particular $Z$-equivariant, where $Z$ is the kernel of the map $\tilde{G} \to G$ as in \cref{Reduction to tildeG Subsubsection}. Therefore, $F_s$ induces an endofunctor of $\D(G/N)$, given by the composite \[\D(G/N) \xrightarrow{\sim} \D(\tilde{G}/\tilde{N})^Z \xrightarrow{F_s} \D(\tilde{G}/\tilde{N})^Z \xleftarrow{\sim} \D(G/N)\] the equivalence of \cref{Invariants of G Acting on Sheaves on X is Invariants on X Mod G}, $F_s$, and the inverse of this equivalence. By a light abuse of notation, we also denote this functor by $F_s$.

\begin{Corollary}\label{Fs preserves kernel of avpsi}
The functor $F_s$ preserves the category $\D(G/N)_{\text{deg}}$. 
\end{Corollary}

\begin{proof}
The forgetful functor $\text{oblv}^Z: \D(G/N) \to \D(\tilde{G}/\tilde{N})$ (see \cref{Reduction to Split Simply Connected SS Group times Split Torus}) is $t$-exact because $Z$ acts on $G$ compatibly with the $t$-structure, and is conservative. Therefore, the forgetful functor $\text{oblv}^Z$ reflects the property of being an eventually coconnective object in the kernel of the associated $\Avpsi$ functor, and thus we may prove this claim for those groups with $[G, G]$ simply connected, so that $F_s$ acts by a symplectic Fourier transform. 

In this case, note that $F_{s}$ preserves the eventually coconnective subcategory because $j_{*, dR}$ has finite cohomological amplitude \cite[Proposition 1.5.29]{HTT}, the symplectic Fourier transform is $t$-exact \cite[Theorem 1.2.1(iii)]{KaLa}, and, if $j$ is \textit{any} open embedding of smooth schemes, $j^!$ is $t$-exact. Therefore $F_s$ preserves the generating set of $\D(G/N)_{\text{deg}}$ by claim (3) of \cref{Avpsi factors through order two coinvariants}.
\end{proof}

\begin{Corollary}
Assume $\mathcal{F} \in \D(\tilde{G}/\tilde{N})_{\text{nondeg}}$ and write $\mathbf{F}j_{*, dR}\mathcal{F} \simeq j_{*, dR}(\F')$ for some $\F' \in \D(\tilde{G}/\tilde{N})$ via \cref{Nondeg and Zero Section of Vector Bundle Don't Talk}. Then $\F' \in \D(\tilde{G}/\tilde{N})_{\text{nondeg}}$.
\end{Corollary}
\renewcommand{\G}{\mathcal{G}}
\begin{proof}
Now assume $\G \in \D(\tilde{G}/\tilde{N})_{\text{deg}}$. Then 
\raggedbottom
\[\uHom_{\D(\tilde{G}/\tilde{N})}(\G, \F') \xrightarrow{\sim} \uHom_{\D(V_s)}(j_{*, dR}(\mathcal{G}), \mathbf{F}j_{*, dR}(\F)) \simeq \uHom_{\D(V_s)}(\mathbf{F}j_{*, dR}(\mathcal{G}), j_{*, dR}(\F))\]

\noindent where the first equivalence follows since $j_{*, dR}$ is an open embedding and therefore fully faithful, the second follows since $\mathbf{F}$ is canonically its own adjoint. By adjunction, \[\uHom_{\D(V_s)}(\mathbf{F}j_{*, dR}(\mathcal{G}), j_{*, dR}(\F)) \simeq \uHom_{\D(V_s)}(j^!\mathbf{F}j_{*, dR}(\mathcal{G}), \F) = \uHom_{\D(\tilde{G}/\tilde{N})}(F_s(\G), \F)\] and this space vanishes by \cref{Fs preserves kernel of avpsi}.
\end{proof}

\renewcommand{\G}{\mathbb{G}}
Since each $F_w$ for a given $w$ may be written as a composite of symplectic Fourier transformations, we see that $F_w$ preserves the degenerate and nondegenerate subcategories. We obtain the following theorem for the case $G = \tilde{G}$, and derive the general case from this one:

\begin{Corollary} \label{Nondegeneracy is Closed Under Actions}
The full subcategory $\D(G/N)_{\text{nondeg}}$ of $H^0(\Gamma(\D_{G/N}))\text{-mod}$ is preserved by the action of $T$ and the Gelfand-Graev action. In particular, the category $\D(G/N)_{\text{nondeg}}$ obtains a $G \times (T \rtimes W)$-action, where the $W$-action is given by the Gelfand-Graev action, and in particular each simple reflection acts by $F_s$.
\end{Corollary}

\begin{proof}
We have shown that $W$ acts on $\D(\tilde{G}/\tilde{N})$ in a manner which, in the notation of \cref{Reduction to Split Simply Connected SS Group times Split Torus}, commutes with the $Z$-action. Taking the $Z$-invariants and again using \cref{Reduction to Split Simply Connected SS Group times Split Torus}, we obtain our claim for general $G$. 
\end{proof}

\begin{Remark} See also \cite{BraKazSchwartz}, where an analytic version of this $W$-action on the Schwartz space of the universal flag variety is constructed.
\end{Remark}

\subsubsection{Weyl Group Action on Nondegenerate Categories}
Using \cref{Nondegeneracy is Closed Under Actions}, by the definition of nondegenerate $G$-categories (\cref{NondegenerateDefinition}) we obtain:

    \begin{Corollary}\label{Corollary on Cats}
    For any nondegenerate $G$-category $\C$, the category $\C^{N}$ admits a canonical $T \rtimes W$ action. 
    \end{Corollary}
    
    We also have the following:
    
    \begin{Corollary}\label{Corollary on T,w Cats}
    For any nondegenerate $G$-category $\C$, the category $\C^{N, (T,w)}$ is a module for the monoidal category $\IndCoh(\LTd \times_{\LTd/\Wext} \LTd)$, given a monoidal structure by the convolution formalism. In particular, $\C^{N, (T,w)}$ is acted on by the group $\Wext$. 
    \end{Corollary}
    
    \begin{proof}
    We claim, in fact, for any category $\D$ with a $T \rtimes W$ action, the category $\D^{T,w}$ acquires an action of $\Wext$. This follows since $\D^{T, w} \simeq \D(T \rtimes W)^{T_{\ell}, w} \otimes_{\D(T \rtimes W)} \D$ and because of an equivalence of categories
    \raggedbottom
    \[\uEnd_{T \rtimes W} (\D(T \rtimes W)^{T,w}) \xrightarrow{\sim} \uEnd_{\D(T)^W}(\D(T)^{T, w}) \simeq \uEnd_{\IndCoh(\LTd/\Wext)}(\IndCoh(\LTd))\]
    
    \noindent where the first functor is given by $(-)^W: \D(T \rtimes W)\text{-mod}(\DGCatContk) \to \D(T \rtimes W)^{W \times W}\text{-mod}(\DGCatContk)$, which is an equivalence by \cref{BZGO}, and the second equivalence is given by the Mellin transform. The convolution formalism \cite[Chapter 5, Section 5]{GaRoI} therefore yields a monoidal functor from $\IndCoh(\LTd \times_{\LTd/\Wext} \LTd)$ to this endomorphism category. The second claim follows since we have a monoidal functor $\IndCoh(\Wext) \to \IndCoh(\LTd \times_{\LTd/\Wext} \LTd)$, which yields a weak $\Wext$ action. This weak action automatically upgrades to a strong action since the canonical map $\Wext \to \Wext_{dR}$ is an isomorphism, which in turn follows because the de Rham functor commutes with colimits and $\Wext$ is a colimit of points.  
    \end{proof}

\subsubsection{Symplectic Fourier Transformations and Stalks}
  For the remainder of this section, fix a simple reflection $s \in W$ and write $W^{s} := \{w \in W : \ell(w) \leq \ell(ws)\}$ so that $W = W^s \sqcup W^ss$. For any subset $R$ of the Weyl group, set $X_R \xhookrightarrow{l_R} G/N$ to denote the union of the cells $NwB/N$ for every $w \in R$. In particular, this variety need not be closed. Let $\mathbb{G}_m^{\alpha}$ denote the image of the coroot $\alpha^{\vee}: \mathbb{G}_m \to G$. 
    
\begin{Proposition}\label{SwapsSupport} Fix an $\F \in \D(G/N)$. 
\begin{enumerate}
    \item Assume that the averaging with respect to the right $\mathbb{G}_m^{\alpha}$-action vanishes on $\F$. Then if $l_{W^ss}^!(\F) = 0,$ we have $l_{W^s}^!(F_s(\F)) = 0$.
    \item Assume that $\F$ is left $N$-equivariant (which, by \cref{Unipotent Implies Oblv FF}, is a property and not additional structure) and has the property that the pushforward to $G/Q_s$ vanishes. If $l_{W^s}^!(\F) \simeq 0$, then either $\F \simeq 0$ or $l_{W^s}^!(F_s(\F))$ is nonzero. 
\end{enumerate}

\end{Proposition}

We prove this after proving the following lemma, which will be used only in the proof of (2): 

\begin{Lemma}\label{N-Equivariant Sheaf on Vs with No Support on Ls is Ga Equivariant}
Assume $\F \in \D(V_s)^N$ and let $\ell_s: \mathcal{L}_s \xhookrightarrow{} V_s$ denote the map and line bundle in \cref{Kazhdan-Laumon SFT Construction Reminder}. Then if $\ell_s^!(\F) \simeq 0$, $\F$ is $\mathbb{G}_a$-equivariant.
\end{Lemma}

\begin{proof}
Consider the decomposition 
\raggedbottom
\[V_s \cong \bigsqcup_{w \in W^s}(N_w\dot{w}P_s/N_{w_0s} \mathop{\times}\limits^{\text{SL}_{2, \alpha}} \mathbb{A}^2)\]

\noindent induced by the parabolic Bruhat decomposition, where $N_w \leq N$ is some closed subgroup depending on $w$. The condition of being $\mathbb{G}_a$-equivariant is closed under colimits, since the forgetful functor commutes with colimits. Therefore we may check that the !-restriction to each locally closed subscheme $C^s(w)$ of this decomposition is $\mathbb{G}_a$-equivariant. Furthermore, by assumption that $\ell_s^!(\F) \simeq 0$, we may check this upon further !-restriction to the open subset of $C^s_0(w)$. However, on this subset, the right $\mathbb{G}_a$-action agrees with left multiplication by some one-dimensional subgroup of $N$. Therefore, the property of $N$-invariance immediately gives the property of right $\mathbb{G}_a$-invariance, where again by \cref{Unipotent Implies Oblv FF} this is a property by the fact that $\mathbb{G}_a$ is unipotent.
\end{proof}

\begin{proof}[Proof of \cref{SwapsSupport}]
We use the notation of \cref{Kazhdan-Laumon SFT Construction Reminder}. 

To prove the first claim, assume that $\F$ is such that $\text{Av}_*^{\mathbb{G}_m^{\alpha}}(\F) \simeq 0$ and that $l_{s, *, dR}l_s^!(\F) \xrightarrow{\sim} \F$. Let $\F' := j_{*, dR}(\F)$ where $j: G/N \xrightarrow{} V_s$ is the open embedding. We now claim that $\text{Av}_*^{\mathbb{G}_a}(\F') \simeq 0$ as well. To see this, note the following diagram commutes:
\raggedbottom
\begin{equation}\label{Inclusion of Complement of Zero Section Into Line Bundle over G Mod Qs}
  \xymatrix@R+2em@C+2em{
 U \ar[r]^{\mathring{j}}  \ar[dr] & \mathcal{L}_s \ar[d]^{p|_{\mathcal{L}_s}} \\
 & G/Q_s
  }
 \end{equation}

\noindent where $U := \mathcal{L}_s\backslash (G/Q_s)$ is the complement of the zero section in $\mathcal{L}_s$ and both downward pointing arrows are the structural maps. Pushing forward by the left arrow gives the $\mathbb{G}_m$-averaging and the pushforward by the right arrow gives the $\mathbb{G}_a$-averaging. Therefore we see that
\raggedbottom
\[\ell_s^!(\mathbf{F}(j_{*, dR}(\F))) \simeq (\omega \ell_s)^!\mathbb{F}(\F') \simeq \omega|_{\mathcal{L}_s}^!p^{\vee, !} \mathbb{F}(\F')[1] \simeq \omega|_{\mathcal{L}_s}^!\mathbb{F}_{\mathcal{L}_s}(\text{Av}_*^{\mathbb{G}_a}\F')[1] \simeq 0\]

\noindent where the first step uses the definition of the symplectic Fourier transform and of $\F'$, the second step uses the functoriality of !-pullback and the commutativity of \labelcref{Dual Morphism to Ell_s is projection dual under symplectic identification}, and the third step uses \cref{Functoriality of Fourier Transform} and, in particular, $\mathbb{F}_{\mathcal{L}_s}: \D(V_s/\mathbb{G}_a) \xrightarrow{\sim} \D((V_s/\mathbb{G}_a)^{\vee})$ is the Fourier transform. Because $l_{W^s}^!$ is given by the composite $\mathring{j}^!\ell_s^!$, we obtain our first claim. 

To prove the second claim, assume $\F$ is $N$-equivariant, $l_{W^s}^!(\F) \simeq 0$, the de Rham pushforward of $\F$ to $G/Q_s$ vanishes, and $\F$ is nonzero.  We first note that $j_{*, dR}(\F)$ is nonzero (since $\F$ is and $j_{*, dR}$ is fully faithful) as well as left $N$-equivariant since $j: G/N \xhookrightarrow{} V_s$ is left $G$-equivariant, and the assumption that $l^!_{W^s}(\F) \simeq 0$ implies that $\ell_s^!\F' \simeq 0$ as well. Therefore, by \cref{N-Equivariant Sheaf on Vs with No Support on Ls is Ga Equivariant}, we have that $\F' \simeq p^!(\mathcal{G})$ for some \textit{nonzero} $\mathcal{G} \in \D(V_s/\mathbb{G}_a)$. In particular, we see that
\raggedbottom
\begin{equation}\label{Fourier isos}\ell_s^!\mathbf{F}(\F') \simeq (\omega \ell_s)^!\mathbb{F}(\F') \simeq \omega|_{\mathcal{L}_s}^!p^{\vee, !}\mathbb{F}(\F') \simeq \omega|_{\mathcal{L}_s}^!\mathbb{F}_{\mathcal{L}_s}(p_{*, dR}(\F')) \simeq \omega|_{\mathcal{L}_s}^!\mathbb{F}_{\mathcal{L}_s}(\mathcal{G})\end{equation}

\noindent where the first step is the definition of $\mathbb{F}$, the second step follows by the functoriality of !-pullback and the commutativity of \labelcref{Dual Morphism to Ell_s is projection dual under symplectic identification}, and the third follows from \cref{Functoriality of Fourier Transform}, and the fourth follows from \cref{Unipotent Implies Oblv FF}. 
Note also that, if $z: G/Q_s \xhookrightarrow{} V_s$ denotes the zero section of $V_s$ and $\tilde{z}: G/Q_s \xhookrightarrow{} V_s/\mathbb{G}_a$ denotes the zero section of $V_s/\mathbb{G}_a$, then 
\raggedbottom
\[z^!\omega|_{\mathcal{L}_s}^!\mathbb{F}_{\mathcal{L}_s}(\mathcal{G}) \simeq \tilde{z}^!(\mathbb{F}_{\mathcal{L}_s}(\mathcal{G})) \simeq \tilde{t}_{*, dR}(\mathcal{G}) \simeq t_{*, dR}(\F) \simeq 0\]

\noindent where $\tilde{t}$ and $t$ are the respective maps to the base $G/Q_s$. Here, the first step uses the fact that the relevant map of zero sections commutes, the second equivalence follows from \cref{Functoriality of Fourier Transform}, the third equivalence is by the fact that $\mathcal{G} \simeq p_{*, dR}(\F')$, and the final follows by assumption on $\F$. If $\jmath$ denotes the complement of the zero section in $\mathcal{L}_s$, then we have equivalences \[\omega|_{\mathcal{L}_s}^!\mathbb{F}_{\mathcal{L}_s}(\mathcal{G}) \simeq \ell_s^!\mathbf{F}(\F')  \xrightarrow{\sim} \jmath_{*, dR}\jmath^!\ell_s^!\mathbf{F}(\F') \simeq \jmath_{*, dR}l_s^!(F_s(\F))\] from \labelcref{Fourier isos} and the fact that $z^!\omega|_{\mathcal{L}_s}^!\mathbb{F}_{\mathcal{L}_s}(\mathcal{G}) \simeq 0$. Since $\jmath_{*, dR}l_s^!(F_s(\F))$ is nonzero and $\jmath_{*, dR}$ is fully faithful, we deduce that $l_s^!(F_s(\F))$ is nonzero, as required. 
\end{proof}
    \begin{Corollary}\label{WGenerators}
    The category $\CatN$ is generated as a $W$-category by the essential image of the functor $\D(B/N)^N \to \CatN$ given by composing the pushforward of the inclusion map $N\backslash B/N \xhookrightarrow{i} N\backslash G/N$ with the quotient functor. 
    \end{Corollary}

    \begin{proof}
        Fix some nonzero $\F \in \D(N\backslash G/N)_{\mathrm{nondeg}}$. If $Z := \mathrm{ker}(\tilde{G} \to G)$, then since these terms are all compatible with the $Z$-action, we may reduce to the case where $G$ is the product of a semisimple simply connected group and a torus. In this case, we will prove the claim by induction on the minimal $w$ such that restriction to the closure of the cell $N\backslash NwB/N$ is nonzero. For our base case, if the restriction to $N\backslash B/N$ is nonzero, then $\uHom(i_{*, dR}(\mathcal{G}), \F)$ is nonzero for some $\mathcal{G} \in \D(N\backslash B/N)$. Inductively, if we choose a minimal $w \neq 1$, then there exists some simple reflection $s := s_{\alpha}$ such that $sw < w$. In this case, we may apply \cref{SwapsSupport} as the condition of nondegeneracy guarantees that the pushforward to $G/Q_s$ vanishes. Therefore $l^!_{W^s}(F_s(\F))$ is nonzero, and so by induction there exists some $w' \in W$ and $\mathcal{G} \in \D(N\backslash B/N)$ such that $\uHom(F_{w'}(\mathcal{G}), F_s(\F))$ is nonzero. Since $\uHom(F_{w'}(\mathcal{G}), F_s(\F)) \simeq \uHom(F_sF_{w'}(\mathcal{G}), \mathcal{F}) \simeq \uHom(F_{sw'}(\mathcal{G}), \mathcal{F})$ we see that $\uHom(F_{sw'}(\mathcal{G}), \mathcal{F})$ is nonzero, as desired.
    \end{proof}
    \begin{Corollary}\label{WaffGenerators}
    The category $\CatTwTw$ is generated by the objects in the set $\{w\delta : w \in \Wext\}$, where $\delta$ is the monoidal unit of $\CatTwTw$.
    \end{Corollary}

    \begin{proof}
    This follows from the fact that the set $\{\lambda \delta' : \lambda \in \characterlatticeforT\}$ is a set of compact generators in the Harish-Chandra category $\D(T)^{\text{(}T \times T\text{, w)}}$, where $\delta'$ denotes the monoidal unit, and \cref{WGenerators}.
    \end{proof}
    
\subsubsection{Triviality of Gelfand-Graev Action on Equivariant Categories}
In this section, we make precise and prove in \cref{Action of Order Two Simple Reflection Group on Associated G_m Invariants is Trivial} that the Gelfand-Graev action is trivial on equivariant categories associated to $\D(G/N)_{\text{nondeg}}$. Recall the map of vector bundles $\mathcal{L}_s \xrightarrow{\ell_s} V_s$ as in \cref{Kazhdan-Laumon SFT Construction Reminder}, a morphism of vector bundles over the base $S := G/[P_s, P_s]$. We now prove a variation of \cite[Proposition 1.2.3.1]{Lau} in our context. 

\newcommand{\linebundleofVs}{\mathcal{L}_s}
\begin{Corollary}\label{Constant Along A1 Means Fourier Transform Is}
Let $t: \linebundleofVs \to S$ be the terminal map of $S$-schemes. Then we have a canonical $B \times T$-equivariant isomorphism of functors $\omega^!\mathbb{F}_{V_{s}}\ell_{s, *, dR}t^! \simeq \ell_{s, *, dR}t^!$.
\end{Corollary}

\begin{proof}
Since our vector bundle is a rank-two symplectic vector bundle, we obtain maps such that each square of the following diagram is Cartesian: 

\begin{equation*}
  \xymatrix@R+2em@C+2em{
  \linebundleofVs  \ar[r]^{\ell_s} \ar[d]^{t} & V_s \ar[r]^{\omega} \ar[d]^{p} & V_s^{\vee} \ar[d]^{\ell_s^{\vee}}\\
  S  \ar[r]^{z} & V_s/\mathbb{G}_a \ar[r]^{\sim} & \linebundleofVs^{\vee}
  }
 \end{equation*}
 
\noindent where we use the notation of \cref{Dual to Rank One Vector Bundle is Projection Remark}. Therefore, setting $\tilde{z}$ to denote the composite of the two lower horizontal arrows, we obtain isomorphisms via various base change morphisms
\raggedbottom
\[\omega^!\mathbb{F}_{V_{s}}\ell_{s, *, dR}t^! \simeq \omega^!\ell_s^{\vee, !}\mathbb{F}_{\linebundleofVs}t^![1] \simeq \omega^!\ell_s^{\vee, !}\tilde{z}_{*, dR} \simeq \ell_{s, *, dR}t^!\]

\noindent where, specifically, the first two isomorphisms are given by the base changes required in constructing the functor of \cref{Functoriality of Fourier Transform} and the third is given by base changing along our Cartesian diagram above. 
\end{proof}

\newcommand{\sheafatoriginofmoddingoutbyoneGmalpha}{\delta_{1H}}
\begin{Proposition}\label{SFT for Each Simple Reflection Can Be Identified with Identity}
Fix a simple reflection $s \in W$, and, as above, let $F_s: \D(G/N)_{\text{nondeg}} \to \D(G/N)_{\text{nondeg}}$ denote the endofunctor given by the $W$-action above indexed by some simple root $\alpha$ (which, for $[G, G]$ simply connected, is the symplectic Fourier transformation given by $s$). Then the induced map on right $\mathbb{G}_m^{\alpha}$-equivariant objects $F_s: \D(G/N)^{\mathbb{G}_m^{\alpha}}_{\text{nondeg}} \to \D(G/N)^{\mathbb{G}_m^{\alpha}}_{\text{nondeg}}$ can be identified with the identity functor. 
\end{Proposition}

\begin{proof} Since the $W$-action commutes with the central action, it suffices to prove this in the case where the canonical map $\tilde{G} \to G$ of \cref{Reduction to tildeG Subsubsection} is an isomorphism. Let $H := \mathbb{G}_m^{\alpha}N$. By \cref{UniversalCaseRemark}, it suffices to provide an isomorphism $F_s(J^!(\sheafatoriginofmoddingoutbyoneGmalpha)) \cong J^!(\sheafatoriginofmoddingoutbyoneGmalpha)$ in $\D(G/H)_{\text{nondeg}}^{H}$. Since the forgetful functor is fully faithful at the level of abelian categories (see, for example \cite[Section 10.3]{RaskinAffineBBLocalization}), we may equivalently show that $F_s(J^!(i_{*, dR}\omega_{H/N})) \cong J^!(i_{*, dR}\omega_{H/N})$, where by Kashiwara's lemma we identify $\omega_{H/N} \in \D(G/N)$. Note that we may identify $J^!(i_{*, dR}\omega_{H/N})$ with $J^!(\tilde{i}_{*, dR}(\omega_{\overline{H/N}}))$, where $\tilde{i}: \overline{H/N} \xhookrightarrow{} V_s$ is the inclusion of the closure of $H/N$ in $V_s$. Therefore, we see:
\raggedbottom
\[F_s(J^!(i_{*, dR}{\omega_{H/N}})) \xleftarrow{\sim} F_s(J^!(\tilde{i}_{*, dR}(\omega_{\overline{H/N}}))) \simeq J^!(\omega^!\mathbf{F}(\tilde{i}_{*, dR}(\omega_{\overline{H/N}}))) \simeq\] \[J^!(\tilde{i}_{*, dR}(\omega_{\overline{H/N}})) \simeq J^!(i_{*, dR}\omega_{H/N})\]

\noindent where the second equivalence follows from \cref{Nondegeneracy is Closed Under Actions} and the definition of the symplectic Fourier transform, the second to last equivalence is given by \cref{Constant Along A1 Means Fourier Transform Is}, since $\overline{H/N}$ is naturally a closed subscheme of $\linebundleofVs$.
\end{proof}

\begin{Corollary}\label{Action of Order Two Simple Reflection Group on Associated G_m Invariants is Trivial}
The $\langle s_{\alpha} \rangle$-action on $\D(G/N)^{\mathbb{G}_m^{\alpha}}_{\text{nondeg}}$ is the trivial action.
\end{Corollary}

As above, let $H := N\mathbb{G}_m^{\alpha}$. We prove this after first proving the following lemma:
\begin{Lemma}\label{Underlying Set of Endomorphisms of Delta Sheaf in Nondegenerate Quotient is Ground Field}
For $\delta_{1H}$ the monoidal unit of $\D(H\backslash G/H)$, then we have $H^0\uEnd_{\D(H\backslash G/H)_{\text{nondeg}}}(J^!(\delta_{1H})) \cong k$. 
\end{Lemma}

\begin{proof}
We may equivalently show that the only maps $\delta_{1H} \to J_*J^!(\delta_{1H})$ are the scalar multiples of the unit of the adjunction $(J_*, J^!)$. Since the functor which forgets $H$-equivariance is fully faithful at the level of abelian categories (see, for example \cite[Section 10.3]{RaskinAffineBBLocalization}) we may equivalently show the space of maps $\text{Hom}_{\D(N\backslash G/H)^{\heartsuit}}(\delta_{1H}, Y)$ is one-dimensional, where $Y := \text{oblv}^H\tau^{\leq 0} J_*J^!(\delta_{1H})$. 

Let $\delta_{1H} \to Y$ denote any map. We note that $\delta_{1H}$ is simple (for example, it is the pushforward by a simple holonomic $\D$-module at a point). Furthermore, note that $\delta_{1H}$ is not degenerate, since it does not lie in the kernel of $\Avpsi$.  By the long exact sequence associated to the cofiber sequence $I_*I^!\delta_{1H} \to \delta_{1H} \to J_*J^!\delta_{1H}$, we see that the quotient $J_*J^!(\delta_{1H})/\delta_{1H}$ lies in the kernel of $J^!$. However, if the composite $\delta_{1H} \to Y \to Y/\delta_{1H}$ were nonzero, by the simplicity of $\delta_{1H}$ we would see that this map is an injection. This would, however, imply that $\delta_{1H}$ lies in $\D(G/H)_{\text{deg}}$, since this category is closed under subobjects--a contradiction. 
\end{proof}
\begin{proof}[Proof of \cref{Action of Order Two Simple Reflection Group on Associated G_m Invariants is Trivial}] 
Since the $W$-action commutes with the central action, it suffices to prove this in the case where the canonical map $\tilde{G} \to G$ of \cref{Reduction to tildeG Subsubsection} is an isomorphism. The action of $\langle s_{\alpha} \rangle$ on $\D(G/H)_{\text{nondeg}}$ is entirely determined by the map of pointed spaces $A: B\langle s_{\alpha} \rangle \to G\text{-Cat}^{\sim}$, where the basepoint goes to the category $\D(G/H)_{\text{nondeg}}$. We identify $B\langle s_{\alpha} \rangle \simeq \mathbb{R}P^{\infty}$ with its standard cell structure. By \cref{SFT for Each Simple Reflection Can Be Identified with Identity}, we see that the one-cell of $B\langle s_{\alpha} \rangle$ is sent to an object $I \in \uEnd_G(\D(G/H))$ equivalent to the identity. Since the identity of $\uEnd_G(\D(G/H)_{\text{nondeg}})$ is sent to the sheaf $J^!(\delta_{1H}) \in \D(H\backslash G/H)_{\text{nondeg}}$ under the equivalence $\uEnd_G(\D(G/H)_{\text{nondeg}}) \simeq \D(H\backslash G/H)_{\text{nondeg}}$ of \cref{Inv=Coinv}, we see that the endomorphisms of the identity functor are discrete. Therefore, our map $A$ is entirely determined by the object $I$ and the image of the two-cell, say $E \in \text{Hom}(\text{id}_{\D(G/H)_{\text{nondeg}}}, I^2)$. Moreover, any two such pairs $(I_1, E_1)$ and $(I_2, E_2)$ are equivalent if there exists some natural transformation $i: I_1 \xrightarrow{\sim} I_2$ such that $E_1i^2 \simeq E_2$. In particular, since we have shown that $F_s$ acts by the identity functor, it suffices to show that the canonical natural transformation 
\raggedbottom
\begin{equation}\label{Endofunctor to Compute if Z mod 2Z Action Given by Identity Is Trivial}\text{id}_{\D(G/H)_{\text{nondeg}}} \xrightarrow{E_s} F_sF_s \xrightarrow{F_sE_s^{-1}} F_s \xrightarrow{E_s^{-1}} \text{id}_{\D(G/H)_{\mathrm{nondeg}}}\end{equation}

\noindent is equal to the identity natural transformation of $G$-categories, where $E_s$ is the equivalence of \cref{SFT for Each Simple Reflection Can Be Identified with Identity}. By \cref{Underlying Set of Endomorphisms of Delta Sheaf in Nondegenerate Quotient is Ground Field}, we see that this endomorphism is given by scaling by some nonzero $x \in k$. However, the category $\D(G/H)_{\text{nondeg}}$ contains the full $\langle s_{\alpha} \rangle$-subcategory $\D(N^-_{\psi}\backslash G/N)^{\mathbb{G}_m^{\alpha}}$ by \cref{SupportOfWhittakerSheaves} and \cref{WActionOnWhittaker}. In particular, for objects of this category, the natural transformation of \labelcref{Endofunctor to Compute if Z mod 2Z Action Given by Identity Is Trivial} scales by 1. Therefore, we see that $x = 1$ and so the natural transformation of \labelcref{Endofunctor to Compute if Z mod 2Z Action Given by Identity Is Trivial} is the identity, and thus the $\langle s_{\alpha} \rangle$-action is trivial. 
\end{proof}
\section{Nondegenerate Category $\mathcal{O}$}\label{Computations in BGG Category O Section}
Fix some $\lambda \in \LTd(k)$. In this section, we study the category $\D(G/_{\lambda}B)^N$. A priori, this category is defined as a limit of DG categories. However, it turns out that this category is the derived category of its heart, which we show in \cref{Twisted D-Modules are Derived Category of Heart Section}. The category $\D(G/_{\lambda}B)^{N, \heartsuit}$ identifies with the ind-completion of the category $\mathcal{O}_{\lambda}$, the abelian subcategory of the BGG category $\mathcal{O}$ of objects\footnote{Note that $\O_{\lambda}$ is not cocomplete--its objects are, by definition, finitely generated $U\LG$-modules. It seems that these notations minimize conflicts with existing literature. Note in particular that what in the introduction we called the universal category $\mathcal{O}$, which we denoted $\LG\text{-Mod}^{N, (T,w)}$, is defined to be a DG category.} whose central character is given by $\chi_{\lambda}$, where $\chi$ denotes the Harish-Chandra map as in for example \cite[Chapter 1.9]{HumO}. 

By the Beilinson-Bernstein localization theorem \cite[Th\'eor\`eme Principal]{BB} (see also \cite{BBAProofOfJantzenConjectures}, \cite[Theorem 1.2.3]{ChenGaitsgory}) we then see that the category $\D(G/_{\lambda}B)^N$ is the (unbounded) derived category of the ind-completion of the BGG category $\mathcal{O}_{\lambda}$. Using this, we can provide an explicit description of the nondegenerate and degenerate subcategories associated to $\D(G/_{\lambda}B)^N$. We do so after reviewing a few results about the BGG category $\mathcal{O}$.


\begin{Remark} 
Note that, in particular, $\mathcal{O}_{\lambda}$ is not a block if $\lambda$ is not integral. However, the fact that $\mathcal{O}_{\lambda}$ is generated by projective modules implies that one has the block decomposition 
\raggedbottom
\[\D(G/_{\lambda}B)^N \simeq \oplus_{\lambda' \in S} \D(G/_{\lambda}B)^{(T, \mathcal{L}_{[\lambda']})\text{-mon}}\]

\noindent where $S$ denotes the set of elements in $\LTd/\Lambda(k)$ in the $W(k) = W$ orbit of $\lambda$ (for a direct explanation of this in terms of sheaves on $\D(G/_{\lambda}B)$, see \cite[Lemma 2.10]{LusYun}). In particular, this set is finite. 
\end{Remark}

\begin{Remark}
Often, the results and definitions for the BGG (abelian) category $\mathcal{O}$ are only stated for semisimple Lie algebras. However, once a central character $\lambda: Z(\LG) \to k$ is fixed, we note that the category $\mathcal{O}^{\LG}_{\lambda}$, defined with the same axioms in \cite[Section 1.1]{HumO} but for our reductive Lie algebra $\LG$, is equivalent to the category of representations at a given central character for the associated semisimple Lie algebra $\mathcal{O}^{\LG'}_{\lambda}$, where $\LG' := \text{Lie}([G, G])$. 

This follows since $U\LG \simeq U\LG' \otimes U(\text{Lie}(Z(G)^{\circ}))$ by \cref{Split Reductive Group Isogeny Lemma}, and so the generalized central character requirement, as well as the axiom that the maximal torus acts by a character on any object of $\mathcal{O}^{\LG}_{\lambda}$ implies that the action of $U(\text{Lie}(Z(G)^{\circ}))$ is entirely determined by the central character. Therefore a $U(\LG)$-representation is entirely determined by its restriction to the $U\LG'$-factor, as desired. 

In particular, while some references below only refer to the BGG category $\mathcal{O}$ for representations of semisimple Lie algebras at a given central character, the results all hold mutatis mutandis for reductive Lie algebras.
\end{Remark}


\subsection{Soergel's Classification of $\mathcal{O}_{\lambda}$}\label{Soergel Classification of Olambda Subsection}
We now briefly recall Soergel's classification of $\mathcal{O}_{\lambda}$ for some field-valued point $\lambda$ of $\LTd$. To this end, following \cite[Section 3.5]{HumO} we will say that $\lambda$ is \textit{antidominant} if $\langle \lambda + \rho, \alpha^{\vee}\rangle$ is not a positive integer for all positive coroots $\alpha^{\vee}$. 

Fix some antidominant $\lambda'$ of $\LTd$ in the $W, \cdot$ orbit of $\lambda$, and let $I_{\lambda'}$ be the indecomposable injective hull of the simple indexed by $\lambda'$. Let  \[W_{[\lambda]} := \{w \in W : w\lambda - \lambda \in \rootlattice\}\] denote the \textit{integral Weyl group} associated to $\lambda$; we give equivalent descriptions of this group in \cite[Proposition 3.2]{GannonDescentToTheCoarseQuotientForPseudoreflectionAndAffineWeylGroups}.

\newcommand{\tildeV}{\tilde{\mathbb{V}}}
\newcommand{\tildeClambda}{\tilde{C}_{\lambda}}
\begin{Theorem}\label{Soergel Summary} \cite{Soe1} For $\lambda$ as above, we have the following: \begin{enumerate}
    \item (Endomorphismensatz) The canonical map $Z\LG \to \uEnd_{\D(N \backslash G/_{\lambda}B)}(I_{\lambda'}) \in \text{Vect}^{\heartsuit}$ surjects and has the same kernel as the surjective composite $Z\LG \xrightarrow{\chi} \Symt \to C_{\lambda'}$, where $C_{\lambda'} := \Symt/\Symt^{W_{[\lambda]}}_+$ is the coinvariant algebra associated to $W_{[\lambda]}$, and therefore induces an isomorphism of classical vector spaces $\uEnd_{\D(N \backslash G/_{\lambda}B)}(I_{\lambda'}) \cong C_{\lambda'}$. 
    \item (Struktursatz) Let $\underline{I} := \oplus_{\mu} I_{\mu}$ denote the direct sum of the indecomposable injective hulls of all simple objects indexed by antidominant $\mu$ of $\LTd$ in the $W, \cdot$ orbit of $\lambda$, and let $\tildeV_I: \D(N\backslash G/_{\lambda}B) \to \tildeClambda\text{-mod}$ denote the contravariant functor $\tildeV_I(-) = \uHom(-, \underline{I})$, where $\tildeClambda := \uEnd_{\mathcal{O}_{\lambda}}(\underline{I})$. Then the functor $\tildeV_I$ is fully faithful on injective objects in $\mathcal{O}_{\lambda}$.
\end{enumerate}
\end{Theorem}

\begin{Remark}\label{Equivalence of Summary of Soergels Theorem With What He Literally Wrote}
Soergel's Struktursatz is often phrased for a block indexed by antidominant $\lambda'$ of $\LTd$ as the claim that the functor $\mathbb{V} := \text{Hom}_{\mathcal{O}_{\lambda'}}(P_{\lambda'}, -)$, where $P_{\lambda'}$ is the indecomposable projective cover associated to $\lambda'$, is fully faithful on projective objects in the block of $\mathcal{O}$ containing $P_{\lambda'}$. However, this is equivalent to our formulation above, as we explain now. 

Recall that any indecomposable projective cover of some $L_{\lambda}$ labelled by an antidominant $\lambda$ is also its injective hull \cite[Theorem 4.11]{HumO}. Therefore, the duality functor $\mathbb{D}$ gives an isomorphism $\tildeV_I \simeq \tildeV_P \circ \mathbb{D}$, where $\tildeV_P := \uHom(\underline{P}_{\lambda}, -)$ for $\underline{P}_{\lambda}$ is the direct sum of the indecomposable projective cover associated to antidominant $k$-points of $\LTd$ whose images agree under the Harish-Chandra map. 

Furthermore, each block of $\mathcal{O}_{\lambda}$ contains a unique antidominant projective \cite[Chapter 4.9]{HumO}. We therefore have that, in the notation of the Struktursatz in \cref{Soergel Summary}, $\tildeClambda \cong \oplus_{\mu}\uEnd_{\mathcal{O}_{\lambda}}(I_{\mu}) \cong \oplus_{\mu}\uEnd_{\mathcal{O}_{\lambda}}(P_{\mu})$, and so, in particular, $\tildeV_P \cong \oplus_{\mu}\uHom(P_{\mu}, -)$.  
\end{Remark}
\renewcommand{\G}{\mathcal{G}}
Soergel's results will allow us to realize the category $\mathcal{D}(G/_{\lambda}B)^N$ as the category $R_{\lambda}\text{-mod}$ for some classical ring $R_{\lambda}$ as follows. Specifically, let $\G \in \mathcal{O}_{\lambda}$ denote the direct sum of all $|W|$-many indecomposable projectives in $\mathcal{O}_{\lambda}$. Then $\G$ is a compact generator for $\mathcal{D}(G/_{\lambda}B)^N \simeq \text{Ind}(\mathcal{D}^b(\mathcal{O}_{\lambda}))$, and in particular setting $R_{\lambda} := \uEnd_{\mathcal{D}(G/_{\lambda}B)^N}(\G) \simeq  \uEnd_{\mathcal{O}_{\lambda}}(\G)$, we see that $R_{\lambda}$ is concentrated in degree zero by the projectivity of $\G$ and so $\mathcal{D}(G/_{\lambda}B)^N \simeq R_{\lambda}\text{-mod}$. The projectivity of our generators gives that this equivalence is $t$-exact, and so we recover the equivalence of abelian categories $\mathcal{O}_{\lambda} \simeq R_{\lambda}\text{-mod}^{\heartsuit, c}$. Through the Struktursatz, one can provide an alternate description of the ring $R_{\lambda}$, which we will not use here (for an excellent recent survey of this in the case where $\lambda$ is integral, see \cite[Chapter 15]{EliasMakisumiThielWilliamsonIntroToSoergelBimodules}).

\subsection{The Functor $\Avpsi$ on $\mathcal{O}_{\lambda}$}\label{The Functor Avpsi On Olambda Subsection} Fix some $\lambda \in \LTd(k)$. 

\begin{Proposition}\label{Sym(t)-Action Factors Through Finite Subscheme}
The left action of $\IndCoh(\LTd/\characterlatticeforT)$ on the category $\D(N \backslash G/_{\lambda}B)$ factors through the action of $C_{\lambda}\text{-mod}$ via the monoidal functor 
\raggedbottom
\[\QCoh(\text{Spec}(C_{\lambda})) \xrightarrow{\Upsilon_{\text{Spec}(C_{\lambda})}} \IndCoh(\text{Spec}(C_{\lambda})) \xrightarrow{f^!} \IndCoh(\LTd/\characterlatticeforT)\]

\noindent where $C_{\lambda}$ is as in \cref{Soergel Summary} and $f$ denotes the composite $\text{Spec}(C_{\lambda}) \xhookrightarrow{} \LTd \to \LTd/\characterlatticeforT$.  
\end{Proposition}

\begin{proof}
The block decomposition $\D(G/_{\lambda}B)^N \simeq \oplus_{[\lambda'] \in S} \D(G/_{\lambda}B)^{(T, \mathcal{L}_{[\lambda']})\text{-mon}}$ discussed above implies that the action of 
\raggedbottom
$\IndCoh(\LTd/\characterlatticeforT)$ factors through the action of \[\IndCoh((\LTd/\characterlatticeforT)^{\wedge}_S) \simeq \IndCoh(\coprod_{\lambda'}(\LTd/\characterlatticeforT)^{\wedge}_{[\lambda']}) \simeq \IndCoh(\coprod_{\lambda'}\LT^{\ast, \wedge}_{\lambda'})\]

\noindent where the $\lambda'$ are arbitrarily chosen antidominant lifts of $[\lambda']$ and the second equivalence is a general fact about formal completions with actions of discrete groups, see \cite{GannonDescentToTheCoarseQuotientForPseudoreflectionAndAffineWeylGroups}. We forget this action to an action of $\Upsilon_{\LTd}: \QCoh(\LTd) \xrightarrow{\sim} \IndCoh(\LTd)$ by the pullback map induced by  $\coprod_{\lambda'}\LT^{\ast, \wedge}_{\lambda'} \to \LTd$. 

Let $\G$ be a direct sum of all indecomposable projective objects of $\mathcal{O}_{\lambda}$. Then $\mathcal{G}$ is a projective generator of $\mathcal{D}(G/_{\lambda}B)^N$, the action of $\QCoh(\LTd) \simeq \Symt\text{-mod}$ is determined by the map $\Symt \to \uEnd_{\mathcal{D}(G/_{\lambda}B)^N}(\G)$. The projectivity of $\G$ implies that this endomorphism ring is concentrated in degree zero. The functor $\tildeV$ can be canonically equipped with a $\Symt$-linear structure, and so in particular the following diagram commutes
\raggedbottom
\begin{equation*}
  \xymatrix@R+2em@C+2em{
\Symt \ar[d]^{\exists} \ar[r] &  \uEnd_{\D(N \backslash G/_{\lambda}B)}(\mathcal{G}) \ar[d]^{\sim}  \\
C_{\lambda} \ar[r]^{} & \uEnd_{C_{\lambda}\text{-mod}}(\tildeV(\mathcal{G}))
  }
 \end{equation*}
 \noindent where the right equivalence is given by the struktursatz. Therefore, the action of $\QCoh(\coprod_{\lambda'}\LT^{\ast, \wedge}_{\lambda'})$ factors through an action of $C_{\lambda}\text{-mod}$, and so the claim follows since $\Upsilon$ intertwines pullbacks for $\QCoh$ and $\IndCoh$ for laft prestacks \cite[Chapter 6, Section 3.3.5]{GaRoII}.  
\end{proof}
\renewcommand{\G}{\mathbb{G}}

\newcommand{\Oindad}{\mathcal{D}^{\lambda}(B\backslash G/N)_{\text{ad}}}
\newcommand{\Hll}{\mathcal{H}_{\lambda, -\lambda}}

\begin{Proposition}\label{Avpsi is Translation} Using the notation of \cref{Soergel Summary}, we have the following: \begin{enumerate}
    \item The induced functor $\text{Vect} \simeq \D(G/_{\lambda}B)^{N^{-}, \psi} \xrightarrow{\AvNshifted} \D(G/_{\lambda}B)^N$ sends $k \in \text{Vect}$ to an object isomorphic to $\underline{I}_{\lambda}$.
    \item We may identify the left adjoint to this composite functor (which by abuse of notation we also denote $\Avpsishifted$) with the functor $\mathbb{D}_{\text{Vect}}\tildeV \mathbb{D}$, where $\mathbb{D}$ denotes the Verdier duality functor on $\D(G/_{\lambda}B)^N$. 
\end{enumerate}

\end{Proposition}

\begin{proof}
The composite is $t$-exact by \cref{Ginzburgt-Exactness}, and so if we let $A$ denote the object $\AvNshifted(k)$, $A \in \D(G/_{\lambda}B)^{N, \heartsuit}$. Because the adjoint to $\AvNshifted$ is also $t$-exact by \cref{BBMShiftedLeftAdjointIsExact}, $A$ is an injective object. Every injective object of $\D(G/_{\lambda}B)^{N, \heartsuit}$ is a direct sum of indecomposable injective objects and therefore it suffices to compute $\uHom(L, A) \simeq \uHom(\Avpsishifted(L), k) \cong \Avpsishifted(L)$ for each simple object $L$. 

Direct computation shows that $\Avpsishifted: \D(G/N)^{N} \to \D(N^-_{\psi} \backslash G/N) \simeq \D(T)$ is the identity upon restriction to the torus. Therefore, we see that in our case, $\Avpsishifted$ sends the Verma $\Delta_{\lambda}$ to the one-dimensional vector space of $\text{Vect} \simeq \D(G/_{\psi} N^-)^{B_{[\lambda]}}$. For any other Verma module $\Delta$, there exists some bi-equivariant sheaf $\tilde{\Delta}$ for which $- \star \tilde{\Delta}$ is an equivalence and the sheaf $\Delta_{\lambda} \star \tilde{\Delta}$ is concentrated in degree zero and isomorphic to $\Delta$ \cite[Lemma 3.4, Lemma 3.5]{LusYun}. The residual equivariance from the fact that $\Avpsishifted: \D(N^-_{\psi}\backslash G) \to \D(N\backslash G)$ is $G$-equivariant therefore implies that the images of \textit{all} Verma modules are isomorphic to the unique one-dimensional vector space concentrated in degree zero. Furthermore, we have that $\Avpsishifted$ preserves projective objects since it is $t$-exact by \cref{BBMShiftedLeftAdjointIsExact} and all objects in $\text{Vect}^{\heartsuit}$ are projective. 

Now, assume we are given a simple $L(\mu) \in \mathcal{O}$ for which $\mu$ is not antidominant. Then $L(\mu)$ can be written as a subobject of a quotient of a Verma module by another Verma module by Verma's theorem \cite[Theorem 4.6]{HumO}. Therefore, by $t$-exactness of $\Avpsishifted$ and the fact that this injective map of Verma modules is sent to an isomorphism, $\Avpsi(L(\mu))$ vanishes.  Furthermore, the simples $L(\mu)$ corresponding to antidominant $\mu$ are isomorphic to their corresponding Verma modules, and therefore, by the above, $\Avpsi$ sends them to the one-dimensional vector space. Thus we see that $A$ is an injective object whose vector space of maps from antidominant simples is one-dimensional, and whose vector space of maps from any other simple is zero, implying claim (1). 

The above argument also implies that $\mathbb{D}_{\text{Vect}}\Avpsishifted \simeq \uHom(-, \underline{I}_{\lambda})$. Projectives associated to such antidominant objects of $\LTd(k)$ are self-dual \cite[Chapter 7.16]{HumO}, and the Beilinson-Bernstein localization intertwines the duality functor on $\mathcal{O}$ with Verdier duality. Therefore, we see that $\mathbb{D}_{\text{Vect}}\Avpsishifted \mathbb{D} \simeq \uHom(\underline{P}_{\lambda}, -)$. Since Verdier duality is an equivalence, we obtain (2). 
\end{proof}

\subsection{Nondegenerate Category $\mathcal{O}$}
We now proceed to explicitly classify the category $(\D(G/_{\lambda}B)_{\text{nondeg}})^N$, which we informally think of as the nondegenerate BGG category $\mathcal{O}_{\lambda}$ or, more accurately, the derived category of its ind-completion. We provide an alternate description of this category in \cref{Alternate Description of Nondegenerate Olambda Subsubsection} and use this to provide an explicit description of the nondegenerate category $\mathcal{O}_{\lambda}$. 

\newcommand{\DGNBlambdacecko}{\mathcal{D}(B_{\lambda}\backslash G/N)_{\text{cecko}}}
\newcommand{\DGNBlambdanoncecko}{\mathcal{D}(B_{\lambda}\backslash G/N)_{\text{non-cecko}}}
\subsubsection{Equivalence of Alternate Definition of Nondegenerate Category $\mathcal{O}_{\lambda}$}\label{Alternate Description of Nondegenerate Olambda Subsubsection}
In this section, we prove the following technical result which may be skipped at first pass. Let $\DGNBlambdacecko$ denote the full subcategory of $\D(B_{\lambda}\backslash G/N)$ generated under colimits by eventually coconnective objects in the kernel of the right Whittaker averaging functor $\Avpsi$. This category admits a fully faithful functor $\DGNBlambdacecko \xhookrightarrow{I_{*, \text{cecko}}} \D(B_{\lambda}\backslash G/N)_{\text{deg}}$. For any object $\mathcal{F} \in \D(B_{\lambda}\backslash G/N)$ which is a colimit of eventually coconnective objects in the kernel of $\text{Av}_!^{\psi}: \D(B_{\lambda}\backslash G/N) \to \D(B\backslash_{\lambda}G/_{\psi}N^-)$ has the property that $\text{oblv}^{B_{\lambda}}(\F)$ is a colimit of eventually coconnective objects in the kernel of $\text{Av}_!^{\psi}: \D(G/N) \to \D(G/_{\psi}N^-)$ since this forgetful functor is $t$-exact and continuous. A priori, this map need not be an equivalence. However, we have the following result: 

\begin{Proposition}\label{All Equivariant Objects Which Forget to Ceckos are Ceckos in the equivariant category}
The inclusion functor $\DGNBlambdacecko \xhookrightarrow{} (\D(G/N)_{\text{deg}})^{B_{\lambda}}$ is an equivalence.
\end{Proposition}

We will now reduce the proof of \cref{All Equivariant Objects Which Forget to Ceckos are Ceckos in the equivariant category} to \cref{Nonecko Category Olambda is Nondeg Olambda} after setting some notation. By right-completeness of the category $\D(B_{\lambda}\backslash G/N)_{\text{cecko}}$ with respect to its $t$-structure and \cref{Colim in AB5 is Colim of Image}, we can identify $\D(B_{\lambda}\backslash G/N)_{\text{cecko}}$ as the full subcategory generated by \textit{compact} objects in the kernel of the right Whittaker averaging functor, using the fact that $\Avpsi$ is $t$-exact up to cohomological shift i.e. \cref{BBMShiftedLeftAdjointIsExact}. Therefore, \cref{Subcategory Generated By Compact Objects Implies Inclusion Preserves Compact Objects} shows that the inclusion functor also preserves compacts, and thus the inclusion functor $I_{*, \text{cecko}}$ falls into the setup of \cref{Intro to t-Structures on Quotient Categories}. Let $\DGNBlambdanoncecko$ denote the resulting quotient category. Since we can identify $\DGNBlambdanoncecko$ as the kernel of the right adjoint to the inclusion functor $I^!_{\text{cecko}}$, the proof of \cref{All Equivariant Objects Which Forget to Ceckos are Ceckos in the equivariant category} immediately reduces to the following result:

\begin{Proposition}\label{Nonecko Category Olambda is Nondeg Olambda}
The canonical quotient functor $\DGNBlambdanoncecko \to (\D(G/N)_{\text{nondeg}})^{B_{\lambda}}$ is an equivalence. 
\end{Proposition}

\begin{proof}
We first note that the forgetful functor $\D(B_{\lambda}\backslash G/N) \xrightarrow{\text{oblv}} \D(G/N)$ reflects the property of being an eventually coconnective object in the kernel of right Whittaker averaging $\text{Av}_!^{\psi}$. Therefore, we obtain a canonical equivalence $\DGNBlambdacecko^+ \xhookrightarrow{\sim} (\D(G/N)_{\text{deg}})^{B_{\lambda}, +}$. By \cref{Quotient Category Identification on Eventually Coconnective Subcategories}, we see that the quotient functor of \cref{Nonecko Category Olambda is Nondeg Olambda} is an equivalence on the eventually coconnective subcategories. 

We claim that both categories of \cref{Nonecko Category Olambda is Nondeg Olambda} have a compact generator which is eventually coconnective. For the category $\DGNBlambdanoncecko$, this is a direct consequence of the fact that $J^!$ kills any non-antidominant simple because non-antidominant simples are eventually coconnective objects of $\D(B_{\lambda}\backslash G/N)$ in the kernel of the right Whittaker averaging functor, and so $J^!(\oplus_a L_a)$ is a compact generator where $a$ varies over all antidominant weights in the $(W, \cdot)$-orbit of $\lambda$. An identical argument says the image under the quotient functor $J^!_{\text{non-cecko}}: \D(B_{\lambda}\backslash G/N) \to \DGNBlambdanoncecko$ of the antidominant simples form a set of compact generators for $\DGNBlambdanoncecko$. However, the image of any object in the heart is in particular eventually coconnective, we see that the quotient functor of \cref{Nonecko Category Olambda is Nondeg Olambda} is fully faithful and maps a compact generator to a compact generator, and thus is an equivalence.
\end{proof}

\subsubsection{Nondegeneracy at a Fixed Character - The Eventually Coconnective Case}
We now study the eventually coconnective objects of the nondegenerate and degenerate category $\mathcal{O}_{\lambda}$:

\begin{Lemma}\label{Equivalent Notions of Degeneracy for Ind Derived Category O}
The following not necessarily cocomplete DG categories are equivalent:
\begin{enumerate}
    \item The subcategory of $\D(G/N)^{B_{\lambda}, +}$ given by those objects $\F$ for which $\text{oblv}^{B_{\lambda}}(\F)$, which lies in $\D(G/N)^+$ since the $t$-structure on $\D(G/N)^{B_{\lambda}}$ is such that the forgetful functor is $t$-exact, lies in  $\D(G/N)_{\text{deg}}^+$.   
    \item The subcategory of $\D(G/N)^{B_{\lambda}, +}$ given by those objects $\F$ for which $\Avpsi(\F) \simeq 0$ in $\D(G/_{\psi}N^-)^{B_{\lambda}, +}$.
    \item The full subcategory of $\D(G/N)^{B_{\lambda}, +}$ generated under filtered colimits and extensions\footnote{In other words, if any two of the three elements in a cofiber sequence lie in our category, so too does the third.} by the simples $L_{\mu}$ such that $\mu$ lies in the ($W, \cdot$)-orbit of $\lambda$ and $\mu$ is not antidominant. 
\end{enumerate}
\end{Lemma}

\begin{proof}
The forgetful functor $\text{oblv}^{B_{\lambda}}$ is $t$-exact and conservative by assumption, and furthermore we have $\text{oblv}^{B_{\lambda}}\Avpsi \simeq \Avpsi\text{oblv}^{B_{\lambda}}$ since the functor $\Avpsi$ averages with respect to the right action. Therefore, the functor $\text{oblv}^{B_{\lambda}}$ reflects the property of being an eventually coconnective object of the kernel of the respective right Whittaker averaging $\Avpsi$. Therefore, by the definition of $\D(G/N)_{\text{deg}}$ (see \cref{Definition of D(G/N)deg}), we see the equivalences of the not necessarily cocomplete DG categories of (1) and (2).

It remains to show the equivalence of the not necessarily cocomplete DG categories of (2) and (3). By the identification $\Avpsi: \D(G/N)^{B_{\lambda}} \to \D(G/_{\psi}N^-)^{B_{\lambda}}$ with dual of the functor $\tildeV$ (\cref{Avpsi is Translation}), all of the simples $L_{\mu}$ as in (3) lie in the kernel of this functor. Thus, by exactness and continuity of $\Avpsi$, we see that the category of (3) lies in the subcategory of (2).

Conversely, assume we are given some eventually coconnective object $\F \in \D(G/N)^{B_{\lambda}, +}$. This category is the derived category of its heart by \cref{N Invariant Twisted D Modules on Flag Variety Is Derived Category of Heart Proposition} and the $t$-structure on $\D(G/N)^{B_{\lambda}}$ is right-complete by \cref{Conservative t Exact Functor to right-complete Implies right-complete and Dual Statement}, since it admits a $t$-exact, conservative functor $\text{oblv}^{B_{\lambda}}$ to a category $\D(G/N)$ with a right-complete $t$-structure. In particular, we can write $\F$ as a filtered colimit of objects with finitely many nonzero cohomology groups. In turn, this implies that we can write $\F$ as a filtered colimit of objects obtained by successive extensions (i.e. cofiber sequences) of objects in the heart. Because $\Avpsi$ also commutes with cohomological shifts, it suffices to show that if $M \in \D(G/N)^{B_{\lambda}, \heartsuit}$ lies in the kernel of $\Avpsi$, then $M$ is a successive extension of objects $L_{\mu}$ as in (3). 

The category $\D(G/N)^{B_{\lambda}, \heartsuit}$ can be identified with the ind-completion of the BGG category $\mathcal{O}_{\lambda}$ by \cref{Twisted D-Modules are Derived Category of Heart Section}; see \cref{Ind Completion of Abelian Category in Higher Categorical Setting Remains Abelian}. In particular, by \cref{Colim in AB5 is Colim of Image}, we may write $M$ as a union of its compact subobjects. Since $\mathcal{O}_{\lambda}$ is closed under direct summands, this implies that we may write $M$ as a union of objects of $\mathcal{O}_{\lambda}$. In particular, either $M$ lies in the abelian category generated by extensions of the simples as in (3) or there exists some compact subobject which has some $L_{\nu}$ as a composition factor, where $\nu$ is antidominant and in the $(W, \cdot)$-orbit of $\lambda$. In the latter case, though, we see by \cref{Avpsi is Translation} that $\Avpsi(L_{\nu})$ does not vanish. Therefore this latter case does not occur, so $M$ can be realized as a filtered colimit of successive extensions of simple objects labeled by non-antidominant weights, as desired. 
\end{proof}

\begin{Lemma}\label{Enhanced V is Monadic on Eventually Coconnective Derived Ind Category O}
The induced functor $\uHom(\underline{P}_{\lambda}, -): \D(G/_{\lambda}B)^{N, +}_{\text{nondeg}} \to \text{Vect}^+$ is monadic, and induces an equivalence 
\raggedbottom
\[\D(G/_{\lambda}B)^{N, +}_{\text{nondeg}} \xrightarrow{\tildeV} \oplus_w C_{w\lambda}\text{-mod}^+\]

\noindent where $w$ varies over those $w \in W$ such that $w\lambda$ is antidominant. 
\end{Lemma}

\begin{proof}
We check the conditions of Barr-Beck. We have that all indecomposable projective covers of simples are in the abelian category $\mathcal{O}_{\lambda}$ itself, and therefore the direct sum of all indecomposable projective covers is compact in $\uHom(\underline{P}_{\lambda}, -)$ by \cref{Twisted D-Modules are Derived Category of Heart Section}, and in particular \cref{Ind of Bounded Derived is Bounded Derived of Ind for Abelian Cats with Finite Cohomological Dimension and Proj Generators}. Therefore this functor commutes with geometric realizations, and thus it remains to verify its conservativity. Assume $\F \in \D(G/_{\lambda}B)^{N, +}_{\text{nondeg}}$ is nonzero. By the exactness of $\uHom(\underline{P}_{\lambda}, -)$ (i.e. because this functor commutes with shifts), it suffices to assume that $\F \in \D(G/_{\lambda}B)_{\text{nondeg}}^{N, \geq 0}$ and that $H^0(\F) \cong \tau^{\leq 0}(\F)$ is nonzero. Note that this uses the fact that $\F$ is bounded-below--in any category $\D$ with a $t$-structure and some $\mathcal{G} \in \D^+$, we may choose a maximal $m$ such that $\mathcal{G} \in \D^{\geq m}$, so that in particular $\tau^{\geq m}\mathcal{G} \simeq H^0(\mathcal{G})$ does not vanish. 

Therefore, we similarly see that $H^0(J_*(\F))$ is nonzero, since $J^!$ is $t$-exact and so $J^!(H^0(J_*(\F))) \cong H^0(J^!J_*(\F)) \simeq H^0(\F)$, where the last step follows by the fully faithfulness of $J_*$. Then, since we have a right adjoint $J_*$ to a quotient functor $J^!$ we have seen is $t$-exact in \cref{Raskin t-ExactCatProp}, we have that $J_*(\F)$ lies in the subcategory $\D(G/_{\lambda}B)^{N, \geq 0}$ and $H^0(J_*(\F))$ is nonzero. 

By the $t$-exactness of $\uHom(\underline{P}_{\lambda}, -)$, we have that $\uHom(\underline{P}_{\lambda}, H^0(J_*(\F))) \simeq H^0\uHom(\underline{P}_{\lambda}, J_*(\F))$, and so in particular it suffices to show that $\uHom(\underline{P}_{\lambda}, H^0(J_*(\F)))$ is nonzero. However, $H^0(J_*(\F))$ is a nonzero object in the ind-completion of $\mathcal{O}_{\lambda}$. In particular, by \cref{Colim in AB5 is Colim of Image}, it can be written as an increasing union of objects of $\mathcal{O}_{\lambda}$. It cannot be the case that all objects in this increasing union have composition factors $L_{\mu}$ for $\mu$ not antidominant, for this would imply that $J^!(H^0(J_*(\F)))$ vanishes, which we have seen above cannot happen. Therefore, there exists some object $M \in \mathcal{O}_{\lambda}$ which is a subobject of $H^0(J_*(\F))$ and such that there exists some antidominant $\mu$ for which $L_{\mu}$ is a subquotient of $M$. By the $t$-exactness of $\uHom(\underline{P}_{\lambda}, -)$, we see that $\uHom(\underline{P}_{\lambda}, M)$ contains $\uHom(\underline{P}_{\lambda}, L_{\mu})$ as a subquotient, and therefore $\uHom(\underline{P}_{\lambda}, M)$ cannot vanish. Therefore, again using $t$-exactness of $\uHom(\underline{P}_{\lambda}, -)$, we see that $\uHom(\underline{P}_{\lambda}, H^0(J_*(\F))) \simeq H^0\uHom(\underline{P}_{\lambda}, J_*(\F))$ cannot vanish either, thus verifying the conservativity condition of Barr-Beck. Since the conditions of Barr-Beck apply and the left adjoint to $\uHom(\underline{P}_{\lambda}, -)$ is $\underline{P}_{\lambda} \otimes_k -$, we see that the Endomorphismensatz of \cref{Soergel Summary} gives our desired equivalence. 
\end{proof}

\subsubsection{Nondegenerate Category $\mathcal{O}_{\lambda}$ - The General Case}
We now use the above to give a coherent description of $\D(G/_{\lambda}B)^N_{\text{nondeg}}$:
\raggedbottom
\begin{Proposition}\label{Computation of Nondegenerate Category O at a Block}
There is an equivalence of categories $\D(G/_{\lambda}B)^N_{\text{nondeg}} \simeq \IndCoh(W \mathop{\times}\limits^{W_{[\lambda]}}\text{Spec}(C_{\lambda}))$. 
\end{Proposition}
We prove this after proving the following lemma:

\begin{Lemma}\label{All Local Rings Which Vanish to High Power Have Trivial Subobject Lemma}
Let $C$ be a ring with a unique maximal ideal $\mathfrak{m}$ such that $\mathfrak{m}^N = 0$ for $N \gg 0$. Then any classical $C$-module (i.e. an element of $C$-mod$^{\heartsuit}$) contains a submodule isomorphic to $C/\mathfrak{m}$. 
\end{Lemma}
\begin{proof}
Let $M$ be any nonzero $C$-module. Then since $M$ is not zero, there exists a nonzero map $C \to M$. Let $I$ denote its kernel, so that we have an induced injection $C/I \xhookrightarrow{} M$. We wish to exhibit a nonzero map $C/\mathfrak{m} \to C/I$. Since $C \to M$ is nonzero, $I \neq C$. Therefore, there exists a nonnegative integer $n$ such that $\mathfrak{m}^{n + 1} \subseteq I \subsetneq \mathfrak{m}^n$, since $\mathfrak{m}^N = 0$ for $N \gg 0$. Let $x \in \mathfrak{m}^n \setminus I$, and consider the map of $C$-modules $C \to C/I$ sending $1$ to $x$. Then the kernel of this map is precisely $\mathfrak{m}$ since the kernel contains $\mathfrak{m}$ and does not contain $1$. Therefore, we obtain an injective map $C/\mathfrak{m} \xhookrightarrow{} C/I$. Composing with our injection $C/I \xhookrightarrow{} M$, we obtain our claim. 
\end{proof}

\begin{proof}[Proof of \cref{Computation of Nondegenerate Category O at a Block}]
By construction, we have a quotient functor \[J^!: \D(G/_{\lambda}B)^N \to \D(G/_{\lambda}B)^N_{\text{nondeg}}\] admitting a continuous right adjoint. In particular, $J^!$ preserves the set of compact generators $\oplus_{w}L_{w \cdot \lambda}$ where $w$ varies over $W$. However, we see that for any $w$ such that $w \cdot \lambda$ is not antidominant, $\Avpsi(L_{w \cdot \lambda})$ vanishes by \cref{Avpsi is Translation} and the fact that all indecomposable injective hulls of simples of $\mathcal{O}_{\lambda}$ are distinct. Therefore, we see that $\underline{L}_a^{\lambda} := \oplus_{w}J^!(L_{w \cdot \lambda})$ is a compact generator, where $w$ varies only over those $w \in W$ such that $w \cdot \lambda$ is also antidominant. Furthermore, since each block of the abelian category $\mathcal{O}_{\lambda}$ contains precisely one simple labeled by an antidominant weight, we see that the category $\D(G/_{\lambda}B)^N_{\text{nondeg}}$ is equivalent to modules over the ring
\raggedbottom
\[A_{\lambda} := \uEnd_{\D(G/_{\lambda}B)^N_{\text{nondeg}}}(\underline{L}^{\lambda}_a) \simeq \oplus_{w} \uEnd_{\D(G/_{\lambda}B)^N_{\text{nondeg}}}(J^!(L_{w\cdot \lambda}))\]

\noindent where we note that there are precisely $[W : W_{\lambda}]$ many (nonzero) objects in the right direct sum. Furthermore, we may compute the right-hand side explicitly as follows. Let $L_a$ be any simple indexed by an antidominant weight. Then, by \cref{Enhanced V is Monadic on Eventually Coconnective Derived Ind Category O}, we see that these endomorphisms are equivalently given by the $C_{\lambda}$-module structure on $\tildeV(L_a) \simeq k$. However, since $C_{\lambda}$ is an Artinian ring with a unique maximal ideal, we see that, by \cref{All Local Rings Which Vanish to High Power Have Trivial Subobject Lemma}, there is precisely one module structure which we may place on $\tildeV(L_a)$--namely, the trivial one. Therefore we see that the category $\D(G/_{\lambda}B)^N_{\text{nondeg}}$ is a direct sum of $[W : W_{[\lambda]}]$ many copies of $E_{\lambda}\text{-mod}$ for $E_{\lambda} := \uEnd_{C_{\lambda}\text{-mod}}(k)$. 

An identical description holds for $\IndCoh(W \mathop{\times}\limits^{W_{[\lambda]}}\text{Spec}(C_{\lambda}))$. Specifically, we note that this category is equivalent to a direct sum of $[W : W_{[\lambda]}]$ many copies of $\IndCoh(\text{Spec}(C_{\lambda}))$ and that moreover if $i: \Spec(k) \xhookrightarrow{} \Spec(C_{\lambda})$ denotes the closed embedding of the unique closed point, then $i_*^{\IndCoh}(k)$ is a compact generator of $\IndCoh(\text{Spec}(C_{\lambda}))$, which, for example, directly follows from \cite[Chapter 4, Proposition 6.2.2]{GaRoI}. Furthermore, since $i_*^{\IndCoh}(k) \in \IndCoh(\Spec(C_{\lambda}))^{\heartsuit}$ and the quotient functor 
\noindent 
\[\Psi_{\Spec(C_{\lambda})}: \IndCoh(\Spec(C_{\lambda})) \to \QCoh(\Spec(C_{\lambda})) \simeq C_{\lambda}\text{-mod}\]

\noindent is $t$-exact and induces an equivalence on the heart \cite[Chapter 4, Proposition 1.2.2]{GaRoI} and is compatible with pushforwards \cite[Chapter 4, Proposition 2.1.2]{GaRoI}, we see that \[\uEnd_{\IndCoh(\Spec(C_{\lambda}))}(i_*^{\IndCoh}(k)) \simeq \uEnd_{C_{\lambda}\text{-mod}}(k) =: E_{\lambda}.\] Thus $\IndCoh(W \mathop{\times}\limits^{W_{[\lambda]}}\text{Spec}(C_{\lambda}))$ is a direct sum of $[W : W_{[\lambda]}]$ many copies of $E_{\lambda}\text{-mod}$, as desired. 
\end{proof}

\subsection{Classification of Degenerate Subcategory of Universal Flag Variety}
\newcommand{\monpluscat}{\D(G/N)^{N, +}_{\text{mon}}}
\newcommand{\monHNpluscat}{\D(G/N)^{N,+}_{\overline{\text{mon}}}}
\newcommand{\DGNmonbar}{\D(G/N)_{\overline{\text{mon}}}}
\newcommand{\DGNtildemon}{\D(\tilde{G}/\tilde{N})_{\text{mon}}}
\newcommand{\DGNtildemonbar}{\D(\tilde{G}/\tilde{N})_{\overline{\text{mon}}}}
\newcommand{\DGNtilde}{\mathcal{D}(\tilde{G}/\tilde{N})}
\newcommand{\DofGNalphaaveragevanishing}{\D(G/N)^{\mathbb{G}_m^{\alpha}, \perp}}
\newcommand{\simpleforLmu}{\mathcal{E}_{\mu}}
\newcommand{\HNtilde}{\mathcal{D}(\tilde{N}\backslash \tilde{G}/\tilde{N})}
\newcommand{\HNtildeplus}{\mathcal{D}(\tilde{N}\backslash \tilde{G}/\tilde{N})^+}
\newcommand{\HNZtilde}{\mathcal{D}(\tilde{N}\backslash \tilde{G}/\tilde{N})^{Z \times Z}}
\newcommand{\HNZtildeplus}{\mathcal{D}(\tilde{N}\backslash \tilde{G}/\tilde{N})^{Z \times Z, +}}
In this section, we provide an alternate description of $\D(G/N)_{\text{deg}}$, see \cref{Simply Connected Version of Nondegenerate Subcategory of D(G/N) is the HN-subcategory Gen'd By Non-antidominant Simples}. Using the notation of \cref{Reduction to tildeG Subsubsection}, we are able to set the following notation. Let $\DGNtildemon$ denote those objects which are $\tilde{Q}_{\alpha}$-monodromic for some simple root $\alpha$, and let $\DGNtildemonbar$ denote the full right $\HNtilde$-subcategory of $\DGNtilde$ generated by $\DGNtildemon$. (This notation will not be used outside this section.) The main result of this section is the following theorem: 

\begin{Theorem}\label{Simply Connected Version of Nondegenerate Subcategory of D(G/N) is the HN-subcategory Gen'd By Non-antidominant Simples}
There is an equivalence of right $\HNtilde$-categories
\raggedbottom
\[\DGNtildemonbar \xhookrightarrow{\sim} \D(\tilde{G}/\tilde{N})_{\text{deg}}\]
\noindent induced by the inclusion $\DGNtildemon \xhookrightarrow{} \D(\tilde{G}/\tilde{N})_{\text{deg}}$ given by \cref{Avpsi Vanishes on Qalpha Monodromic Objects}.
\end{Theorem}

We prove this theorem below after proving some preliminary statements. 

\begin{Lemma}\label{Extension of Non-Integral Gm is Clean}
Fix a simple root $\alpha$, and let $j: \tilde{G}/\tilde{N} \xhookrightarrow{} V_s$ be the open embedding as in \cref{Kazhdan-Laumon SFT Construction Reminder}. Assume $\F \in \D(\tilde{G}/\tilde{N})$ is right $(\mathbb{G}_m^{\alpha}, \chi)$-monodromic for some non-integral $\chi \in (\text{Lie}(\mathbb{G}_m^{\alpha})/\mathbb{Z})(k)$. Then we have a canonical isomorphism $j_!(\F) \xrightarrow{} j_{*, dR}(\F)$ (and, in particular, $j_!(\F)$ is defined).
\end{Lemma}

\begin{proof}
We equip $V_s$ with a $\mathbb{G}_m$-action by scalar multiplication. This makes the embedding $\tilde{G}/\tilde{N} \xhookrightarrow{j} V_s$ $\mathbb{G}_m$-equivariant, where $\mathbb{G}_m$ acts by $\mathbb{G}_m^{\alpha}$ on $\tilde{G}/\tilde{N}$. In particular, the object $j_{*, dR}(\F)$ is $(\mathbb{G}_m^{\alpha}, \chi)$-monodromic. Similarly, let $z: \tilde{G}/Q_s \xhookrightarrow{} V_s$ denote the embedding of the complementary closed subscheme, which is the zero section of $V_s$. Then we see that any object of the form $z_{*, dR}(\mathcal{G})$ is $\mathbb{G}_m^{\alpha}$-monodromic. In particular, because $\chi$ is not integral, we see $\uHom(j_{*, dR}(\F), z_{*, dR}(\mathcal{G})) \simeq 0$. Therefore, if $\F' \in \D(V_s)$ we see that the composite
\raggedbottom
\[\uHom_{\D(V_s)}(j_{*, dR}(\F), \F') \xrightarrow{} \uHom_{\D(V_s)}(j_{*, dR}(\F), j_{*, dR}j^!(\F')) \simeq \uHom_{\D(\tilde{G}/\tilde{N})}(\F, j^!(\F'))\]

\noindent is an isomorphism, where the second equivalence follows from the fully faithfulness of $j_{*, dR}$ and the first arrow can be seen to be an isomorphism by applying $\uHom(j_{*, dR}(\F), -)$ to the cofiber sequence $i_{*, dR}i^!(\F') \to \F' \to j_{*, dR}j^!(\F')$.
\end{proof}

\begin{Lemma}\label{Symplectic Fourier Transform for a Root Swaps Simples Non-integral For That Root}
Let $\mu \in \LTd(k)$ be such that $\langle \alpha^{\vee}, \mu \rangle \notin \mathbb{Z}$. Then $F_{s_{\alpha}}(L_{\mu}) \cong L_{s_{\alpha}\cdot \mu}$.  
\end{Lemma}

\begin{proof}
Since $\D(G/N)^{B_{\lambda}} \xrightarrow{\sim} \D(\tilde{G}/\tilde{N})^{\tilde{B}_{\tilde{\lambda}}}$, we may assume that $G$ is semisimple and simply connected. Recall the block decomposition of \cite{LusYun}, $\D(B_{\lambda}\backslash G/N) \simeq \oplus_{\lambda' \in \text{orbit}_W(\lambda)} \D(B_{\lambda}\backslash G)^{B_{\lambda'}\text{-mon}}$.  We have that $L_{\mu}$ is $(T, \mathcal{L}_{[\mu]})$-monodromic, so, in particular, with respect to the $\mathbb{G}_m^{\alpha}$ action, we have that $L_{\mu}$ is $(\mathbb{G}_m^{\alpha}, \mathcal{L}_{[\overline{\mu}]})$-monodromic, where $\overline{\mu}$ is the image of $\mu$ under the projection map $\LTd \to \text{Lie}(\mathbb{G}_m^{\alpha})^{\ast}$. Note that $\langle \alpha^{\vee}, \overline{\mu}\rangle = \langle \alpha^{\vee}, \mu \rangle$ is not an integer. In particular, we see that $F_{s_{\alpha}, *}(L_{\mu}) \in \D(B_{\lambda}\backslash G/N)^{\heartsuit}$ since the canonical map $F_{s_{\alpha}, !}(L_{\mu}) \xrightarrow{} F_{s_{\alpha}, *}(L_{\mu})$ is an isomorphism by \cref{Extension of Non-Integral Gm is Clean} and the functor $F_{s_{\alpha}, !} \simeq j^!\mathbf{F}j_!$ is left $t$-exact and the latter functor $F_{s_{\alpha}, *} \simeq j^!\mathbf{F}j_{*, dR}$ is right $t$-exact.

Now, any object of $\D(B_{\lambda}\backslash G/N)^{\heartsuit}$ is a union of its compact subobjects by \cref{Colim in AB5 is Colim of Image}. Let $L_{\nu}$ be a simple subobject of $F_{s_{\alpha}}(L_{\mu})$. Then, since $F_{s_{\alpha}}$ preserves those sheaves for which the $\mathbb{G}_m^{\alpha}$-averagings vanish (because the pushforward of the sheaf $\omega_{\mathbb{G}_m^{\alpha}}$ to the torus is canonically $\langle s_{\alpha} \rangle$-equivariant), we see that $\nu$ has the property that $\langle \alpha^{\vee}, \nu \rangle \notin \mathbb{Z}$. In particular, we may apply an identical argument to see that $F_{s_{\alpha}}(L_{\nu}) \simeq H^0F_{s_{\alpha}}(L_{\nu})$ is a subobject of $F_{s_{\alpha}}F_{s_{\alpha}}(L_{\mu})$. However, the canonical map $F_{s_{\alpha}}F_{s_{\alpha}}(L_{\mu}) \to L_{\mu}$ is an isomorphism because the kernel is $\SL_{2, \alpha}$-monodromic; in other words, the fiber of the map $F_{s_{\alpha}}^2(\delta_1) \to \delta_1$ is the constant sheaf up to shift, and therefore the fiber of this map is given by $\SL_{2, \alpha}$-averaging up to shift, which vanishes on our $L_{\mu}$. Therefore, we see that $F_{s_{\alpha}}(L_{\nu})$ is a subobject of $L_{\mu}$, and therefore $F_{s_{\alpha}}(L_{\nu}) \simeq L_{\mu}$ since $F_{s_{\alpha}}^2(L_{\nu}) \simeq L_{\nu}$ is nonzero. 

We now show $\nu = s_{\alpha} \cdot \mu$. Write $\mu = w \cdot \lambda$ for some $w \in W$, which is unique because, by assumption, $\lambda$ is regular with respect to the $(W, \cdot)$-action. We first assume that $w$ is such that $s_{\alpha}w > w$ in the Bruhat ordering. Consider the locally closed embedding given by the union of Schubert cells of elements in the subset $\overline{\{w\}} \cup \overline{\{s_{\alpha}w\}}$. We note this subset contains an element of maximal length, namely $s_{\alpha}w$. Furthermore, by (1) of \cref{SwapsSupport} we have that the restriction to the Schubert cell labeled by $w$ vanishes. Therefore, the restriction to the Schubert cell labeled by $s_{\alpha}w$ does not vanish, since the symplectic Fourier transform is a transformation over $G/Q_{s_{\alpha}}$. Thus, $F_{s_{\alpha}}(L_{\mu})$ is a simple object in the heart whose restriction to the cell labeled by $s_{\alpha}w$ is nonzero and whose restriction to any other cell for $u$ with $\ell(u) \geq \ell(w)$ vanishes. Therefore, we see that this simple object is the simple $L_{s_{\alpha}w \cdot \lambda} = L_{s_{\alpha} \cdot \mu}$. 

If instead we had $s_{\alpha}w < w$ in the Bruhat ordering, then we may repeat the above argument swapping the roles of $\nu$ and $\mu$ to see that $F_{s_{\alpha}}(L_{s_{\alpha}\mu}) = L_{\mu}$, so the claim follows by applying $F_{s_{\alpha}}$ to this equality since $F_{s_{\alpha}}^2(L_{s_{\alpha}\mu}) \xleftarrow{\sim} L_{s_{\alpha}\mu}$. 
\end{proof}

\begin{Corollary}\label{The Gelfand-Graev Action Matches Any Non-Antidominant Weight With Some Q-monodromic Weight}
Fix some non-antidominant weight $\mu$. Then there exists some expression $s_1, ..., s_r$ such that, letting $F_{s_{i}}: \D(G/N)^{B_{\lambda}} \to \D(G/N)^{B_{\lambda}}$ denote the induced functor on $B_{\lambda}$-invariant categories, we have $F_{s_{1}}...F_{s_{r}}(L_{\mu}) \cong L_{\nu}$ for some $\nu \in \LTd(k)$ such that $\langle \alpha^{\vee}, \nu \rangle \in \mathbb{Z}^{\geq 0}$ for some simple root $\alpha$. In particular, $L_{\nu} \in \D(B_{\lambda}\backslash G)^{Q_{\alpha}\text{-mon}}$.
\end{Corollary}

\begin{proof}
Since $\D(G/N)^{B_{\lambda}} \xrightarrow{\sim} \D(\tilde{G}/\tilde{N})^{\tilde{B}_{\tilde{\lambda}}}$, we may assume that $G$ is semisimple and simply connected. For any given $\mu \in \LTd(k)$, let $\Xi_{\mu}^{> 0}$ denote the set of positive coroots $\gamma$ such that $\langle \mu + \rho, \gamma \rangle \in \mathbb{Z}^{> 0}$. Fix some non-antidominant weight $\mu$. By definition of antidominance, $\mu$ has the property that $\Xi_{\mu}^{> 0}$ is nonempty. We induct on the minimal height of an element in $\Xi_{\mu}^{> 0}$. For the base case, we note that if there is an element of $\Xi_{\mu}^{> 0}$ of height one, then it must be a simple coroot $\alpha^{\vee}$, since simple coroots are precisely the positive coroots of height one. Since $\langle \mu + \rho, \alpha^{\vee} \rangle = \langle \mu, \alpha^{\vee} \rangle + 1$, we deduce that $\langle \mu, \alpha^{\vee} \rangle \in \Z^{\geq 0}$. Therefore, taking the empty expression and $\mu = \nu$ gives our claim in this case.

Now assume $\gamma \in \Xi_{\mu}^{> 0}$ is some minimal height coroot of height larger than one. In particular, $\gamma$ is positive but not simple. Therefore, following the discussion in \cite[Section 0.2]{HumO}, there exists some simple coroot $\alpha^{\vee}$ such that $s_{\alpha}(\gamma)$ is a positive coroot whose height is less than the height of $\gamma$. Since \[\gamma - s_{\alpha}(\gamma) = \langle \alpha, \gamma \rangle \alpha^{\vee}\] we deduce that \begin{equation}\label{Height inequality}
    \langle \alpha, \gamma \rangle > 0
\end{equation} from the fact that the height of $s_{\alpha}(\gamma)$ is less than the height of $\gamma$. 

We now claim that for this coroot $\alpha^{\vee}$, $\langle \mu, \alpha^{\vee}\rangle$ is not an integer. To see this, we first compute that \[\langle \mu + \rho, s_{\alpha}(\gamma)\rangle = \langle \mu, s_{\alpha}(\gamma)\rangle + \langle \rho, s_{\alpha}(\gamma)\rangle = \langle s_{\alpha}(\mu), \gamma\rangle + \langle \rho, s_{\alpha}(\gamma)\rangle\] \[= \langle \mu - \langle \mu, \alpha^{\vee}\rangle\alpha, \gamma\rangle + \langle \rho, s_{\alpha}(\gamma)\rangle = \langle \mu, \gamma\rangle - \langle \mu, \alpha^{\vee}\rangle\langle \alpha, \gamma\rangle + \langle \rho, s_{\alpha}(\gamma)\rangle.\] We deduce that \begin{equation*}\langle \mu, \gamma\rangle - \langle \mu, \alpha^{\vee}\rangle\langle \alpha, \gamma\rangle + \langle \rho, s_{\alpha}(\gamma)\rangle \notin \Z^{> 0}\end{equation*} since $s_{\alpha}(\gamma)\notin \Xi^{>0}_{\mu}$ by minimality of $\gamma$. Writing $p := \langle \mu + \rho, \gamma\rangle$, we deduce that  \begin{equation*}p - \langle \rho, \gamma\rangle  - \langle \mu, \alpha^{\vee}\rangle\langle \alpha, \gamma\rangle + \langle \rho, s_{\alpha}(\gamma)\rangle \notin \Z^{> 0}\end{equation*} from the fact that $ p - \langle \rho, \gamma\rangle = \langle \mu, \gamma\rangle$. Therefore, \[-\langle\alpha, \gamma\rangle(1 + \langle \mu, \alpha^{\vee}\rangle) = -\langle \alpha, \gamma\rangle\langle \rho, \alpha^{\vee}\rangle - \langle \mu, \alpha^{\vee}\rangle\langle\alpha, \gamma\rangle = \langle \rho, -\langle \alpha, \gamma\rangle\alpha^{\vee}\rangle - \langle \mu, \alpha^{\vee}\rangle\langle\alpha, \gamma\rangle\] \[ = \langle \rho, s_{\alpha}(\gamma) - \gamma\rangle - \langle \mu, \alpha^{\vee}\rangle\langle\alpha, \gamma\rangle = - \langle \rho, \gamma\rangle  - \langle \mu, \alpha^{\vee}\rangle\langle \alpha, \gamma\rangle + \langle \rho, s_{\alpha}(\gamma)\rangle \notin \Z^{> -p}\] and so, in summary, \begin{equation}\label{Not in Zgp}
    -\langle\alpha, \gamma\rangle(1 + \langle \mu, \alpha^{\vee}\rangle) \notin \Z^{> -p}.
\end{equation}

Because $\langle\alpha, \gamma\rangle > 0$ by \labelcref{Height inequality}, from \labelcref{Not in Zgp} we see that $1 + \langle \mu, \alpha^{\vee}\rangle$ cannot be a negative integer or zero. Therefore, if $\langle \mu, \alpha^{\vee} \rangle$ were an integer at all, we would have that $1 + \langle \mu, \alpha^{\vee} \rangle > 0$, so $\langle \mu, \alpha^{\vee} \rangle > -1$, and so $\langle \mu, \alpha^{\vee} \rangle$ must therefore be a \textit{nonnegative} integer. But this would imply that \[\langle \mu + \rho, \alpha^{\vee} \rangle = \langle \mu, \alpha^{\vee} \rangle + 1 \in \Z^{>0}\] which would in turn imply that $\alpha^{\vee} \in \Xi^{>0}_{\mu}$, which would in turn violate our minimality assumption on $\gamma$.

We have just shown that $\langle \mu, \alpha^{\vee} \rangle \notin \Z$. From this fact, \cref{Symplectic Fourier Transform for a Root Swaps Simples Non-integral For That Root} gives that $F_{s_{\alpha}}(L_{\mu}) \simeq L_{s_{\alpha} \cdot \mu}$. Observe also that $s_{\alpha}(\gamma) \in \Xi_{s_{\alpha} \cdot \mu}^{> 0}$ since \[\langle s_{\alpha} \cdot \mu + \rho, s_{\alpha}(\gamma)\rangle = \langle s_{\alpha}(\mu + \rho), s_{\alpha}(\gamma)\rangle = \langle \mu + \rho, \gamma\rangle \in \Z^{> 0}\] since $\gamma \in \Xi^{>0}_{\mu}$. Therefore, by induction, since $s_{\alpha}(\gamma)$ has smaller height than $\gamma$, we may write $F_{s_1}...F_{s_{r-1}}(L_{s_{\alpha} \cdot \mu}) \cong L_{\nu}$. Setting $s_r = s_{\alpha}$, we obtain our desired isomorphism. 

The final sentence follows from \cite[Theorem 9.4]{HumO}. Specifically, this theorem states that our simple $L_{\nu} \in \text{(}\LG\text{-mod)}_{\chi_{\lambda}}^{N, \heartsuit}$ is in particular locally finite for the action of $\text{Lie}(Q_{\alpha})$, where we recall that $Q_{\alpha}$ is the commutator of the associated parabolic $P_{\alpha}$. We have seen that, as $G$ is simply connected, $Q_{\alpha}$ may be non-canonically written as a semidirect product of simply connected groups $U_{w_0s_{\alpha}} \rtimes \SL_2$. The fact that $L_{\nu}$ is $N$-equivariant implies that this representation may be lifted to the subgroup $U_{w_0s}$. Moreover, the local finiteness of the action for $\text{Lie}(Q_{\alpha})$ implies the local finiteness for the action of $\LSL_2$, and so we see that there exists a corresponding $\SL_2$-representation as well since $\SL_2$ is simply connected. Therefore we see that $L_{\nu} \in \text{(}\LG\text{-mod)}_{\chi_{\lambda}}^{Q_{\alpha}\text{-mon}, \heartsuit} \simeq \D(B_{\lambda}\backslash G)^{Q_{\alpha}\text{-mon}, \heartsuit}$ which establishes our claim. 
\end{proof}

\begin{proof}[Proof of \cref{Simply Connected Version of Nondegenerate Subcategory of D(G/N) is the HN-subcategory Gen'd By Non-antidominant Simples}]
We show the functor is essentially surjective. Let $\F \in \D(\tilde{G}/\tilde{N})^{\tilde{N}, +}_{\text{deg}}$. Since the essential image is closed under colimits, we may assume that $\F$ is eventually coconnective, and, since the inclusion functor is a map of left $\IndCoh(\LTd/\characterlatticeforT)$-categories, we may check that the essential image of the associated $\tilde{B}_{\lambda}$-invariants agree for all field-valued $\lambda$ by \cref{Can Check Equivalence on Each field-valued Point for Groups}. Using \cref{Categorical Extension of Scalars Proposition}, we may assume that $\lambda \in \LTd(k)$. Furthermore, any eventually coconnective object can be written as a colimit of its cohomology groups since the $t$-structure on $\D(\tilde{B}_{\lambda}\backslash \tilde{G}/\tilde{N})$ is right-complete by \cref{Conservative t Exact Functor to right-complete Implies right-complete and Dual Statement}. 

Note further that by \cref{Avpsi is Translation} the functor $\Avpsi$ has the same kernel as the contravariant functor $\tildeV_I$. An object $M \in \D(\tilde{B}_{\lambda}\backslash \tilde{G}/\tilde{N})^{\heartsuit}$ is in the kernel of $\tildeV_I$ if and only if it does not admit any subquotient of some antidominant simple. Therefore by \cref{Equivalent Notions of Degeneracy for Ind Derived Category O}, it suffices to prove that each $L_{\mu}$ for which $\mu$ lies in the $(W, \cdot)$-orbit of $\lambda$ and $\mu$ is not antidominant, $L_{\mu}$ lies in the right $\HN$-orbit of some object of $\D(\tilde{B}_{\lambda}\backslash \tilde{G})^{\tilde{Q}_{\alpha}\text{-mon}}$ for some simple root $\alpha$. However, this is precisely the content of \cref{The Gelfand-Graev Action Matches Any Non-Antidominant Weight With Some Q-monodromic Weight}, since the symplectic Fourier transformations are given by convolution with the Kazhdan-Laumon sheaves. 
\end{proof}

In particular, \cref{Simply Connected Version of Nondegenerate Subcategory of D(G/N) is the HN-subcategory Gen'd By Non-antidominant Simples} implies that, if $G$ is simply connected, a nondegenerate $G$-category in the sense of \cref{NondegenerateDefinition}  is equivalently nondegenerate in the sense of \cref{Nondegenerate G Category Definition for Simply Connected Group}. 

\begin{Remark}
If $G = \text{PGL}_2$, $\D(G/N)_{\text{deg}}$ is not generated by right $G$-monodromic objects. For example, an immediate consequence of \cref{Avpsi is Translation} is the fact that the two dimensional representation $k^2$ of $\mathfrak{pgl}_2$ is an eventually coconnective object in the kernel of $\Avpsi$; however, it is not $G$-monodromic since, for example, if we let $\mu$ denote the central character of the two dimensional representation, then we have that $k^2 \in \D(B_{\mu}\backslash G/N)^{\heartsuit}$ and
\raggedbottom
\[\D(B_{\lambda}\backslash G/N)^{G\text{-mon}} \simeq \LG\text{-mod}^G_{\chi_{\mu}} \simeq \text{Rep}(\text{PGL}_2)_{\chi_{\mu}}\]

\noindent is zero.
\end{Remark}

    \appendix

\section{Twisted D-Modules are the Derived Category of their Heart}\label{Twisted D-Modules are Derived Category of Heart Section}
Fix some $\mu \in \LTd(k)$, and let $\lambda \in \LTd(k)$ denote some regular antidominant weight in $\mu + \characterlatticeforT$. In this appendix, we prove that \[\D(G/_{\mu} B)^N \simeq \D(G/_{\lambda}B)^N\] is the unbounded derived category of its heart. Here, by derived category, we mean the canonical DG structure on the derived category of its heart in the sense of \cite[Definition 1.3.5.8]{LuHA}. To define the (unbounded) derived category of an abelian category $\mathcal{A}$ in this sense, we must first argue that $\mathcal{A}$ is a Grothendieck abelian category.

\begin{Lemma}\label{Derived Category is Defined for DG Categories with t-Structure Compatible with Filtered Colimits Remark}
The category $\D(G/_{\lambda}B)^{N, \heartsuit}$ is a Grothendieck abelian category. 
\end{Lemma}

\begin{proof}
Note that the $t$-structure on $\D(G)$ is compatible with filtered colimits \cite[Section 4.3.2]{GaiRozCrystals}. In particular, the connective (respectively, coconnective) objects can be identified with the ind-completion of the subcategory of connective (respectively, coconnective) objects which are compact \cite[Chapter 4, Lemma 1.2.4]{GaRoI}, and so we also see that the $t$-structure on $\D(G)$ is accessible. Therefore, since the forgetful functor $\text{oblv}: \D(G/_{\lambda}B)^{N} \to \D(G)$ is $t$-exact and conservative, we also see that the $t$-structure on $\D(G/_{\lambda}B)^{N, \heartsuit}$ is compatible with filtered colimits and accessible. Therefore $\D(G/_{\lambda}B)^{N, \heartsuit}$ is a Grothendieck abelian category by \cite[Remark 1.3.5.23]{LuHA}. 
\end{proof}

\begin{Proposition}\label{N Invariant Twisted D Modules on Flag Variety Is Derived Category of Heart Proposition}
The category $\D(G/_{\lambda}B)^N$ is the unbounded derived category of its heart.
\end{Proposition}

This entire appendix will be devoted to the proof of \cref{N Invariant Twisted D Modules on Flag Variety Is Derived Category of Heart Proposition}. The strategy will proceed as follows: in \cref{Bounded By Below Twisted D-Modules are Bounded Derived Category of Heart}, we show that the bounded-below derived categories agree. Then, in \cref{Left-Completeness of Derived Category} and \cref{Identification of Derived Category O with N-Invariants of Flag Variety}, we show that both categories of \cref{N Invariant Twisted D Modules on Flag Variety Is Derived Category of Heart Proposition} are left-complete, thus proving \cref{N Invariant Twisted D Modules on Flag Variety Is Derived Category of Heart Proposition}. 

\subsection{Bounded-Below Twisted D-Modules are the Bounded Derived Category of the Heart}\label{Bounded By Below Twisted D-Modules are Bounded Derived Category of Heart}
In \cref{Left Bounded N Invariant Twisted D Modules on Flag Variety Is Left Bounded Derived Category of Heart Proposition}, we will show that the bounded below variant of \cref{N Invariant Twisted D Modules on Flag Variety Is Derived Category of Heart Proposition} holds. We will do this via the following general claim: 

\begin{Lemma}\cite[Proposition 1.3.3.7, Dual Version]{LuHA}\label{Category is Left-Bounded Derived Category of Heart Iff Injectives Have Correct Mapping Property}
Assume we are given a stable $\infty$-category $\C$ equipped with a right-complete $t$-structure for which $\C^{\heartsuit}$ has enough injectives. Then we have an induced $t$-exact functor $\D^+(\C^{\heartsuit}) \to \C^+$, and this functor is an equivalence if and only if for any object $I \in \C^{\heartsuit}$ which is injective at the abelian categorical level (i.e. the functor $\text{Hom}_{\C^{\heartsuit}}(-, I)$ is an exact functor of abelian categories), and $Y \in \C^{\heartsuit}$, there exists some $Z \in \C^{\heartsuit}$ and an epimorphism $Z \twoheadrightarrow Y$ such that $\text{Ext}^i_{\C}(Z, I) \cong 0$ for all $i > 0$.    
\end{Lemma}

We first show this Ext vanishing for objects of $\mathcal{O}_{\lambda}$:

\begin{Lemma}\label{The Higher Ext of Compact Objects and Compact Injective Objects Vanishes}
Assume $\mathcal{F} \in \D(G/_{\lambda}B)^{N, c, \heartsuit}$ and $J \in \D(G/_{\lambda}B)^{N, c, \heartsuit}$ is injective. Then $\uHom_{\D(G/_{\lambda}B)^N}(\F, J)$ is concentrated in degree 0. 
\end{Lemma}

The proof below closely follows ideas from \cite{BeilinsonGinzburgSoergelKoszulDualityPatternsInRepresentationTheory}. 

\begin{proof}[Proof of \cref{The Higher Ext of Compact Objects and Compact Injective Objects Vanishes}]
The category $\D(G/_{\lambda}B)^{N, c, \heartsuit}$ is of finite length, and each object admits a filtration by simple objects of $\D(G/_{\lambda}B)^{N, c, \heartsuit}$. Each of these simples is given by the intermediate extension of the generator of $\D(N\backslash NwB)^{B_{\lambda}, \heartsuit}$. We will denote this intermediate extension by $L_w$. We show, by induction on the ordering in $W$, that for a fixed $w \in W$ any object which admits a filtration by simples $L_v$ for $v \leq w$ has this Ext vanishing property. We recall that our injective object $J$ admits a filtration by costandard objects $\nabla_w$, i.e. those objects which are the $(*, dR)$-pushforward of a generator of $\D(N\backslash NwB)^{B_{\lambda}}$ (in $\mathcal{O}_{\lambda}$, for example, this is the Verdier dual version of \cite[Theorem 3.10]{HumO}). 

For the base case, when $\ell(w) = 0$ (and therefore $w = 1$), we let $M$ be any object which admits a filtration whose subquotients are all of the form $L_1$. Then, since $\Delta_1 \xrightarrow{\sim} L_1$, any such $M$ is the direct sum of copies of $\Delta_1$. Therefore our desired Ext vanishing follows from the fact that there are no extensions between any $\Delta_1$ and objects which admit filtrations whose subquotients are $\nabla_w$ (which, in turn, follows by induction on the length of a filtration and the fact that $i_1^!(\nabla_w) \simeq 0$ if $w \neq 1$). Finally, for a general $w \in W$, we have a short exact sequence 
\raggedbottom
\[0 \to S \to \Delta_w \to L_w \to 0\]

\noindent where $S$ admits a filtration of objects $L_v$ for $v < w$. We then obtain an exact sequence
\raggedbottom
\[\text{Ext}^{j - 1}(\Delta_w, J) \to \text{Ext}^{j - 1}(S, J) \to \text{Ext}^j(L_w, J) \to \text{Ext}^j(\Delta_w, J)\]

\noindent for any $j > 0$. Note that the last term in this sequence vanishes because, in $\D(G/_{\lambda}B)^N$, there are no extensions between any $\Delta_w$ and objects which admit filtrations whose subquotients are $\nabla_w$. If $j = 1$, then, we may identify the third term as the cokernel of the map $\text{Hom}(\Delta_w, J) \to \text{Hom}^{j - 1}(S, J)$, where the maps are equivalently taken in $\D(G/_{\lambda}B)^{N, \heartsuit}$, a full subcategory of $\D(G/_{\lambda}B)^N$. In particular, the injectivity of $J$ shows that this map is a surjection so the cokernel vanishes. If $j > 1$, we similarly obtain that $\text{Ext}^{j - 1}(\Delta_w, J)$ vanishes and so that our inductive hypothesis shows that $\text{Ext}^j(L_w, J) \xleftarrow{\sim} \text{Ext}^{j - 1}(S, J)$ vanishes, since $S$ admits a filtration by simple objects $L_v$ for $v < w$. 
\end{proof}

We now use the following general lemma to reduce to the above computations.

\begin{Lemma}\label{Every Injective of Ind-Completion Has Product Which is Product of Indecomposable Injectives}
Let $\mathcal{A}'$ be an abelian category closed under subquotients in its ind-completion $\mathcal{A}$ with enough injectives such that every injective object of $\mathcal{A}'$ is a finite direct product of indecomposable injective objects. Then, for every injective object $I \in \mathcal{A}$, there exists some injective $M \in \mathcal{A}$ such that $I \times M \cong \prod_n I_n$ for $I_n \in \mathcal{A}'$ indecomposable injectives. 
\end{Lemma}

\begin{proof}
Assume $I \in \mathcal{A}$ is some injective object. By assumption and \cref{Colim in AB5 is Colim of Image}, we can write $I$ as an increasing union of objects of $\mathcal{A}'$, say $I = \cup_n A_n$. By assumption that $\mathcal{A}'$ has enough injectives, we may choose an injection $A_n \xhookrightarrow{} I_n$ into an injective object $I_n \in \mathcal{A}'$. Note that $\mathcal{A}$ is a Grothendieck abelian category, so $I_n$ is an injective object of $\mathcal{A}$ and so the inclusion map $i_n: A_n \subseteq I$ extends to a map $f_n: I \to I_n$. 

Consider the canonical induced map $f: I \to \prod_n I_n$ with projection maps $pr_m: \prod_n I_n \to I_m$. The map $f$ is injective, since if $K$ denotes the kernel of $f$, then $\text{pr}_nf|_{A_n \cap K} = i_n|_{A_n \cap K}$ so $A_n \cap K = 0$ for all $n$, and thus $K = \cup_n (A_n \cap K) = 0$. Write each $I_n$ as a product of indecomposable injectives and relabel, if necessary, so that each $I_n$ is an indecomposable injective object; we may do this since every injective object of $\mathcal{A}'$ is a finite direct product of indecomposable injective objects. However, $\prod_n I_n$ is injective since the property of being an injective object is closed under products, and so the injective map $I \xhookrightarrow{f} \prod_n I_n$ splits, using:

\begin{equation*}
  \xymatrix@R+2em@C+2em{
 I \ar[r]^{f} \ar[d]^{\text{id}}  & \prod_n I_n \ar[dl]^{\exists} \\ I
  }
 \end{equation*}
 
 \noindent and so we can write $\prod_n I_n \cong I \times M$ for some $M \in \mathcal{A}$. The fact that $\prod_n I_n$ and $I$ are both injective implies that $M$ is, which follows from the fact that a product in an abelian category is injective only if each factor is. 
\end{proof}

\begin{Corollary}\label{Left Bounded N Invariant Twisted D Modules on Flag Variety Is Left Bounded Derived Category of Heart Proposition}
The induced functor from the left-bounded derived category of the heart of $\D(G/_{\lambda}B)^N$ to the category $\D(G/_{\lambda}B)^{N, +}$ is an equivalence. 
\end{Corollary}
\begin{proof}
By \cref{Category is Left-Bounded Derived Category of Heart Iff Injectives Have Correct Mapping Property}, it suffices to show that $\uHom_{\D(G/_{\lambda}B)^N}(\mathcal{F}, I)$ is concentrated in degree 0 if $\F, I \in \D(G/_{\lambda}B)^{N, \heartsuit}$ and $I$ is injective. However, the category $\D(G/_{\lambda}B)^{N, \heartsuit}$ satisfies the hypotheses of \cref{Every Injective of Ind-Completion Has Product Which is Product of Indecomposable Injectives}, so we can find some injective $M$ of $\D(G/_{\lambda}B)^{N, \heartsuit}$ for which 
\raggedbottom
\[\uHom(\F, I) \times \uHom(\F, M) \simeq \uHom(\F, I \times M) \simeq \prod \uHom(\F, I_n)\]

\noindent for $I_n$ compact injective objects of $\D(G/_{\lambda}B)^{N, \heartsuit}$. Therefore, it suffices to show that $I$ is a compact injective object. Furthermore, by compact generation of $\D(G/_{\lambda}B)^N$, we may write $\F$ as a filtered colimit of compact objects. Since the $t$-structure on Vect is compatible with filtered limits, to show that $\uHom(\F,  I)$ is concentrated in degree 0, it suffices to assume that $\F$ is compact. The claim then follows from \cref{The Higher Ext of Compact Objects and Compact Injective Objects Vanishes}. 
\end{proof}

\subsection{Left-Completeness of Derived Category}\label{Left-Completeness of Derived Category}
Let $\mathcal{O}_{\lambda}:= \D(N\backslash G/_{\lambda}B)^{\heartsuit}$. Note that the category $\mathcal{D}^{b}(\mathcal{O}_{\lambda})$ is canonically equipped with a $t$-structure, so that $\text{Ind}(\mathcal{D}^{b}(\mathcal{O}_{\lambda}))$ is equipped with a canonical $t$-structure such that the inclusion functor $\mathcal{D}^{b}(\mathcal{O}_{\lambda}) \xhookrightarrow{} \text{Ind}(\mathcal{D}^{b}(\mathcal{O}_{\lambda}))$ is $t$-exact and the heart of the $t$-structure is $\text{Ind}(\mathcal{O}_{\lambda})$, see e.g. \cite[Chapter 4, Lemma 1.2.4]{GaRoI}. 

\begin{Remark}\label{Ind Completion of Abelian Category in Higher Categorical Setting Remains Abelian}
We note in passing that the ind-completion of any ordinary (1,1)-category is always an ordinary (1,1)-category. This is because, by definition, the ind-completion of an ordinary category $C$ is defined as the full subcategory of presheaves on $C$, i.e. $\text{Fun}(C^{op}, \text{Spc})$ containing the representable functors and closed under filtered colimits. In particular, since all representable objects map to discrete spaces (i.e. sets) and discrete spaces are closed under filtered colimits, we see that the classical definition of the ind-completion of $C$ agrees with the higher categorical definition as in \cite[Section 5.3]{LuHTT}. 
\end{Remark}

We now claim the following: 

\begin{Lemma}\label{Ind of Bounded Derived is Bounded Derived of Ind for Abelian Cats with Finite Cohomological Dimension and Proj Generators}
If $\mathcal{A}$ is some small abelian category of finite cohomological dimension with a compact projective generator, then we have an equivalence of DG categories $\text{Ind}(\mathcal{D}^{b}(\mathcal{A})) \simeq \D(\text{Ind}(\mathcal{A}))$, where the right-hand category denotes the canonical DG model on the derived $\infty$-category of the Grothendieck abelian category $\text{Ind}(\mathcal{A})$ given by the ind-completion of $\mathcal{A}$. 
\end{Lemma}

\begin{proof}[Proof of \cref{Ind of Bounded Derived is Bounded Derived of Ind for Abelian Cats with Finite Cohomological Dimension and Proj Generators}]
Let $\mathcal{G} \in \mathcal{A}$ denote the compact projective generator. Then every object $M \in \mathcal{A}$ admits a finite length projective resolution, so the smallest subcategory containing $\mathcal{G}$ and all finite colimits is $\mathcal{D}^b(\mathcal{A})$. Since, by definition, every object of $\text{Ind}(\mathcal{D}^b(\mathcal{A}))$ is a filtered colimit of objects of $\mathcal{D}^b(\mathcal{A})$, we see that the smallest subcategory containing $\mathcal{G}$ and all colimits is $\text{Ind}(\mathcal{D}^b(\mathcal{A}))$ itself. Therefore, the functor 
\raggedbottom
\[\uHom_{\text{Ind}(\mathcal{D}^b(\mathcal{A}))}(\mathcal{G}, -): \text{Ind}(\mathcal{D}^b(\mathcal{A})) \to E\text{-mod}\] 

\noindent where $E := \uEnd_{\text{Ind}(\mathcal{D}^b(\mathcal{A}))}(\mathcal{G})$, gives an equivalence of categories by Barr-Beck. Furthermore, the fact that $\mathcal{G}$ is a projective object implies that $E$ is a discrete (i.e. classical) algebra and that this equivalence is $t$-exact. Thus $\text{Ind}(\mathcal{D}^b(\mathcal{A}))$ is equivalent to a category which is the unbounded derived category of its heart by a $t$-exact functor. We therefore see that $\text{Ind}(\mathcal{D}^b(\mathcal{A}))$ is itself the unbounded derived category of its heart.
\end{proof}

By \cite[Theorem B.1.3]{AGKRRVStackOfLocalSystemsWithRestrictedVariation}, we therefore see:

\begin{Corollary}\label{The unbounded derived category of Olambda of the ind-completion is left-complete}
The unbounded derived category of the ind-completion of $\mathcal{O}_{\lambda}$ is left-complete in its $t$-structure.
\end{Corollary}

\subsection{Identification of Derived Categories}\label{Identification of Derived Category O with N-Invariants of Flag Variety}
Recall in \cref{The unbounded derived category of Olambda of the ind-completion is left-complete} that we showed the unbounded derived category of $\D(N\backslash G/_{\lambda}B)^{\heartsuit}$ is left-complete. Therefore, as we have identified the bounded-below subcategories in \cref{Left Bounded N Invariant Twisted D Modules on Flag Variety Is Left Bounded Derived Category of Heart Proposition}, to identify the unbounded derived categories it suffices to show the following claim:

\begin{Proposition}\label{t-Structure on Twisted D-Modules is Left-Complete}
The $t$-structure on $\D(G/_{\lambda}B)^N$ is left-complete.
\end{Proposition}

We will prove \cref{t-Structure on Twisted D-Modules is Left-Complete} after showing the following general lemma: 

\begin{Lemma}\label{Conservative Right Adjoint of Bounded Cohomological Amplitude to Left-Complete Category Implies Source Is}
Let $R: \mathcal{C} \to \mathcal{D}$ denote a conservative functor between categories with $t$-structures such that $R$ commutes with limits (i.e. is a right adjoint) and has bounded cohomological amplitude. Then if the $t$-structure on $\mathcal{D}$ is left-complete, then the $t$-structure on $\mathcal{C}$ is also left-complete. 
\end{Lemma}

\begin{proof}
Let $\F \in \C$. We wish to show the canonical map $\F \to \text{lim}_m\tau^{\geq m}(\F)$ is an equivalence, where $m$ varies over the nonpositive integers. By conservativity of $R$ and the fact that $R$ commutes with limits, we see that we may equivalently show the canonical map $R(\F) \to \text{lim}_mR(\tau^{\geq m}\F)$ is an equivalence. However, note that 
\raggedbottom
\[\text{lim}_mR(\tau^{\geq m}\F) \simeq \text{lim}_m \text{lim}_n \tau^{\geq n}R(\tau^{\geq m}\F) \simeq \text{lim}_n \text{lim}_m \tau^{\geq n}R(\tau^{\geq m}\F)\]

\noindent where the first expression uses the left-completeness of the $t$-structure on $\D$ and the second uses the fact that limits commute with limits. Now let $[p,q]$ denote the cohomological amplitude of $R$ for $p, q \in \mathbb{Z}$. We see that for any fixed $n$ we have that if $m < n + p$ then
\raggedbottom
\[\tau^{\geq n}R(\tau^{\geq m}\F) \simeq \tau^{\geq n}R(\tau^{\geq n + p}\F) \simeq \tau^{\geq n}R(\F)\]

\noindent where both steps follow from the definition of bounded cohomological amplitude. Therefore, continuing the above chain of equivalences, we see
\raggedbottom
\[\text{lim}_mR(\tau^{\geq m}\F) \simeq \text{lim}_n\tau^{\geq n}R(\F) \xleftarrow{\sim} R(\F)\]

\noindent where the last step follows by the left-completeness of $\D$, as desired. 
\end{proof}

\begin{proof}[Proof of \cref{t-Structure on Twisted D-Modules is Left-Complete}]
By \cref{Conservative Right Adjoint of Bounded Cohomological Amplitude to Left-Complete Category Implies Source Is}, we may exhibit a finite set of compact generators of finite cohomological amplitude, since then we may take the direct sum $\mathcal{G}$ and set $R := \uHom_{\D(G/_{\lambda}B)^N}(\mathcal{G}, -)$. However, we take our compact generators as the standard objects $\Delta_{w} \in \D(G/_{\lambda}B)^{N, \heartsuit}$. Then maps from each standard object are given by a $!$-restriction by a locally closed map of smooth schemes, which in particular has finite cohomological amplitude \cite[Proposition 1.5.13, Proposition 1.5.14]{HTT}. 
\end{proof}
\printbibliography
    \end{document}